\documentclass[leqno]{article} %12pt

\usepackage[frenchb,english]{babel}
\usepackage[T1]{fontenc}

\usepackage{amsmath}
\usepackage{amssymb}
\usepackage{amsfonts}
\usepackage{enumerate}
\usepackage{vmargin}
\usepackage[all]{xy}
\usepackage{mathrsfs}
\usepackage{mathtools}
\usepackage{lmodern}
\usepackage[colorlinks=true,linkcolor=blue,pagebackref=true]{hyperref}%colorlinks=true,urlcolor=blue,
\usepackage{ulem}
\usepackage[all]{xy}
\usepackage{comment}
\usepackage{centernot}
\setmarginsrb{3cm}{3cm}{3.5cm}{3cm}{0cm}{0cm}{1.5cm}{3cm}
\newcommand{\B}{\mathrm{B}}
\newcommand{\C}{\mathrm{C}}

\newcommand{\W}{\mathrm{W}}
\let\L\relax
\newcommand{\L}{\mathrm{L}}
\newcommand{\SL}{\mathrm{SL}}
\newcommand{\M}{\mathrm{M}}

\newcommand{\N}{\ensuremath{\mathbb{N}}}

\newcommand{\R}{\ensuremath{\mathbb{R}}}

\renewcommand{\leq}{\ensuremath{\leqslant}}
\renewcommand{\geq}{\ensuremath{\geqslant}}
\newcommand{\qed}{\hfill \vrule height6pt  width6pt depth0pt}

\newcommand{\bnorm}[1]{ \big\| #1  \big\|}

\newcommand{\norm}[1]{\left\Vert#1\right\Vert}
\newcommand{\xra}{\xrightarrow}

\newcommand{\ot}{\otimes}
\newcommand{\epsi}{\varepsilon}
\newcommand{\ovl}{\overline}
\newcommand{\otvn}{\ovl\ot}

\newcommand{\la}{\langle}
\newcommand{\ra}{\rangle}
\newcommand{\co}{\colon}
\renewcommand{\d}{\mathop{}\mathopen{}\mathrm{d}} %opérateur différentiel
\let\i\relax %\i ne fait plus rien !
\newcommand{\i}{\mathrm{i}}
\newcommand{\w}{\mathrm{w}}
\newcommand{\ov}{\overset}
\let\cal\relax
\newcommand{\cal}{\mathcal}
\newcommand{\dec}{\mathrm{dec}}
\newcommand{\Dec}{\mathrm{Dec}}

\newcommand{\Ad}{\mathrm{Ad}}

\newcommand{\QWEP}{\mathrm{QWEP}}

\newcommand{\dist}{\mathrm{dist}}
\newcommand{\Id}{\mathrm{Id}}
\newcommand{\VN}{\mathrm{VN}}
\newcommand{\CB}{\mathrm{CB}}

\newcommand{\IF}{\mathrm{(IF)}}

\DeclareMathOperator{\ind}{ind}
\DeclareMathOperator{\tr}{Tr}

\DeclareMathOperator{\supp}{supp} %support
\DeclareMathOperator{\card}{card} %cardinal
\DeclareMathOperator{\Span}{span} %span
\newcommand{\cb}{\mathrm{cb}} %completely bounded
\newcommand{\HS}{\mathrm{HS}} %Herz-Schur
\newcommand{\CC}{\mathrm{CC}}
\newcommand{\inner}{\mathrm{int}}
\renewcommand{\subseteq}{\subset}
\newtheorem{thm}{Theorem}[section]
\newtheorem{defi}[thm]{Definition}

\newtheorem{prop}[thm]{Proposition}
\newtheorem{conj}[thm]{Conjecture}

\newtheorem{cor}[thm]{Corollary}
\newtheorem{lemma}[thm]{Lemma}

\newtheorem{example}[thm]{Example}
\newtheorem{remark}[thm]{Remark}

\newenvironment{proof}[1][]{\noindent {\it Proof #1} : }{\hbox{~}\qed
\smallskip
}

\usepackage{tocloft}
\numberwithin{equation}{section}
\usepackage[nottoc,notlot,notlof]{tocbibind}%biblio dans sommaire
\let\OLDthebibliography\thebibliography
\renewcommand\thebibliography[1]{
  \OLDthebibliography{#1}
  \setlength{\parskip}{0pt}
  \setlength{\itemsep}{0pt plus 0.3ex}
}

\begin{document}
\selectlanguage{english}
\title{\bfseries{Decomposable Fourier Multipliers and an Operator-Algebraic Characterization of Amenability}}
%
%Projections, multipliers and decomposable maps on noncommutative $\L^p$-spaces II
\date{}

\author{\bfseries{C\'edric Arhancet - Christoph Kriegler}}

\maketitle

\begin{abstract}
We study the algebra $\mathfrak{M}^{\infty,\mathrm{dec}}(G)$ of decomposable Fourier multipliers on the group von Neumann algebra $\mathrm{VN}(G)$ of a locally compact group $G$, and its relation to the Fourier–Stieltjes algebra $\mathrm{B}(G)$. For discrete groups, we prove that these two algebras coincide isometrically. In contrast, we show that the identity $\mathfrak{M}^{\infty,\mathrm{dec}}(G) = \mathrm{B}(G)$ fails for various classes of non-discrete groups, and that, among second-countable unimodular groups, inner amenability ensures the equality. Our approach relies on the existence of contractive projections preserving complete positivity from the space of completely bounded weak* continuous operators on $\mathrm{VN}(G)$ onto the subspace of completely bounded Fourier multipliers. We show that such projections exist in the inner amenable case. As an application, we obtain a new operator-algebraic characterization of amenability. We also investigate the analogous problem for the space of completely bounded Fourier multipliers on the noncommutative $\mathrm{L}^p$-spaces $\mathrm{L}^p(\mathrm{VN}(G))$, for $1 \leq p \leq \infty$. Using Lie group theory and results stemming from the solution to Hilbert's fifth problem, we prove that second-countable unimodular finite-dimensional amenable locally compact groups admit compatible projections at \( p = 1 \) and \( p = \infty \). These results reveal new structural links between harmonic analysis, operator algebras, and the geometry of locally compact groups.
\end{abstract}

\makeatletter
 \renewcommand{\@makefntext}[1]{#1}
 \makeatother
 \footnotetext{\noindent
% The second author is supported by the research program .\\
 2020 {\it Mathematics subject classification:}
 46L51, 43A15, 46L07, 43A07. 
% 46L51 Noncommutative measure and integration
% 46M35 Abstract interpolation of topological vector spaces [See also 46B70]
% 46L07 Operator spaces and completely bounded maps [See also 47L25]
% 43A22 Homomorphisms and multipliers of function spaces on groups, semigroups, etc.
% 43A15 Lp-spaces and other function spaces on groups, semigroups, etc.
% 47D03 Groups and semigroups of linear operators For nonlinear operators.
% 43A15 $L^p$-spaces and other function spaces on groups, semigroups, etc.
% 43A07 Means on groups, semigroups, etc.; amenable groups
}
%noncommutative $\L^p$-spaces, operator spaces

\tableofcontents

%%%%%%%%%%%%%%%%%%%%%%%%%%%%%%%%%%%%%%%%%%%%%%%%%%%%%%%%%%%%%%%%%%%%%%%%%%%%%%%%%%%%%%%%%%%%%%%%%%%%%%%%%%%%%%%%%%%%%%%%%%%%%%%%%%%%%%%%%%%%%%%%%%%%%%
%%%%%%%%%%%%%%%%%%%%%%%%%%%%%%%%%%%%%%%%%%%%%%%%%%%%%%%%%%%%%%%%%%%%%%%%%%%%%%%%%%%%%%%%%%%%%%%%%%%%%%%%%%%%%%%%%%%%%%%%%%%%%%%%%%%%%%%%%%%%%%%%%%%%%%
%%%%%%%%%%%%%%%%%%%%%%%%%%%%%%%%%%%%%%%%%%%%%%%%%%%%%%%%%%%%%%%%%%%%%%%%%%%%%%%%%%%%%%%%%%%%%%%%%%%%%%%%%%%%%%%%%%%%%%%%%%%%%%%%%%%%%%%%%%%%%%%%%%%%%%
\section{Introduction}
\label{sec:Introduction}
%%%%%%%%%%%%%%%%%%%%%%%%%%%%%%%%%%%%%%%%%%%%%%%%%%%%%%%%%%%%%%%%%%%%%%%%%%%%%%%%%%%%%%%%%%%%%%%%%%%%%%%%%%%%%%%%%%%%%%%%%%%%%%%%%%%%%%%%%%%%%%%%%%%%%%
%%%%%%%%%%%%%%%%%%%%%%%%%%%%%%%%%%%%%%%%%%%%%%%%%%%%%%%%%%%%%%%%%%%%%%%%%%%%%%%%%%%%%%%%%%%%%%%%%%%%%%%%%%%%%%%%%%%%%%%%%%%%%%%%%%%%%%%%%%%%%%%%%%%%%%
%%%%%%%%%%%%%%%%%%%%%%%%%%%%%%%%%%%%%%%%%%%%%%%%%%%%%%%%%%%%%%%%%%%%%%%%%%%%%%%%%%%

\subsection{Context and motivation} The theory of Fourier–Stieltjes algebras and Fourier multipliers has played a central role in abstract harmonic analysis since the foundational works of Eymard \cite{Eym64}, Haagerup \cite{Haa79} and de Canniere--Haagerup \cite{DCH85}. The theory of Fourier multipliers on general (possibly non-abelian) groups  has seen significant developments \cite{CGPT23}, \cite{CoH89}, \cite{HaL13}, \cite{JuR03}, \cite{JMP14}, \cite{JMP18}, \cite{LaS11}, \cite{MeR17}, \cite{PRS22}, \cite{PST25} particularly in relation to weak amenability, related approximation properties of von Neumann algebras, noncommutative $\L^p$-spaces and Schur multipliers.

In the setting of group von Neumann algebras, the space of completely bounded Fourier multipliers provides a natural extension of the classical Fourier--Stieltjes algebra. The aim of this paper is to explore the structural relation between the Fourier-Stieltjes algebra $\B(G)$ and the algebra $\frak{M}^{\infty,\dec}(G)$ of decomposable Fourier multipliers on the group von Neumann algebra $\VN(G)$ of a locally compact group $G$, and to use this relation to characterize various forms of amenability in locally compact groups. Our results offer a new analytic perspective on inner amenability and decomposability, and provide an operator-algebraic route to understanding the harmonic analytic properties of $G$.

We prove that, for discrete groups, these two algebras coincide isometrically. In contrast, we show that the identity $\frak{M}^{\infty,\dec}(G) = \B(G)$ fails in general for non-discrete groups. One of our main results establishes that the equality is a consequence of inner amenability among second-countable unimodular locally compact groups. Our approach involves a detailed study of bounded projections from the space of completely bounded operators on the von Neumann algebra $\VN(G)$ preserving complete positivity onto the space of completely bounded Fourier multipliers. We show that the existence of such projections is intimately related to structural properties of the group, such as (inner) amenability. This leads to a new analytic characterization of amenability, formulated in operator-algebraic terms.

We further investigate the setting of noncommutative $\L^p$-spaces, using tools from geometric group theory and the structure theory of locally compact groups. In particular, we also study the existence of bounded projections from the space of completely bounded operators on the noncommutative $\L^p$-space $\L^p(\VN(G))$ onto the space of completely bounded $\L^p$-Fourier multipliers.

Recall that the Fourier-Stieltjes algebra $\B(G)$ of a locally compact group $G$ is a generalization of the algebra of bounded regular complex Borel measures of an abelian locally compact group to non-abelian groups. Since its introduction by Eymard in \cite{Eym64}, this commutative unital Banach algebra has become a central object in noncommutative harmonic analysis and is closely related to the unitary representation theory of $G$. More precisely, the elements of $\B(G)$ are exactly the matrix coefficients of continuous unitary representations of $G$ on complex Hilbert spaces, i.e. 
\begin{equation}
\label{BG-as-entries}
\B(G) 
\ov{\mathrm{def}}{=} \big\{ \langle \pi(\cdot)\xi,\eta \rangle_H : \pi \text{ is a unitary representation of } G  \text{ on $H$ and } \xi,\eta \in  H \big\}.
\end{equation}
The norm is defined by
\begin{equation}
\label{Norm-BG}
\norm{\varphi}_{\B(G)}
\ov{\mathrm{def}}{=}  \inf_{} \norm{\xi}\norm{\eta},
\end{equation}
where the infimum is taken over all $\pi, \xi,\eta$ such that $\varphi=\langle\pi(\cdot)\xi,\eta\rangle_H $. The operations of this algebra are pointwise multiplication and addition. Also note that $\B(G)$ is a complete invariant of $G$, i.e.~$\B(G_1)$ and $\B(G_2)$ are isometrically isomorphic as Banach algebras if and only if $G_1$ and $G_2$ are topologically isomorphic as locally compact groups as proved by Walter in \cite{Wal74} (see also \cite{Wal70} and \cite[Theorem 3.2.5 p.~99]{KaL18}).

Decomposable maps is a class of operators between $\mathrm{C}^*$-algebras generalizing completely positive maps. The class of decomposable maps is perhaps the most general class of tractable operators. If $A$ and $B$ are $\mathrm{C}^*$-algebras, recall that a linear map $T \co A \to B$ is called decomposable \cite{Haa85} if there exist linear maps $v_1,v_2 \co A \to B$ such that the linear map
\begin{equation}
\label{Matrice-2-2-Phi}
\Phi=\begin{bmatrix}
   v_1  &  T \\
   T^\circ  &  v_2  \\
\end{bmatrix}
\co \M_2(A) \to \M_2(B), \quad \begin{bmatrix}
   a  &  b \\
   c &  d  \\
\end{bmatrix}\mapsto 
\begin{bmatrix}
   v_1(a)  &  T(b) \\
   T^\circ(c)  &  v_2(d)  \\
\end{bmatrix}
\end{equation}
is completely positive, where $T^\circ(c) \ov{\mathrm{def}}{=} T(c^*)^*$. In this case, the maps $v_1$ and $v_2$ are completely positive and the decomposable norm of $T$ is defined by
\begin{equation}
\label{Norm-dec}
\norm{T}_{\dec,A \to B}
\ov{\mathrm{def}}{=} \inf\big\{\max\{\norm{v_1},\norm{v_2}\}\big\},
\end{equation}
where the infimum is taken over all maps $v_1$ and $v_2$. See the books \cite{BlM04}, \cite{EfR00} and \cite{Pis03} for more information on this classical notion. We also refer to \cite{ArK23} and \cite{JuR04} for the analogue notion for operators acting on a noncommutative $\L^p$-space $\L^p(\mathcal{M})$ associated to a von Neumann algebra $\cal{M}$ endowed with a normal semifinite faithful trace, for any $1 \leq p \leq \infty$. If $\cal{M}$ is approximately finite-dimensional (which is equivalent to injective), it is known that we have the isometric complex interpolation formula  
\begin{equation}
\label{Regular-as-interpolation-space}
\Dec(\L^p(\mathcal{M})) 
=(\CB(\mathcal{M}),\CB(\L^1(\mathcal{M})))^\frac{1}{p},
\end{equation}
for the Banach space $\Dec(\L^p(\mathcal{M}))$ of decomposable operators acting on the noncommutative $\L^p$-space $\L^p(\mathcal{M})$, which is a combination of \cite[Theorem 3.7]{Pis95} and the isometric identification \cite[Theorem 3.24 p.~41]{ArK23} between regular and decomposable operators. Here $\CB(\L^1(\mathcal{M}))$ is the space of completely bounded operators acting on the Banach space $\L^1(\mathcal{M})$ and the space $\CB(\mathcal{M})$ is defined similarly.

Recall that the group von Neumann algebra $\VN(G)$ of a locally compact group $G$ is the von Neumann algebra generated by the range $\lambda(G)$ of the left regular representation $\lambda$ of $G$ on the complex Hilbert space $\L^2(G)$ and that the subspace $\Span \{\lambda_s : s \in G\}$ is weak* dense in $\VN(G)$. If $G$ is abelian, then the von Neumann algebra $\VN(G)$ is $*$-isomorphic to the algebra $\L^\infty(\hat{G})$ of essentially bounded functions on the Pontryagin dual $\hat{G}$ of $G$. As fundamental models of quantum groups, these algebras play a crucial role in operator algebras. A Fourier multiplier acting on $\VN(G)$ is a weak* continuous linear operator $T \co \VN(G) \to \VN(G)$ that satisfies $T(\lambda_s)=\varphi_s\lambda_s$ for all $s \in G$, for some measurable function $\varphi \co G \to \mathbb{C}$. In this case, we let $M_\varphi \ov{\mathrm{def}}{=} T$. Our first result, proved in Corollary \ref{dec-vs-B(G)-discrete-group} is the following statement. This identification provides a concrete realization of the abstract space $\B(G)$ in terms of decomposable Fourier multipliers on $\VN(G)$, thereby bridging representation theory and decomposable operators.

\begin{thm}
The Fourier-Stieltjes algebra $\B(G)$ of a discrete group $G$ is canonically isometrically isomorphic to the algebra $\frak{M}^{\infty,\dec}(G)$ of decomposable Fourier multipliers on the group von Neumann algebra $\VN(G)$ via the map $\varphi \to M_\varphi$.
\end{thm}

This identification further highlights the ubiquity of the Fourier-Stieltjes algebra $\B(G)$ for a discrete group $G$. Indeed, we will show in Proposition \ref{prop-B(G)-inclus-dec} that for any locally compact group $G$ there exists a well-defined injective \textit{contractive} map from the Fourier-Stieltjes algebra $\B(G)$ into the space $\frak{M}^{\infty,\dec}(G)$ of decomposable Fourier multipliers and we will also examine the surjectivity of this map. We will show that the following property\footnote{\thefootnote. The subscript w* means <<weak* continuous>>. With the projection provided by \cite[Proposition 3.1 p.~24]{ArK23}, we could replace the space $\CB_{\w^*}(\VN(G))$ by the space $\CB(\VN(G))$ in this definition.} plays an important role in this problem.

\begin{defi}
\label{Defi-tilde-kappa}
Let $G$ be a locally compact group. We say that $G$ has property $(\kappa_\infty)$ if there exists a bounded projection $P_{G}^\infty \co \CB_{\w^*}(\VN(G)) \to \CB_{\w^*}(\VN(G))$ preserving the complete positivity onto the space $\mathfrak{M}^{\infty,\cb}(G)$ of completely bounded Fourier multipliers on the von Neumann algebra $\VN(G)$. In this case, the infimum of bounds of such projections will be denoted $\kappa_\infty(G)$:
\begin{equation}
\label{Kprime-def}
\kappa_\infty(G)
\ov{\mathrm{def}}{=} \inf \norm{P_{G}^\infty}_{\CB_{\w^*}(\VN(G)) \to \CB_{\w^*}(\VN(G))}.
\end{equation}
Finally, we let $\kappa_\infty(G) \ov{\mathrm{def}}{=} \infty$ if the locally compact group $G$ does not have $(\kappa_\infty)$.
\end{defi}

The constant $\kappa_\infty(G)$ is a variant of the relative projection constant 
$$
\inf \big\{\norm{P}_{X \to X} : P\text{ is a bounded projection from $X$ onto }Y \big\}
$$ 
of a closed subspace $Y$ of a Banach space $X$, e.g.~\cite[Definition 4.b.1 p.~231]{Kon86} or \cite[p.~112]{Wot91}. The property $(\kappa_\infty)$ means in particular that the space\footnote{\thefootnote. This space is denoted sometimes $\M_0\mathrm{A}(G)$ or $\M_\cb \mathrm{A}(G)$.} $\mathfrak{M}^{\infty,\cb}(G)$ is complemented in the space $\CB_{\w^*}(\VN(G))$ of weak* continuous completely bounded operators acting on $\VN(G)$. Indeed, in Proposition \ref{conj-1-1-correspondance}, we will prove that this property suffices to ensure that the previous inclusion $\B(G) \hookrightarrow \frak{M}^{\infty,\dec}(G)$ is a bijection. In order to prove that this map is an isometry, we need a matricial generalization of property $(\kappa_\infty)$ (satisfied for any discrete group $G$) and surprisingly (at first sight) the use of results on \textit{groupoids}. Note also that the existence of non-discrete and non-abelian locally compact groups with $(\kappa_\infty)$ (and even with the stronger property $(\kappa)$ of Definition \ref{Defi-complementation-G}) was a rather surprising result of the paper \cite{ArK23} since the proof of property $(\kappa_\infty)$ of a discrete group $G$ is an average argument relying on the compactness of the compact quantum group $(\VN(G),\Delta)$ defined by the group von Neumann algebra $\VN(G)$ and its canonical coproduct $\Delta$. According to \cite[Theorem 6.38 p.~121]{ArK23}, a second-countable pro-discrete locally compact group $G$ satisfies $\kappa_\infty(G)=1$. With sharp contrast, we will observe in this paper (see Example \ref{example-SL}), as announced in \cite{ArK23}, that the unimodular locally compact group $G=\SL_2(\R)$ does not have $(\kappa_\infty)$.

If $G$ is a locally compact group, with our result we can insert the space $\frak{M}^{\infty,\dec}(G)$ of decomposable Fourier multipliers acting on the von Neumann algebra $\VN(G)$ in the classical contractive inclusion $\B(G) \subseteq \frak{M}^{\infty,\cb}(G)$:
\begin{equation}
\label{Inclusions}
\B(G) 
\subseteq \frak{M}^{\infty,\dec}(G)
\subseteq \frak{M}^{\infty,\cb}(G).
\end{equation}
It is known \cite[p.~54]{Pis01} that the equality $\B(G) = \frak{M}^{\infty,\cb}(G)$ characterizes amenability for locally compact groups. This observation allows us to revisit another nice characterization of amenability of Lau and Paterson \cite[Corollary 3.2 p.~161]{LaP91} \cite[p.~85]{Pat88a}, which is described by the next theorem. 

\begin{thm}[Lau-Paterson]
\label{Th-Lau-Paterson}
Let $G$ be a locally compact group. The following properties are equivalent.
\begin{enumerate}
	\item The group von Neumann algebra $\VN(G)$ is injective and $G$ is inner amenable.
	\item $G$ is amenable. 
\end{enumerate}
\end{thm}

Recall that a locally compact group $G$ equipped with a left Haar measure is inner amenable if there exists a conjugation-invariant state on the algebra $\L^\infty(G)$. We introduce the following conjecture.

\begin{conj}
\label{conj}
Let $G$ be a locally compact group.
\begin{enumerate}
	\item $G$ is inner amenable if and only if we have the equality $\B(G) = \frak{M}^{\infty,\dec}(G)$.
	\item The von Neumann algebra $\VN(G)$ is injective if and only if we have $\frak{M}^{\infty,\dec}(G)= \frak{M}^{\infty,\cb}(G)$.
\end{enumerate}
\end{conj}  

We will prove the <<only if>> part of the first assertion for second-countable unimodular locally compact groups by showing in Theorem \ref{thm-SAIN-tilde-kappa} that inner amenability implies $\kappa_\infty(G)=1$, and hence $\B(G) = \frak{M}^{\infty,\dec}(G)$. We refer to Section \ref{Sec-approach} for a detailed presentation of our approach. This is our first main result.

\begin{thm}
Let $G$ be a second-countable unimodular locally compact group. If $G$ is inner amenable then we have $\B(G) = \frak{M}^{\infty,\dec}(G)$.
\end{thm}

The <<only if>> part of the second assertion of Conjecture \ref{conj} is true by a classical result of Haagerup \cite[Corollary 2.8 p.~201]{Haa85}. A consequence of our results is that the second point of Conjecture \ref{conj} is true for discrete groups, see Theorem \ref{Thm-conj-discrete-case}, and also for second-countable unimodular inner amenable locally compact groups (see Corollary \ref{cor-inner-66}), i.e.~we can state the following result. 
	
\begin{thm}
Let $G$ be a discrete group or a second-countable unimodular inner amenable locally compact group. Then the von Neumann algebra $\VN(G)$ is injective if and only if we have $\frak{M}^{\infty,\dec}(G)= \frak{M}^{\infty,\cb}(G)$.
\end{thm}
	
As a byproduct, we also obtain in Theorem \ref{thm-links-K-injective} the following new characterization of amenability, which is in the same spirit as the characterization of Lau and Paterson, previously discussed in Theorem \ref{Th-Lau-Paterson}.
	
\begin{thm}
\label{thm-links-K-injective-intro}
Let $G$ be a second-countable unimodular locally compact group. Then the following are equivalent.
\begin{enumerate}
	\item $\VN(G)$ is injective and $G$ has $(\kappa_\infty)$.
	\item $G$ is amenable.
\end{enumerate}
\end{thm}

Finally, we will observe in Example \ref{Example-SL2} that the first inclusion in \eqref{Inclusions} can also be strict, e.g.~for $G=\SL_2(\R)$. The converses of these results will need further investigations. We also give in Section \ref{subsec-inner-Folner} other characterizations of inner amenability for unimodular locally compact groups and we will use one of these in the proofs of our results.

Note that if $G$ is a unimodular locally compact group, there exists a canonical normal semifinite faithful trace on the group von Neumann algebra $\VN(G)$, allowing to introduce the associated noncommutative $\L^p$-space $\L^p(\VN(G))$ for any $1 \leq p < \infty$. In this context, we can introduce the space $\mathfrak{M}^{p,\cb}(G)$ of completely bounded Fourier multipliers acting on the noncommutative $\L^p$-space $\L^p(\VN(G))$. 
 
We describe a new class of locally compact groups with the following property introduced in \cite[Definition 1.1 p.~3]{ArK23} which requires a bounded projection at the level $p=\infty$ and a compatible\footnote{\thefootnote. The following remark is important to note. If $P_G^\infty \co \CB_{\w^*}(\VN(G)) \to \CB_{\w^*}(\VN(G))$ is a bounded projection onto the subspace $\mathfrak{M}^{\infty,\cb}(G)$ then we can define a map $P_{G}^1 \co \CB(\L^1(\VN(G))) \to \CB(\L^1(\VN(G)))$ by
$$
P_G^1(T) 
\ov{\textrm{def}}{=} (P_{G}^\infty(T^*)_*), \quad T \in \CB(\L^1(\VN(G))). 
$$ 
It is then easy to check that $P_G^1$ is a bounded projection preserving complete positivity onto the subspace $\mathfrak{M}^{1,\cb}(G)$ and its norm is equal to the one of $P_G^\infty$. It is important to note that there is no evidence that the maps $P_G^\infty$ and $P_G^1$ are compatible in the sense of interpolation. Consequently, the properties $(\kappa_\infty)$ and $(\kappa)$ seem to be different.} bounded projection at the level $p=1$. The compatibility is taken in the sense of interpolation theory described in the books \cite{BeL76} and \cite{Tri95}. This compatibility property is crucial in a companion paper %Section \ref{subsec-application-decomposable-norm-Fourier} 
in order to describe the decomposable norm of Fourier multipliers acting on noncommutative $\L^p$-spaces with the interpolation formula \eqref{Regular-as-interpolation-space} and a classical argument.

\begin{defi}
\label{Defi-complementation-G}
We say that a locally compact group $G$ has property $(\kappa)$ if there exist compatible bounded projections $P_{G}^\infty \co \CB_{\w^*}(\VN(G)) \to \CB_{\w^*}(\VN(G))$ and $P_{G}^1 \co \CB(\L^1(\VN(G))) \to \CB(\L^1(\VN(G)))$ onto the subspaces $\mathfrak{M}^{\infty,\cb}(G)$ and $\mathfrak{M}^{1,\cb}(G)$, preserving the complete positivity. In this case, we introduce the constant 
\begin{equation}
\label{Kappa-eq-def}
\kappa(G)
\ov{\mathrm{def}}{=} \inf \max\Big\{\norm{P_G^\infty}_{\CB_{\w^*}(\VN(G)) \to \CB_{\w^*}(\VN(G))},\norm{P_G^1}_{\CB(\L^1(\VN(G))) \to \CB(\L^1(\VN(G)))} \Big\},
\end{equation}
where the infimum is taken on all admissible couples $(P_G^\infty,P_G^1)$ of projections. 
Finally, we let $\kappa(G) \ov{\mathrm{def}}{=} \infty$ if the locally compact group $G$ does not have $(\kappa)$.
\end{defi}

The well-known average trick \cite[proof of Lemma 2.5]{Haa16} of Haagerup essentially implies that $\kappa(G)=1$ for any discrete group $G$, see \cite[Section 4.2]{ArK23}. In \cite[Proposition 6.43 p.~125]{ArK23} and \cite[Theorem 6.38 p.~121]{ArK23}, it is proved that an abelian locally compact group satisfies $\kappa(G)=1$ and that a second-countable pro-discrete locally compact group $G$ satisfies $\kappa(G)=1$. It is equally proved in \cite[Theorem 6.16 p.~96]{ArK23} that some class of second-countable unimodular locally compact groups approximable by lattice subgroups have $(\kappa)$.

Another very significant result that we obtain in this paper is described in the following statement, see Corollary \ref{cor-the-compatible-complementation}. Let us first recall that the concept of dimension of a \textit{suitable} topological space can be defined using the small inductive dimension, the large inductive dimension, or the covering dimension. In the case of a locally compact group $G$, these three notions of dimension coincide. We refer to Section \ref{Sec-finite-dim} for more background.

\begin{thm}
\label{th-intro-kappa}
A second-countable unimodular finite-dimensional amenable locally compact group $G$ has property $(\kappa)$. 
\end{thm}

An upper estimate of $\kappa(G)$ is possible for some groups. For example, in the case of a second-countable unimodular totally disconnected amenable locally compact group, our method gives $\kappa(G)=1$, which is a sharp result. Note that these results complement the result of our previous paper \cite{ArK23}. From this point of view, totally disconnected locally compact groups behave better than Lie groups, phenomenon that we already noticed in \cite{ArK23}. We refer to Section \ref{Sec-approach} for a detailed presentation of our approach.

The proof relies on the structure of finite-dimensional locally compact groups extracted from the solution to Hilbert's fifth problem. More precisely, we use a version of Iwasawa's local splitting theorem, which says that an $n$-dimensional second-countable locally compact group is \textit{locally} isomorphic to the product of a totally disconnected compact group $K$ and a Lie group $L$ of dimension $n$, to reduce the problem to totally disconnected groups and to connected Lie groups. It allows us to use doubling constants of the Carnot-Carath\'eodory metric of connected Lie groups for small balls to construct special suitable <<noncommutative functions>>, which are crucial for our proof. 

We will prove in Corollary \ref{cor-the-full-referees-complementation} a different result for the case of a second-countable unimodular amenable locally compact group $G$, using some other special <<noncommutative functions>>. We obtain that the space $\mathfrak{M}^{p,\cb}(G)$ of completely bounded Fourier multipliers on the noncommutative $\L^p$-space $\L^p(\VN(G))$ is contractively complemented in the space $\CB(\L^p(\VN(G)))$ of completely bounded operators acting on the Banach space $\L^p(\VN(G))$, by a contractive projection preserving the complete positivity.
Note that this map is \textit{contractive} which is better than the \textit{boundedness} of the maps $P_G^1$ and $P_G^\infty$ provided by Theorem \ref{th-intro-kappa} (when it applies), but only for \textit{one} value of $p$.
Note also that this property does not characterize amenability if $1<p<\infty$ since a discrete group $G$ such that the von Neumann algebra $\VN(G)$ is $\QWEP$\footnote{\thefootnote. It is not clear if this assumption is removable or not. It is required by the use of vector-valued noncommutative $\L^p$-spaces and the most general known theory needs the $\QWEP$ assumption.} also satisfies this property, see \cite[p.~334]{JuR03} and \cite[Theorem 4.2 p.~62]{ArK23}.

\begin{thm}
\label{thm-the-full-referees-intro}
Let $G$ be a second-countable unimodular amenable locally compact group. Let $1 < p < \infty$ such that $\frac{p}{p^*}$ is rational. Then there exists a contractive projection
\[ 
P^p_G \co \CB(\L^p(\VN(G))) \to \CB(\L^p(\VN(G))) 
\]
onto the subspace $\mathfrak{M}^{p,\cb}(G)$, preserving the complete positivity.
\end{thm}

In contrast to Theorem \ref{th-intro-kappa}, this result cannot be used to characterize the norms of decomposable multipliers on noncommutative $\L^p$-spaces.

Our results deepen the connection between the structural theory of locally compact groups and the analytic properties of their associated operator algebras, with new characterizations of amenability emerging from the perspective of Fourier multipliers.

%We finish the paper by providing in Theorem \ref{thm-strongly-non-dec} an example of a Fourier multiplier on the von Neumann algebra $\VN(G)$ of the free group $G=\mathbb{F}_2$ on two generators, which is not approximable by weak* continuous decomposable operators with respect to the completely bounded norm $\norm{\cdot}_{\cb,\VN(G) \to \VN(G)}$. Since such a multiplier does not belong to the closure $\ovl{\B(G)}^{\cb}$, this improves a classical result \cite[Corollary 3.10]{DCH85} of de Canni\`ere and Haagerup, which says that the contractive inclusion $\B(G) \subseteq \mathfrak{M}^{\infty,\cb}(G)$ is strict. Our method is different.
%
%\begin{thm}
%\label{Conj-existence-strongly}
%Let $G=\mathbb{F}_2$ be the free group on two generators. There exists a completely bounded Fourier multiplier $M_\varphi \co \VN(\mathbb{F}_2) \to \VN(\mathbb{F}_2)$ which is not approximable by weak* continuous decomposable operators with respect to the completely bounded norm $\norm{\cdot}_{\cb,\VN(G) \to \VN(G)}$.
%\end{thm}

%%%%%%%%%%%%%%%%%%%%%%%%%%%%%%%%%%%%%%%%%%%%%%%%%%%%%%%%%%%%%%%%%%%%%%%%%%%%%%%%%%%%%%%%%%%%%%
\subsection{Structure of the paper} 
To facilitate access to individual topics, each section is made as self-contained as possible. The paper is structured as follows.

Section \ref{Overview-kappa} provides background on Fourier-Stieltjes algebras, groupoids, and operator algebras. In Proposition \ref{Prop-Ruan-Dec}, we demonstrate that the space $\Dec(A,B)$ of decomposable operators between $\C^*$-algebras admits a canonical operator space structure. In Section \ref{Sec-FS-and-dec}, we prove in Proposition \ref{prop-B(G)-inclus-dec} that for any locally compact group $G$, there exists a well-defined injective completely contractive map from the Fourier-Stieltjes algebra $\B(G)$ into the space $\mathfrak{M}^{\infty,\dec}(G)$, consisting of decomposable Fourier multipliers on the group von Neumann algebra $\VN(G)$. Furthermore, we explore in Proposition \ref{conj-1-1-correspondance} the relationship between the equality $\B(G) = \mathfrak{M}^{\infty,\dec}(G)$ and property $(\kappa_{\infty})$. We also show in Theorem \ref{dec-vs-B(G)-discrete-group} that the Fourier-Stieltjes algebra $\B(G)$ of a discrete group $G$ is isometrically isomorphic to the algebra $\mathfrak{M}^{\infty,\dec}(G)$ of decomposable Fourier multipliers.  In Example \ref{Example-SL2}, we demonstrate that for $G = \SL_2(\R)$, the inclusion $\B(G) \subset \mathfrak{M}^{\infty,\dec}(G)$ is strict. 

In Section \ref{Sec-prelim-inner}, we provide background on inner amenability and amenability. In Section \ref{subsec-inner-Folner}, we establish various characterizations of inner amenability for unimodular locally compact groups, using asymptotically central nets of functions or inner F\o{}lner nets. These characterizations are employed in Section \ref{Sec-Herz-Schur}.

Section \ref{Sec-prel-complet} offers background on measurable Schur multipliers and Plancherel weights on group von Neumann algebras. In Section \ref{Sec-approach}, we outline the technical approach of this chapter. Section \ref{Mappings} presents the construction of some Schur multipliers derived from a (weak* continuous if $p = \infty$) completely bounded map $T \co \L^p(\VN(G)) \to \L^p(\VN(G))$, acting on the noncommutative $\L^p$-space $\L^p(\VN(G))$ of a second-countable unimodular locally compact group $G$. In Section \ref{Sec-Herz-Schur}, we show that the symbol can be chosen as a Herz-Schur symbol if the group $G$ is inner amenable. Section \ref{Section-p=1-p-infty} examines the symbols of these Schur multipliers for $p=1$ and $p=\infty$. Section \ref{Sec-convergence-continuous} explores the convergence of the symbols of Schur multipliers, while Section \ref{Sec-finite-dim} focuses on Lie groups and totally disconnected locally compact groups, culminating in the proof in Section \ref{Sec-Th-complementation} that unimodular finite-dimensional amenable locally compact groups have property $(\kappa)$. In Section \ref{Section-Schur}, we construct a contractive projection from the space of completely bounded Schur multipliers $\mathfrak{M}^{p,\cb}_G$ onto the subspace $\mathfrak{M}^{p,\cb,\HS}_G$ of Herz-Schur multipliers, in the case where $G$ is amenable. This result will be used in Section \ref{Sec-Th-complementation}, which contains our main complementation results.

In Section \ref{Sec-charac-amen}, Theorem \ref{thm-links-K-injective} presents a new characterization of amenability for second-countable unimodular locally compact groups. In Example \ref{example-SL}, we observe that the unimodular locally compact group $G=\SL_2(\R)$ does not have property $(\kappa_\infty)$.  Finally, in Section \ref{Sec-Herz}, we show that if there exists a bounded projection $Q \co \mathfrak{M}^{\infty}_G \to \mathfrak{M}^{\infty}_G$ onto the space of completely bounded Herz-Schur multipliers $\mathfrak{M}^{\infty,\HS}_G$ over the space $\cal{B}(\L^2(G))$ of bounded operators on the Hilbert space $\L^2(G)$, preserving the complete positivity for some second-countable unimodular locally compact group $G$ such that the von Neumann algebra $\VN(G)$ is injective, then $G$ must be amenable.

%%%%%%%%%%%%%%%%%%%%%%%%%%%%%%%%%%%%%%%%%%%%%%%%%%%%%%%%%%%%%%%%%%%%%%%%%%%%%%%%%%%%%%%
\section{Fourier-Stieltjes algebras and decomposable multipliers on $\VN(G)$}
\label{sec-Divers}
%%%%%%%%%%%%%%%%%%%%%%%%%%%%%%%%%%%%%%%%%%%%%%%%%%%%%%%%%%%%%%%%%%%%%%%%%%%%%%%%%%%
\subsection{Preliminaries}
\label{Overview-kappa}
%Choose a fixed Hilbert space $H_0$ with dimension $\card(G)$ and following \cite[Definition 7.19]{Phi1} we define a unitary representation $U$ of $G$ to be the direct sum of all possible unitary representations of $G$ on subspaces of $H$ \textbf{ce qui precede n'est pas certain: }

%\paragraph{$\sigma$-strong* topology}
%Recall that if $\cal{M}$ is a von Neumann algebra acting on a complex Hilbert space $H$ then the $\sigma$-strong* topology on $\cal{M}$ is defined by the seminorms $x \mapsto \big(\sum_{n=1}^{\infty} (\norm{x(\xi_n)}^2+\norm{x^*(\xi_n)}^2)\big)^{\frac{1}{2}}$, where $\xi_n \in H$ and $\sum_{n=1}^{\infty} \norm{\xi_n}^2<\infty$. The following result is well-known and useful when dealing with operator-valued integrals. We refer to \cite[Lemma 2.2 p.~4]{Arh20} for a detailed proof.
%
%\begin{lemma}
%\label{Lemma-quasi-complete}
%Let $\cal{M}$ be a von Neumann algebra acting on a complex Hilbert space $H$. When $\cal{M}$ is endowed with the canonical locally convex structure inducing the $\sigma$-strong* topology, it is quasi-complete.
%\end{lemma}

\paragraph{Decomposable maps} 
Recall that the notion of decomposable map is defined in \eqref{Matrice-2-2-Phi}. Consider some $\mathrm{C}^*$-algebras $A$, $B$ and $C$. Let $T_1 \co A \to B$ and $T_2 \co B \to C$ be some decomposable maps. Then by \cite[Proposition 1.3 (5) p.~177]{Haa85}
\begin{comment}
\footnote[3]{3. With obvious notations, we can write
$$
\Phi_{T_2}\circ\Phi_{T_1}
=\begin{bmatrix} 
   w_1  &  T_2 \\
   T_2^\circ  &  w_2  \\
\end{bmatrix} 
\circ
\begin{bmatrix} 
    v_1  &   T_1 \\
    T_1^\circ  &  v_2  \\
\end{bmatrix} 
=
\begin{bmatrix} 
   w_1\circ v_1  &  T_2 \circ T_1 \\
   T_2^\circ\circ T_1^\circ  &   w_2 \circ v_2  \\
\end{bmatrix} 
=
\begin{bmatrix} 
   w_1\circ v_1           &  T_2 \circ T_1 \\
   (T_2 \circ T_1)^\circ  &   w_2 \circ v_2  \\
\end{bmatrix} .
$$
Moreover, by composition, this map is completely positive. Hence $T_2 \circ T_1$ is decomposable. Furthermore, we deduce that
\begin{align*}
\MoveEqLeft
  \norm{T_2 \circ T_1}_{\dec}  
		\leq \max\big\{\norm{w_1\circ v_1}, \norm{w_2 \circ v_2}\big\}\\
		&\leq \max\big\{\norm{w_1}\norm{v_1}, \norm{w_2} \norm{v_2}\big\}
		\leq \max\big\{\norm{w_1},\norm{w_2}\big\} \max\big\{\norm{v_1},\norm{v_2}\big\}.
\end{align*}
Passing to the infimum, we conclude that $\norm{T_2 \circ T_1}_{\dec} \leq \norm{T_2}_{\dec}\norm{T_1}_{\dec}$.
\end{comment} 
that the composition $T_2 \circ T_1 \co A \to C$ is decomposable and that 
\begin{equation}
\label{Composition-dec}
\norm{T_2 \circ T_1}_{\dec, A \to C} 
\leq \norm{T_2}_{\dec, B \to C} \norm{T_1}_{\dec, A \to B}.
\end{equation}
By \cite[Proposition 1.3 (4) p.~177]{Haa85}, any completely positive map $T \co A \to B$ between $\C^*$-algebras is decomposable and we have
\begin{equation}
\label{dec-et-cp}
\norm{T}_{\dec,A \to B} 
=\norm{T}_{\cb,A \to B}
=\norm{T}_{A \to B}.
\end{equation}
It is known that the space $\Dec(A, B)$ of decomposable maps is a Banach space by \cite[Proposition 1.4 p.~182]{Haa85} and coincides with the span of completely positive maps, see \cite[p.~175]{Haa85}. By \cite[Proposition 1.3 (3) p.~177]{Haa85} or \cite[Lemma 5.4.3 p.~96]{EfR00}, any decomposable map $T \co A \to B$ is completely bounded with $\norm{T}_{\cb,A \to B}  \leq \norm{T}_{\dec,A \to B}$. Moreover, if $B$ is injective, then according to \cite[Theorem 1.6 p.~184]{Haa85}, we have
\begin{equation}
\label{dec=cb}
\norm{T}_{\dec,A \to B}
=\norm{T}_{\cb,A \to B}.
\end{equation}
We will use in the first proof of Proposition \ref{prop-B(G)-inclus-dec} the following elementary lemma.

\begin{lemma}
\label{Lemma-tensor-dec-2}
Let $\cal{M}_1,\cal{M}_2$ and $\cal{N}$ be von Neumann algebras and let $T \co \cal{M}_1 \to \cal{M}_2$ be a weak* continuous decomposable map. Then we have a well-defined weak* continuous decomposable map $\Id_\cal{N} \ot T \co \cal{N} \otvn \cal{M}_1 \to \cal{N} \otvn \cal{M}_2$ and 
\begin{equation}
\label{dec-tensor-2}
\norm{\Id_\cal{N} \ot T}_{\dec,\cal{N} \otvn \cal{M}_1 \to \cal{N} \otvn \cal{M}_2}
\leq \norm{T}_{\dec,\cal{M}_1 \to \cal{M}_2}.
\end{equation}
\end{lemma}

\begin{proof}
Note that the decomposable map $T$ is completely bounded by \cite[Proposition 1.3 (3) p.~177]{Haa85}. By \cite[p.~40]{BlM04}, 
%\cite[Lemma 1.5 b) p.~461]{DCH85}, 
we infer that we have a well-defined weak* continuous completely bounded map $\Id_\cal{N} \ot T \co \cal{N} \otvn \cal{M}_1 \to \cal{N} \otvn \cal{M}_2$. By \cite[Remark 1.5 p.~183]{Haa85}, the infimum in the definition of the decomposable norm given in \eqref{Norm-dec} is actually a minimum. Consequently, there exist some linear maps $v_1,v_2 \co \cal{M}_1 \to \cal{M}_2$ such that the map $
\begin{bmatrix} 
v_1 & T \\ 
T^\circ & v_2 
\end{bmatrix} 
\co \M_2(\cal{M}_1) \to \M_2(\cal{M}_2)$ is completely positive with $\max\{\norm{v_1},\norm{v_2}\} = \norm{T}_{\dec,\cal{M}_1 \to \cal{M}_2}$. It is not difficult to see that we can suppose that $v_1$ and $v_2$ are weak* continuous by using \cite[Proposition 3.1 p.~24]{ArK23} as in the proof of \cite[Proposition 3.4 p.~26]{ArK23}. Then by \cite[Proposition 4.3.7 p.~225]{Li92} the tensor product
$$
\begin{bmatrix} 
\Id_\cal{N} \ot v_1 & \Id_\cal{N} \ot T \\ 
\Id_\cal{N} \ot T^\circ & \Id_\cal{N} \ot v_2 
\end{bmatrix} 
=\Id_{\cal{N}} \ot
\begin{bmatrix} 
v_1 &  T \\ 
T^\circ & v_2 
\end{bmatrix} 
\co \M_2(\cal{N} \otvn \cal{M}_1) \to \M_2(\cal{N} \otvn \cal{M}_2)
$$
is a well-defined completely positive map. We deduce that the map $\Id_\cal{N} \ot T \co \cal{N} \otvn \cal{M}_1 \to \cal{N} \otvn \cal{M}_2$ is decomposable with 
\begin{align*}
\MoveEqLeft
\norm{\Id_\cal{N} \ot T}_{\dec} 
\ov{\eqref{Norm-dec}}{\leq} \max\{\norm{\Id_\cal{N} \ot v_1},\norm{\Id_\cal{N} \ot v_2 }\} \leq \max\{\norm{v_1}_{\cb},\norm{v_2}_{\cb}\} \\ 
&\ov{\eqref{dec-et-cp}}{=} \max\{\norm{v_1},\norm{v_2}\} = \norm{T}_{\dec,\cal{M}_1 \to \cal{M}_2},
\end{align*}
where we use in the first equality the complete positivity of the linear maps $v_1$ and $v_2$.
\end{proof}

Finally, if $A$ and $B$ are $\mathrm{C}^*$-algebras, with $B$ unital, we will show that the space $\Dec(A, B)$ of decomposable maps can be endowed with an operator space structure. To demonstrate this, suppose that $[T_{ij}]$ belongs to the matrix space $\M_{n}(\Dec(A, B))$, where $n \geq 1$ is an integer. We identify the matrix $[T_{ij}]$ with the map $A \mapsto \M_{n}(B)$, $x \mapsto [T_{ij}(x)]$. We define a norm on the space $\M_{n}(\Dec(A,B))$ by setting 
\begin{equation}
\label{Norms-dec}
\bnorm{[T_{i j}]}_{\M_{n}(\Dec(A,B))} 
\ov{\mathrm{def}}{=} \bnorm{x \mapsto [T_{ij}(x)]}_{\Dec(A,\M_{n}(B))}.
\end{equation} 
In short, we make the identification $\M_{n}(\Dec(A,B)) = \Dec(A,\M_{n}(B))$.

\begin{prop}
\label{Prop-Ruan-Dec}
Let $A$ and $B$ be $\mathrm{C}^*$-algebras, with $B$ unital. When endowed with the matricial norms from \eqref{Norms-dec}, the Banach space $\Dec(A, B)$ acquires the structure of an operator space.
\end{prop}

\begin{proof}
Let $X,Y \in \M_{n}$ and $[T_{i j}] \in \M_{n}(\Dec(A,B))$ for some integer $n \geq 1$. Note that by \cite[Exercise 12.1 p.~251]{Pis03} the two-sided multiplication map $u \co \M_{n}(B) \to \M_{n}(B)$, $y \mapsto (X  \ot 1_B)y(Y  \ot 1_B)$ is decomposable with $\norm{u}_{\dec,\M_{n}(B) \to \M_{n}(B)} \leq \norm{X  \ot 1_B}_{\M_{n}(B)} \norm{Y  \ot 1_B}_{\M_{n}(B)}$. Using this observation in the  equality, we obtain
\begin{align*}
\MoveEqLeft
\bnorm{X[T_{i j}]Y}_{\M_{n}(\Dec(A,B))}  
\ov{\eqref{Norms-dec}}{=} \bnorm{x \mapsto X [T_{ij}(x)]Y}_{\Dec(A,\M_{n}(B))} \\          
&=\bnorm{x \mapsto (X  \ot 1_B) [T_{ij}(x)](Y  \ot 1_B)}_{\Dec(A,\M_{n}(B))} \\
&\leq \norm{X  \ot 1_B}_{\M_{n}(B)} \norm{Y  \ot 1_B}_{\M_{n}(B)} \bnorm{x \mapsto [T_{ij}(x)]}_{\Dec(A,\M_{n}(B))} \\
&\ov{\eqref{Norms-dec}}{=} \norm{X}_{\M_{n}}  \bnorm{[T_{i j}]}_{\M_{n}(\Dec(A,B))} \norm{Y}_{\M_{n}}.
\end{align*} 
Let $[T_{i j}] \in \M_{n}(\Dec(A,B))$ and $[S_{kl}] \in \M_{m}(\Dec(A,B))$ for some integers $n,m \geq 1$. Using \cite[Lemma 6.8 p.~118]{Pis20} in the second equality, we have
\begin{align*}
\MoveEqLeft
\bnorm{[T_{i j}] \oplus [S_{kl}]}_{\M_{n+m}(\Dec(A,B))}             
\ov{\eqref{Norms-dec}}{=} \bnorm{x \mapsto [T_{ij}(x)] \oplus [S_{kl}(x)]}_{\Dec(A,\M_{n+m}(B))} \\
&=\max\Big\{\bnorm{x \mapsto [T_{ij}(x)]}_{\Dec(A,\M_{n}(B))},\bnorm{x \mapsto [S_{kl}(x)]}_{\Dec(A,\M_{m}(B))} \Big\} \\
&\ov{\eqref{Norms-dec}}{=} \max\Big\{\norm{[T_{i j}]}_{\M_{n}(\Dec(A,B))} +\norm{[S_{kl}]}_{\M_{m}(\Dec(A,B))} \Big\}.
\end{align*} 
Now, it suffices to use Ruan's theorem \cite[p.~35]{Pis03} or \cite[Proposition 2.3.6 p.~34]{EfR00}. 
\end{proof}

We finish by providing another formula for the decomposable norm.

\begin{prop}
\label{prop-dec-sqrt}
Consider a decomposable map $T \co A \to B$ between $\C^*$-algebras. Then
\begin{equation}
\label{Norm-dec-sqrt}
\norm{T}_{\dec,A \to B}
= \inf\left\{ \norm{v_1}^{\frac12} \norm{v_2}^{\frac12} \right\},
\end{equation}
where the infimum is taken over all maps $v_1$ and $v_2$ such that the operator $\Phi$ introduced in \eqref{Matrice-2-2-Phi} is completely positive.
\end{prop}

\begin{proof}
The inequality $\geq$ is obvious. Now, we show the reverse inequality, assume that the operator $\Phi \ov{\mathrm{def}}{=} \begin{bmatrix}
   v_1  &  T \\
   T^\circ  &  v_2  \\
\end{bmatrix}$ of \eqref{Matrice-2-2-Phi} is completely positive for some linear maps $v_1,v_2$.
 Suppose that $v_1=0$. Consider a positive element $x \in A$. Since the element $\begin{bmatrix}
  x   &  x \\
  x   &  x \\
\end{bmatrix}$ in $\M_2(A)$ is positive, we see that $\left(\begin{bmatrix}
   0  &  T \\
   T^\circ  &  v_2  \\
\end{bmatrix}\right)\left(\begin{bmatrix}
  x   &  x \\
  x   &  x \\
\end{bmatrix}\right)=\begin{bmatrix}
   0  &  T(x) \\
   T^\circ(x)  &  v_2(x)  \\
\end{bmatrix}$ is positive. By \cite[Proposition 1.3.2 p.~13]{Bha07}, we infer that $\norm{T(x)} \leq \norm{0}\bnorm{v_2(x)^{\frac{1}{2}}}$. As every element of $A$ is a linear combination of positive elements \cite[p.~17]{BlM04}, we conclude that $T=0$. %The situation is identical if $v_2=0$.

So we can suppose that $v_1 \not =0$. For any $t > 0$ we define the positive matrix $A_t \ov{\mathrm{def}}{=} 
\begin{pmatrix} 
\sqrt{t} & 0 \\ 0 & 
\frac{1}{\sqrt{t}}
\end{pmatrix}$ in $\M_2$. Then the linear map
$\Phi_t \co \M_2(A) \to \M_2(B)$, $x \mapsto A_t \Phi(x) A_t$ is also completely positive.
For any $t > 0$, observe that $
\Phi_t 
= \begin{pmatrix} 
t v_1 & T \\ 
T^\circ &  \frac{1}{t} v_2 
\end{pmatrix}
$. 
Hence $\norm{T}_{\dec,A \to B} \ov{\eqref{Norm-dec}}{\leq} \inf_{t > 0} \max \{ t \norm{v_1}, t^{-1} \norm{v_2} \} $. The choice $t = \norm{v_2}^{\frac12} \norm{v_1}^{-\frac12}$ gives the inequality $\norm{T}_{\dec,A \to B} \leq \norm{v_1}^{\frac12} \norm{v_2}^{\frac12}$. 
%(avec $t^2 = \norm{v_2}^{\frac12} \norm{v_1}^{-\frac12}$) $= \norm{v_1}^{\frac12} \norm{v_2}^{\frac12}$.
Taking the infimum on $v_1,v_2$, we obtain $\norm{T}_{\dec,A \to B} \leq \inf\left\{ \norm{v_1}^{\frac12} \norm{v_2}^{\frac12} \right\}$.
\end{proof}

\paragraph{Full group $\mathrm{C}^*$-algebras} 
%Following \cite[Definition 1.4.18 p.~23]{KaL18}, we denote by $\mathrm{P}(G)$ the set of continuous positive definite functions, and we let $\mathrm{P}_1(G) \ov{\mathrm{def}}{=} \{\varphi \in \mathrm{P}(G) : \varphi(e) = 1\}$. 
%Then the Gelfand-Naimark construction \cite[Theorem 13.4.5 p.~288]{Dix77} provides for each continuous positive definite function $\varphi \co G\to \mathbb{C}$ a continuous unitary representation $\pi_\varphi$ of $G$ on a complex Hilbert space $H_\pi$ and a vector $\xi$ in $H_\pi$ such that $\varphi(s)=\la \pi_\varphi(s)\xi,\xi\ra_{H}$ for any $s \in G$. 
%The map $f \mapsto \norm{f} \ov{\mathrm{def}}{=} \sup_\pi \norm{\pi(f)}$, where the supremum is taken on all topologically irreducible representations of the involutive Banach algebra $\L^1(G)$, is a seminorm on the space $\L^1(G)$. Let $I$ be the set of $f \in \L^1(G)$ such that $\norm{f}' = 0$, which is a closed self-adjoint two-sided ideal of $\L^1(G)$. The map $f \mapsto \norm{f}'$ defines a norm on the quotient $\L^1(G)/I$. The full group $\C^*$-algebra $\C^*(G)$ is the completion of the quotient algebra $\L^1(G)/I$ for this norm. In other words, $\C^*(G)$ is the enveloping $\mathrm{C}^*$-algebra of the involutive Banach algebra $\L^1(G)$.
%the continuous unitary representation of $G$ is $U \ov{\mathrm{def}}{=} \oplus_{\varphi \in \mathrm{P}_1(G)} \pi_\varphi$. We let $H =\oplus H_{\varphi \in \mathrm{P}_1(G)}$ on the Hilbert space $H$.

Let $G$ be a locally compact group equipped with a left Haar measure $\mu_G$. Consider the direct sum $U$ of all equivalence classes of cyclic continuous unitary representations of $G$. We denote by $H$ the associated Hilbert space.  Following \cite[Definition 8.B.1 p.~243]{BeH20}, we define the  full group $\C^*$-algebra $\C^*(G)$ to be the norm closure in $\mathcal{B}(H)$ of $U(\L^1(G))$, where $U \co \L^1(G) \to \cal{B}(H)$, $f \mapsto \int_G f(s)U_s \d\mu_G(s)$ denotes the (injective) integrated representation associated to $U$. Here the latter integral is understood in the weak operator sense. So we can identify $\L^1(G)$ as a dense subspace of the algebra $\mathrm{C}^*(G)$. 

Recall that there is a one-to-one correspondence between the continuous unitary representations of the group $G$ and the non-degenerate representations of the $\mathrm{C}^*$-algebra $\mathrm{C}^*(G)$, see \cite[Theorem 13.9.3 p.~303]{Dix77} and \cite[Theorem 12.4.1 p.~1383]{Pal01} for details. 

Finally, we denote by $\W^*(G)$ the enveloping von Neumann algebra of the $\mathrm{C}^*$-algebra $\mathrm{C}^*(G)$, introduced in \cite{Ern64}, under the name <<big group algebra>>. This means that this von Neumann algebra is the weak closure of $\pi(\mathrm{C}^*(G))$, where $\pi$ is the universal representation of the $\C^*$-algebra $\mathrm{C}^*(G)$. By \cite[p.~265]{Dix77}, we have a canonical isometric isomorphism $\W^*(G)=\C^*(G)^{**}$, which is bicontinuous for the weak operator topology on $\W^*(G)$ and the weak* topology on the bidual $\C^*(G)^{**}$.

\begin{example} \normalfont
If the locally compact group $G$ is \textit{abelian} then by \cite[Example p.~225]{Fol16} the $\mathrm{C}^*$-algebra $\mathrm{C}^*(G)$ is $*$-isomorphic to the $\mathrm{C}^*$-algebra $\C_0(\hat{G})$, where $\hat{G}$ is the Pontryagin dual of $G$.
\end{example}

%Recall that the $\mathrm{C}^*$-algebra $\mathrm{C}^*(G)$ is the $\mathrm{C}^*$-enveloping algebra of $\L^1(G)$, i.e.~the completion of $\L^1(G)$ for the norm
%$$
%\norm{f}
%=\sup_\pi \norm{\pi(f)}_{\B(H_\pi)}
%$$
%where the supremum is taken over all equivalence classes of cyclic continuous representations $\pi \co G \to \B(H_\pi)$ of $G$ \cite[Remark 8.B.2]{BeH20} \cite[page 369]{Fel60}.

\paragraph{Fourier-Stieltjes algebras} Recall that we defined the Fourier-Stieltjes algebra $\B(G)$ of a locally compact group $G$ in \eqref{BG-as-entries}. It is known that it is the complex linear span of the set of all continuous positive definite functions on $G$, see \cite[Definition 2.1.5 p.~40]{KaL18}. Equipped with pointwise multiplication and addition $\B(G)$ becomes a commutative unital Banach algebra by \cite[Theorem 2.1.11 p.~44]{KaL18}. 

Let $\mu_G$ be a left Haar measure of $G$. If $\varphi \in \B(G)$ then by \cite[p.~193]{Eym64} the linear form $\omega_\varphi \co \L^1(G) \to \mathbb{C}$ defined by
\begin{equation}
\label{Def-omega-varphi}
\omega_\varphi(f)
\ov{\mathrm{def}}{=} \int_G \varphi(s)f(s) \d\mu_G(s)
\end{equation}
extends to a bounded linear form $\omega_\varphi \co \mathrm{C}^*(G) \to \mathbb{C}$ with $\norm{\varphi}_{\B(G)}=\norm{\omega_\varphi}$. It is well-known that each bounded linear form on $\mathrm{C}^*(G)$ satisfies this description, i.e.~we have $\B(G)=\mathrm{C}^*(G)^*$ isometrically, see \cite[p.~192]{Eym64} or \cite[p.~40]{KaL18}. Moreover, by \cite[Lemma 1.4 p.~370]{Fel60} and \cite[Theorem 1.6.1 p.~29]{KaL18} the linear form $\L^1(G) \to \mathbb{C}$, $f \mapsto \int_G \varphi(s)f(s) \d\mu_G(s)$ extends to a positive linear form on the $\mathrm{C}^*$-algebra $\mathrm{C}^*(G)$ if and only if $\varphi$ is a continuous positive definite function. %, and we let $\mathrm{P}_1(G) \ov{\mathrm{def}}{=} \{\varphi \in \mathrm{P}(G) : \varphi(e) = 1\}$. %, that is the von Neumann algebra generated by .

\begin{example} \normalfont
\label{norm-B(G)-commutatif}
If the locally compact group $G$ is \textit{abelian}, recall that the Fourier transform $\hat{\mu} \co G \to \mathbb{C}$ of a bounded regular complex Borel measure $\mu \in \M(\hat{G})$ on the Pontryagin dual $\hat{G}$ is given by
%\begin{equation}
%\label{Def-Fourier-transform} 
$\hat{\mu}(s)
\ov{\mathrm{def}}{=} \int_{\hat{G}} \ovl{\chi(s)}\d\mu(\chi)$ where $s \in G$. 
%\end{equation}
According to \cite[Exemple p.~92]{Eym64}, $\B(G)$ is the space of Fourier transforms $\varphi=\hat{\mu}$ of bounded regular complex Borel measures $\mu \in \M(\hat{G})$ and $\norm{\varphi}_{\B(G)}=\norm{\mu}_{\M(\hat{G})}$. %Voir KaL p45

%If $\mu$ is a bounded complex measure on $\hat{G}$, 
%In particular, if $\mu$ is positive, we have
%\begin{equation}
%\label{norm-B(G)-commutatif}
%\norm{\hat{\mu}}_{\B(G)}
%\ov{\eqref{norm-pos-def}}{=} \hat{\mu}(e)
%\ov{\eqref{Def-Fourier-transform}}{=} \int_{\hat{G}} \d\mu(\chi)
%=\norm{\mu}_{\M(\hat{G})}.
%\end{equation}
\end{example}

We will use the next observation written without proof in \cite[p.~188]{Eym64}. %in \cite[p.~212]{Wal89}
For the sake of completeness, we give a proof.

\begin{prop}
\label{Conj-pos-def}
Let $G$ be a locally compact group. Let $\varphi \co G \to \mathbb{C}$ be a continuous positive definite function. We have
\begin{equation}
\label{norm-pos-def}
\norm{\varphi}_{\B(G)}
=\varphi(e).
\end{equation} 
\end{prop}

\begin{proof}
\textit{First proof if $G$ is discrete.} Let $G$ be a discrete group. By \cite[13.9.2 p.~303]{Dix77},  the full $\mathrm{C}^*$-algebra $\mathrm{C}^*(G)$ of $G$ is unital\footnote{\thefootnote. Actually, by \cite{Mil71} the full $\mathrm{C}^*$-algebra $\mathrm{C}^*(G)$ of a locally compact group $G$ is unital if and only if $G$ is discrete.}. By \eqref{Def-omega-varphi}, we have a positive linear form $\omega_\varphi \co \mathrm{C}^*(G) \to \mathbb{C}$, $U(s) \mapsto \varphi(s)$. So using \cite[Theorem 4.3.2 p.~256]{KaR97} in the second equality, we conclude that 
$$
\norm{\varphi}_{\B(G)}
=\norm{\omega_\varphi}_{\mathrm{C}^*(G)^*}
=\omega_\varphi(1)
=\omega_\varphi(U(e))
=\varphi(e).
$$ 
\noindent\textit{Second proof if $G$ is locally compact.} Using \cite[Remark 2.1.10 p.~43]{KaL18}, we know that $\norm{\varphi}_{\B(G)} \geq \norm{\varphi}_{\L^\infty(G)} \geq \varphi(e)$.
Furthermore, with \cite[Theorem 13.4.5 p.~288]{Dix77}, we can write $\varphi=\langle \pi(\cdot)\xi,\xi \rangle_H$, where $\pi$ is a continuous unitary representation of $G$ on some complex Hilbert space $H$ and $\xi \in H$. We deduce that $\varphi(e)=\langle \xi,\xi \rangle_H=\norm{\xi}_H^2 
\ov{\eqref{Norm-BG}}{\geq} \norm{\varphi}_{\B(G)}$.
\end{proof}

\paragraph{Fell's absorption principle} Let $G$ be a locally compact group. Let $\pi \co G \to \mathcal{B}(H)$ be any continuous unitary representation of $G$. Recall Fell's absorption principle, e.g.~\cite[Lemma 5.5.3 p.~187]{KaL18} (see also \cite[Proposition 8.1 p.~149]{Pis03} for the discrete case). 
%\textbf{ see also Runde Representations of locally compact groups on QSLp-spaces prop 5.1}. 
If $1_H \co G \to \mathcal{B}(H)$, $s \mapsto \Id_H$ is the identity representation, we have a unitary equivalence
\begin{equation}
\label{Fell}
\lambda \ot \pi
\approx\lambda \ot 1_H.
\end{equation}

%\cite{Ren12}\cite{RaW91} \cite{Bun06}

\paragraph{Groupoids} We refer to \cite{Hah78}, \cite{Muh90}, \cite{Pat03}, \cite{Pat04}, \cite{Pat99}, \cite{Ren97} \cite{Ren80} and \cite{RaW97} for background on groupoids. %[Etale groupoids and their C-algebras Aidan Sims, intro pas mal] [The Fourier algebra of a locally trivial groupoid] [IMPRIMITIVITY THEOREM FOR %GROUPOID REPRESENTATIONS: pas mal pour les rep] 
A groupoid is a set $G$ together with a distinguished subset $G^{(2)} \subseteq G \times G$, a multiplication map $G^{(2)} \to G$, $(s,t) \mapsto st$ and an inverse map $G \to G$, $s \mapsto s^{-1}$ such that
\begin{enumerate}
\item\label{it:gpd0} for any $s \in G$ we have $(s^{-1})^{-1} = s$,

\item\label{it:gpd1} If $(s,t), (t,r) \in G^{(2)}$ then $(st,r)$ and $(s,tr)$ belong to $G^{(2)}$ and $(st)r = s(tr)$,

\item\label{it:gpd2} for any $s \in G$ we have $(s,s^{-1}) \in G^{(2)}$ and if $s,r \in G$ satisfies $(s,r) \in G^{(2)}$, we have $s^{-1} (s r) = r$ and $(sr)r^{-1} = s$.
\end{enumerate}

We say that $G^{(2)}$ is the set of composable pairs. Second axiom shows that for products of three groupoid elements, there is no ambiguity in dropping the parentheses, and simply writing
$str$ for $(st)r$. A groupoid $G$ is a group if and only if its unit space $G^{(0)}$ is a singleton.

Given a groupoid $G$ we shall write $G^{(0)} \ov{\mathrm{def}}{=} \{s^{-1}s : s \in G\}$ and refer to elements of $G^{(0)}$ as units and to $G^{(0)}$ itself as the unit space. Since $(s^{-1})^{-1} = s$ for any $s \in G$, we also have $G^{(0)} = \{ss^{-1} : s \in G\}$. We define the range and domain maps $r,d  \co G \to G^{(0)}$ by
$$
r(s) \ov{\mathrm{def}}{=} ss^{-1}
\qquad\text{ and }\qquad 
d(s)  \ov{\mathrm{def}}{=} s^{-1}s, \quad s \in G.
$$
For any $s,t \in G$ we have $(s,t) \in G^{(2)}$ if and only if $d(s) = r(t)$. For any unit $u \in G^{(0)}$, we let $G^u \ov{\mathrm{def}}{=} r^{-1}(\{u\})$ and $G_u \ov{\mathrm{def}}{=} d^{-1}(\{u\})$.

\paragraph{Measured groupoids} A locally compact groupoid is a groupoid $G$ equipped with a locally compact topology, where the inversion map $s \mapsto s^{-1}$ is continuous, and the multiplication map $(s,t) \mapsto st$ is continuous with respect to the relative topology on $G^{(2)}$, considered as a subset of $G \times G$.

Following \cite[Definition 2.2 p.~16]{Ren80} and \cite[Definition 2.28 p.~24]{Muh90}, a left Haar system for $G$ is defined as being a family $(\nu^u)_{u \in G^{(0)}}$ of positive Radon measures on $G$ such that
\begin{enumerate}
	\item the support $\supp \nu^u$ of the measure $\nu^u$ is $G^u$,
	
	\item for any function $f \in \C_c(G)$, the map $G^{(0)} \to \mathbb{C}$, $u \mapsto \int_G f \d\nu^u$ is continuous,
	
	\item for any function $f \in \C_c(G)$ and any $s \in G$ we have $\int_G f(st) \d \nu^{d(s)}(t)=\int_G f(t) \d \nu^{r(s)}(t)$.
\end{enumerate}

Roughly speaking, to each unit $u$ we associate a measure $\nu^u$ supported on $G^u$. With such system, the space $\C_c(G)$ of continuous functions with compact support, endowed with the operations
$$
(f*g)(s)
\ov{\mathrm{def}}{=} \int_G f(t)g(t^{-1}s) \d\nu^{r(s)}(t), \quad 
f^*(s)
\ov{\mathrm{def}}{=} \ovl{f(s^{-1})}, \quad s \in G,
$$
is a $*$-algebra, according to \cite[p.~38]{Pat99}.

Let us additionally consider a positive Radon measure $\mu$ on the unit space $G^{(0)}$. Following  \cite[Definition 3.1 p.~22]{Ren80} and \cite[p.~86]{Pat99}, we can introduce the measure $\nu \ov{\mathrm{def}}{=} \int_{G^{(0)}} \nu^u\d\mu(u)$ induced on $G$ by $\mu$. 
%$$
%\int_{G} f (\xi) \d\nu^{\mu} (\xi)
%\ov{\mathrm{def}}{=} \int_{G^{(0)}}\int_{G}f(\xi) \d\nu^{u}(\xi) \d\mu(u), \quad f \in \C_c(G).
%$$
The measure $\mu$ is said to be quasi-invariant if the measure $\nu$ is equivalent to its image by the inversion map $G \to G$, $s \mapsto s^{-1}$. A measured groupoid $(G,\nu,\mu)$ is a locally compact groupoid equipped with a left Haar system $\nu$ and a quasi-invariant measure\footnote{\thefootnote. Strictly speaking, only the class on the measure is important. But we do not need this point in this paper.} $\mu$. %\cite[p.~23]{Ren80}
%The function $\Delta \ov{\mathrm{def}}{=} \frac{\d\nu^{\mu}}{\d\nu_{\mu}} \co G\rightarrow \R_{+}^{\times}$ is called the modular function of $\mu$. %If this function is continuous, then $\mu$ is said to be \emph{continuous}.

\begin{example} \normalfont
Every locally compact group $G$ can be viewed as a locally compact groupoid, with $G^{(0)} = \{e\}$, multiplication given by the group operation, and inversion by the usual group inverse. We obtain a measured groupoid with a left Haar measure and the Dirac measure as a quasi-invariant measure on $G^{(0)}$. 
\end{example}

\begin{example} \normalfont %Paterson livre p34
Let $X$ be a locally compact space. Set $G \ov{\mathrm{def}}{=} X \times X$ and $G^{(2)}\ov{\mathrm{def}}{=} \{((x,y),(y,z)) : x,y,z \in X\}$. Moreover, for any $x,y,z \in X$ we define $(x,y)(y,z) \ov{\mathrm{def}}{=} (x,z)$ and $(x,y)^{-1} \ov{\mathrm{def}}{=} (y,x)$. We obtain the pair groupoid (or Brandt groupoid). We have $G^{0}=\{(x,x) : x \in X\}$, which can be identified with $X$. Moreover, for any $x,y \in X$ we have $r(x,y)=x$ and $d(x,y)= y$. For any unit $u \in X$, we have $G_u=X \times \{u\}$ and $G^u=\{u\} \times X$.

If we equip $X$ with a positive Radon measure $\mu$, we can define for any unit $u \in X$ the measure $\nu^u \ov{\mathrm{def}}{=} \delta_u \ot \mu$ on $G$, where $\delta_u$ is the unit measure at $u$. In this case, the measure $\mu$ is quasi-invariant. If we consider the discrete space $X=\{1,\ldots,n\}$ for some integer $n \geq 1$, equipped with the counting measure $\mu_n$, we denote by $\mathrm{P}_n$ the associated measured groupoid. 
\end{example}

\begin{example} \normalfont
If $G_1$ and $G_2$ are groupoids, it is clear that the product $G_1 \times G_2$ has a canonical structure of groupoid with $(G_1 \times G_2)^{(2)}=\big\{((x_1,y_1),(x_2,y_2)) : (x_1,x_2) \in G_1^{(2)}, (y_1,y_2) \in G_2^{(2)}\big\}$, 
\begin{equation}
\label{compo-product}
(x,y)^{-1}=(x^{-1},y^{-1}) 
\quad \text{and} \quad 
(x_1,y_1)(x_2,y_2)=(x_1x_2,y_1y_2).
\end{equation}
We have $(G_1 \times G_2)^{(0)}=G_1^{(0)} \times G_2^{(0)}$, $d(x,y)=(d(x),d(y))$ and $r(x,y)=(r(x),r(y))$ 

If $G_1=\mathrm{P}_n$ and $G_2=G$ is a group, an element of the product $\mathrm{P}_n \times G$ can be written under the form $(i,j,s)$ with $i,j \in \{1,\ldots,n\}$ and $s \in G$. We can see a complex function $F \co \mathrm{P}_n \times G \to \mathbb{C}$ as a $n \times n$ matrix-valued function $[F_{ij}]_{1 \leq i,j \leq n}$ on the group $G$, where $F_{ij} \co G \to \mathbb{C}$, $s \mapsto F(i,j,s)$.
\end{example}

\paragraph{von Neumann algebras and multipliers} The von Neumann algebra of a measured groupoid $(G,\lambda,\mu)$ is the von Neumann algebra generated by $\lambda(\C_c(G))$, where $\lambda$ is the regular representation defined in \cite[p.~55]{Ren80} and \cite[pp.~93-94]{Pat99} of the measured groupoid $(G,\nu,\mu)$. If $G=\mathrm{P}_n$ for some integer $n \geq 1$ then it is easy to check that the von Neumann algebra $\VN(G)$ is $*$-isomorphic to the matrix algebra $\M_n$.

\begin{comment}
Let $G$ be a groupoid. A continuous field of separable Hilbert spaces $((H^u)_{u \in G^{0}} ; \Gamma)$ over $G^{(0)}$ is a $G$-Hilbert bundle if for each $s \in G$ there is a unitary isomorphism of Hilbert spaces $L_\gamma \co H^{d(s)} \to H^{r(s)}$ such that
\begin{enumerate}
	\item for any $u \in G^0$ we have $L_u=\Id$
	
	\item If $st$ makes sense then $L_{st}=L_sL_t$
	
	\item For any bounded section $\xi$ of $H=\cup_{u \in G^0} H^u$, the map $G \to H$, $s \mapsto L_s\xi(d(s))$ is continuous.
\end{enumerate}

The regular representation \cite[p.~55]{Ren80} of the measured groupoid $(G,\lambda,\mu)$ is given by the regular $G$-Hilbert bundle $\L^2(G,\lambda)$. Its fiber at $s$ is $\L^2(G,\lambda^s)$. The set $\C_c(G)$ is a family of continuous sections.

For $s \in G$, define  $L(s) \co \L^2(G,\lambda^{d(s)}) \to \L^2(G,\lambda^{r(s)})$, $\xi \mapsto (t \mapsto \xi(s^{-1}t)$.  Its integrated form is the regular representation Reg of $\C_c(G)$.
\begin{equation}
\label{}
\lambda(f)\xi(u)
=\int_{G^u} f(x)L(x)(\xi(d(x)))\Delta^{-\frac{1}{2}}(x) \d \nu^u(x), \quad u \in G^{(0)}
\end{equation}
$\lambda$ is a representation of $\C_c(G)$ on $\L^2(G^{(0)},\L^2(G^u),\mu)$ (Paterson)
\end{comment}

Following \cite[Definition 3.1 p.~475]{Ren97}, we say that a function $\varphi \in \L^\infty(G)$ induces a bounded Fourier multiplier if it induces a weak* continuous\footnote{\thefootnote. In \cite[Proposition 3.1 p.~474]{Ren97}, <<bounded>> must be replaced with <<weak* continuous>> to ensure the correctness of the statement.} operator $\VN(G) \to \VN(G)$, $\lambda(f) \mapsto \lambda(\varphi f)$.

\paragraph{Positive definite functions} Let $(G,\nu,\mu)$ be a measured groupoid. By \cite[Proposition 1.1 p.~457 and Definition 1.1 p.~458]{Ren97}, a function $\varphi \in \L^\infty(G)$ is said to be positive definite if for any integer $n \geq 1$ and any complex numbers $\alpha_1,\ldots ,\alpha_n \in \mathbb{C}$, the inequality
\begin{equation}      
\label{def-pos-def}
\sum_{k,l=1}^n \alpha_{k} \ovl{\alpha_{l}} \varphi(\gamma_{k}^{-1}\gamma_{l})
\geq 0
\end{equation}
holds for $\mu$-almost all $u \in G^{(0)}$ and $\nu^u$-almost all $\gamma_1,\ldots ,\gamma_n \in G^{u}$. We now naturally relate this condition to \cite[Proposition 8.4 p.~166]{ArK23}, where the proof holds for the case of a locally compact group endowed with a trivial cocycle.

\begin{lemma}
\label{Lemma-Bloc-def-pos}
Let $G$ be a locally compact group and let $n \geq 1$ be an integer. The $n \times n$ matrix-valued function $F=[F_{ij}]_{1 \leq i,j \leq n}$ in the space $\L^\infty(\mathrm{P}_n \times G)$ defines a positive definite function on the groupoid $\mathrm{P}_n \times G$ if and only if for any integer $m \geq 1$, any elements $i_1,\ldots, i_m \in \{1,,\ldots,n\}$, any $s_1,\ldots,s_m \in G$ and any complex numbers $\alpha_1,\ldots,\alpha_m \in \mathbb{C}$, we have the inequality
\begin{equation}
\label{Condition-ArK}
\sum_{k,l=1}^{m} \alpha_k \ovl{\alpha_{l}} F_{i_{k}i_{l}}(s_{k}^{-1}s_{l}) 
\geq 0.
\end{equation}
\end{lemma}

\begin{proof}
Note that the unit space of the groupoid $\mathrm{P}_n \times G$ identifies to $(\mathrm{P}_n \times G)^{(0)}=\mathrm{P}_n^{(0)} \times G^{(0)}=\{1,\ldots,n\}$. Fix some $q \in \{1,\ldots,n\}$. If $\gamma_k=(q,i_k,s_k)$ and $\gamma_l=(q,j_l,s_l)$ are elements of the groupoid $\mathrm{P}_n \times G$ with $i_k,j_l \in \{1,\ldots,n\}$ and $s_k,s_l \in G$ then we have 
$$
\gamma_{k}^{-1}\gamma_{l}
=(q,i_k,s_k)^{-1}(q,j_l,s_l)
\ov{\eqref{compo-product}}{=} (i_k,q,s_k^{-1})(q,j_l,s_l)
\ov{\eqref{compo-product}}{=} (i_k,j_l,s_k^{-1}s_l).
$$
So the condition \eqref{def-pos-def} translates to
\begin{equation}      
\sum_{k,l=1}^m \alpha_{k} \ovl{\alpha_{l}} F(i_k,j_l,s_k^{-1}s_l)
\geq 0
\quad
\text{i.e.}     
\sum_{k,l=1}^m \alpha_{k} \ovl{\alpha_{l}} F_{i_k,j_l}(s_k^{-1}s_l) 
\geq 0
\end{equation}
\end{proof}

%\begin{lemma}
%\label{lem:unique inverse}
%If $\mathcal{G}$ is a groupoid and $\gamma \in \mathcal{G}$, then $(r(\gamma),\gamma)$
%and $(\gamma, s(\gamma))$ belong to $\mathcal{G}^{(2)}$, and
%\[
%r(\gamma)\gamma 
%= \gamma = \gamma s(\gamma).
%\]
%We have $r(\gamma^{-1}) = s(\gamma)$ and $s(\gamma^{-1}) = r(\gamma)$. Moreover,
%$\gamma^{-1}$ is the unique element such that $(\gamma,\gamma^{-1}) \in G^{(2)}$ and $\gamma \gamma^{-1} = r(\gamma)$, and also the unique element such that $(\gamma^{-1},\gamma) \in G^{(2)}$ and $\gamma^{-1}\gamma = s(\gamma)$.
%\end{lemma}
%
%We also quickly see that groupoids have cancellation.
%
%\begin{lemma}
%\label{lem:gpd cancellation}
%Let $\mathcal{G}$ be a groupoid. Suppose that $(\alpha,\gamma), (\beta,\gamma) \in
%\mathcal{G}^{(2)}$ and that $\alpha \gamma = \beta\gamma$. Then $\alpha = \beta$.
%Similarly if $(\gamma,\alpha), (\gamma,\beta) \in \mathcal{G}^{(2)}$ and $\gamma\alpha =
%\gamma\beta$ then $\alpha = \beta$.
%\end{lemma}
%
%\begin{lemma}
%Let $G$ be a groupoid. Then $(\alpha,\beta) \in G^{(2)}$ if and only
%if $s(\alpha) = r(\beta)$. We have
%\begin{enumerate}
%\item\label{it:comp r,s} $r(\alpha\beta) = r(\alpha)$ and $s(\alpha\beta) = s(\beta)$
    %for all $(\alpha,\beta) \in \mathcal{G}^{(2)}$;

%\item\label{it:comp inv} $(\alpha\beta)^{-1} = \beta^{-1}\alpha^{-1}$ for all
    %$(\alpha,\beta) \in \mathcal{G}^{(2)}$; and
%\item\label{it:unit r,s} $r(x) = x = s(x)$ for all $x\in \G^{(0)}$.
%\end{enumerate}
%\end{lemma}

\paragraph{A characterization of functions of Fourier-Stieltjes algebras} In the case of groupoids, we caution the reader that there exist three notions of Fourier-Stieltjes algebra, introduced in the papers \cite{Pat04}, \cite{Ren97} and \cite{RaW97}. We refer to the excellent survey \cite{Pat03} for more information. We require a specific case (for groups) of a result, essentially stated in \cite[Proposition 1.3 p.~459 and Lemma 1.1 p.~460]{Ren97} and \cite[Proposition 5 p.~1266]{Pat04} that is more generally stated for Fourier-Stieltjes algebras associated with measured groupoids. Unfortunately, the proof of \cite[Proposition 1.3 p.~459]{Ren97} is false\footnote{\thefootnote. The operator $L'(\gamma)$ of \cite[Proposition 1.3 p.~459]{Ren97} is not a unitary.} (the result \cite[Proposition 5 p.~1266]{Pat04} is incomplete) and must be corrected. Consequently, we provide an argument sufficient for our purposes.

For the proof, we will use the notion of $G$-Hilbert bundle on a locally compact groupoid $G$, which is a Hilbert bundle $\cal{H}$ over its unit space $G^{(0)}$ such that there is a linear unitary operator $\pi_s \co \cal{H}_{d(s)} \to \cal{H}_{r(s)}$ for each $s \in G$ such that for all continuous bounded sections $\xi$ and $\eta$ of $\cal{H}$, the map $(\xi, \eta) \co G \to \cal{B}(\cal{H}_{d(s)},\cal{H}_{r(s)})$, $s \mapsto \big\la \pi_s\xi(d(s)), \eta(r(s)) \big\ra_{\cal{H}_{r(s)}}$ is continuous, and the map $s \mapsto \pi_s$ is a groupoid homomorphism from $G$ into the isomorphism groupoid of the fibered set $\cup_{u \in G^{(0)}} \cal{H}_u$, see \cite[Chapter~1]{Muh90}. Finally, recall that if $\phi \co G \to \mathbb{C}$ is a continuous function then $\phi$ is positive definite if and only if $\phi$ is of the form $(\xi,\xi)$ for some $G$-Hilbert bundle. This result is proved in \cite[Theorem 1 p.~1264]{Pat04}.
%As in [25], we introduce the function $s \mapsto \la \pi_s\xi(d(s)), \eta(r(s)) \ra$ on $G$ by $(\xi, \eta)$, and will call $(\xi, \eta)$ a coefficient of the Hilbert bundle $\cal{H}$.

\begin{prop}
\label{Prop-carac-BG-2-2}
Let $G$ be a locally compact group. A continuous function $\varphi \co G \to \mathbb{C}$ belongs to the Fourier-Stieltjes algebra $\B(G)$ if and only if there exists continuous positive definite functions $\psi_1,\psi_2 \co G \to \mathbb{C}$ such that the matrix $\begin{bmatrix} 
\psi_1 & \varphi \\ 
\check{\ovl{\varphi}} &  \psi_2
\end{bmatrix}$ defines a continuous positive definite function on the measured groupoid $\mathrm{P}_2 \times G$. In this case, we have
\begin{equation}
\label{Norm-B-G-utile}
\norm{\varphi}_{\B(G)}
= \inf \norm{\psi_1}_{\L^\infty(G)}^{\frac{1}{2}} \norm{\psi_2}_{\L^\infty(G)}^{\frac{1}{2}} ,
\end{equation}
where the infimum is taken over all $\psi_1$ and $\psi_2$ satisfying the previous condition.
\end{prop}

\begin{proof}
$\Rightarrow$: Let $\epsi > 0$. Using \eqref{BG-as-entries} and \eqref{Norm-BG}, we can write $\varphi=\la \pi(\cdot)\xi_1,\xi_2 \ra_{H}$ for some vectors $\xi_1$ and $\xi_2$ in a complex Hilbert space $H$ and some continuous unitary representation $\pi$ of $G$ on $H$ with $\norm{\xi_1}_{H} \norm{\xi_2}_{H} \leq \norm{\varphi}_{\B(G)} + \epsi$. For any $s \in G$, we have 
$$
\check{\ovl{\varphi}}(s)
=\ovl{\la \pi(s^{-1})\xi_1,\xi_2 \ra_{H}}
=\la \xi_2, \pi(s)^*\xi_1\ra_{H}
=\la \pi(s)\xi_2, \xi_1\ra_{H}.
$$ 
Hence $\check{\ovl{\varphi}}=\la \pi(\cdot)\xi_2,\xi_1 \ra_{H}$. Now, we introduce the continuous positive definite functions $\psi_1 \ov{\mathrm{def}}{=} \la\pi(\cdot)\xi_1,\xi_1 \ra_H$ and $\psi_2 \ov{\mathrm{def}}{=} \la \pi(\cdot)\xi_2,\xi_2 \ra_H$ on the group $G$. Now, we consider the Hilbert $(\mathrm{P}_2 \times G)$-bundle $\cal{H}$ over the discrete space $(\mathrm{P}_2 \times G)^{(0)}=\{1,2\}$ defined by $\cal{H}_{1} \ov{\mathrm{def}}{=} H$ and $\cal{H}_{2} \ov{\mathrm{def}}{=} H$ and $\pi_{i,j,s} \ov{\mathrm{def}}{=} \pi(s) \co \cal{H}_{j} \to \cal{H}_{i}$ for any $i,j \in \{1,2\}$ and any $s \in G$. For any $i \in \{1,2\}$, we introduce the vector $\zeta(i) \ov{\mathrm{def}}{=} \xi_i$. This defines a section $\zeta$ of the bundle $\cal{H}$. For any $i,j \in \{1,2\}$ and any $s \in G$, we obtain
\begin{align}
\MoveEqLeft
\label{blabla-34}
(\zeta,\zeta)(i,j,s)   
= \big\la \pi_{i,j,s}\zeta(d(i,j,s)), \zeta(r(i,j,s)) \big\ra_{\cal{H}_i}      
=\big\la \pi_{i,j,s}\zeta(j), \zeta(i) \big\ra_{\cal{H}_i} \\
&=\big\la \pi(s)\xi_j, \xi_i \big\ra_{\cal{H}_i} 
%&=\big\la \pi(s)\xi_j, e_i \ot \xi_i \big\ra_{\cal{H}_i} 
%=\big\la \pi(s)\xi_j, e_i \ot \xi_i \big\ra_{\cal{H}_i} 
=\la \pi(s)\xi_j, \xi_i \ra_{H}. \nonumber
\end{align}
We deduce that
\begin{align*}
\MoveEqLeft
\begin{bmatrix} 
\psi_1 & \varphi \\ 
\check{\ovl{\varphi}} &  \psi_2
\end{bmatrix}
=\begin{bmatrix} 
\la\pi(\cdot)\xi_1,\xi_1 \ra_H & \la \pi(\cdot)\xi_1,\xi_2 \ra_{H} \\ 
\la \pi(\cdot)\xi_2,\xi_1 \ra_{H} &  \la\pi(\cdot)\xi_2,\xi_2 \ra_H
\end{bmatrix}         
\ov{\eqref{blabla-34}}{=} (\zeta,\zeta).
\end{align*}
Consequently by \cite[Theorem 1 p.~1264]{Pat04}, the continuous function $\begin{bmatrix} 
\psi_1 & \varphi \\ 
\check{\ovl{\varphi}} &  \psi_2
\end{bmatrix}$ is positive definite on the measured groupoid $\mathrm{P}_2 \times G$. Moreover, we have
\begin{align*}
\MoveEqLeft
\norm{\psi_1}_{\L^\infty(G)}^{\frac{1}{2}} \norm{\psi_2}_{\L^\infty(G)}^{\frac{1}{2}} 
\ov{\eqref{norm-pos-def}}{=} \psi_1(e)^{\frac{1}{2}}\psi_2(e)^{\frac{1}{2}} 
=\la\pi(e) \xi_1 ,\xi_1 \ra_H^{\frac{1}{2}}  \la\pi(e) \xi_2,\xi_2 \ra_H^{\frac{1}{2}} \\
&= \norm{\xi_1}_H \norm{\xi_2}_H       
\leq \norm{\varphi}_{\B(G)} + \epsi.
\end{align*}
Since $\epsi > 0$ is arbitrary, we conclude that $\inf \norm{\psi_1}_{\L^\infty(G)}^{\frac{1}{2}} \norm{\psi_2}_{\L^\infty(G)}^{\frac{1}{2}} \leq \norm{\varphi}_{\B(G)}$.

$\Leftarrow$: Suppose that there exists some continuous positive definite functions $\psi_1,\psi_2 \co G \to \mathbb{C}$ such that the matrix $F \ov{\mathrm{def}}{=} \begin{bmatrix} 
\psi_1 & \varphi \\ 
\check{\ovl{\varphi}} &  \psi_2
\end{bmatrix}$ defines a continuous positive definite function on the measured groupoid $\mathrm{P}_2 \times G$. By \cite[Theorem 1 p.~1264]{Pat04}, there exists a Hilbert $(\mathrm{P}_2 \times G)$-bundle $\cal{H}$ over the discrete space $(\mathrm{P}_2 \times G)^{(0)}=\{1,2\}$, with groupoid homomorphism $(i,j,s) \mapsto \pi_{(i,j,s)}$, and a section $\zeta \co \{1,2\} \to $ of $\cal{H}$ such that 
$\begin{bmatrix} 
\psi_1 & \varphi \\ 
\check{\ovl{\varphi}} &  \psi_2
\end{bmatrix} = (\zeta,\zeta)$. So we have two complex Hilbert spaces $\cal{H}_1$ and $\cal{H}_2$. Here $\pi_{(i,j,s)} \co \cal{H}_j \to \cal{H}_i$ is a unitary operator. Note that $\pi_{(1,1,e)}=\Id_{\cal{H}_1}$ and $\pi_{(2,2,e)}=\Id_{\cal{H}_2}$. We consider the operator 
\begin{equation}
\label{}
P
\ov{\mathrm{def}}{=} \frac{1}{2}\begin{bmatrix} 
\Id_{\cal{H}_1} & \pi_{(1,2,e)} \\ 
\pi_{(2,1,e)} & \Id_{\cal{H}_2}
\end{bmatrix}
\end{equation}
acting on the Hilbert space $\cal{H}_1 \oplus \cal{H}_2$. It is easy to see that $P$ is a selfadjoint projection and that it commutes with each operator $\frac{1}{2}\begin{bmatrix} 
\pi_{(1,1,s)} & \pi_{(1,2,s)} \\ 
\pi_{(2,1,s)} & \pi_{(2,2,s)}
\end{bmatrix}$. We introduce the complex Hilbert space $H \ov{\mathrm{def}}{=} P(\cal{H}_1 \oplus \cal{H}_2)$. Observe that an element $(x,y)$ in $\cal{H}_1 \oplus \cal{H}_2$ belongs to the subspace $H$ if and only if $\pi_{(1,2,e)}(y)=x$ and $\pi_{(2,1,e)}(x)=y$. Consequently, we can consider the operator
$$
\tilde{\pi}_{s}
\ov{\mathrm{def}}{=} \frac{1}{2}\begin{bmatrix} 
\pi_{(1,1,s)} & \pi_{(1,2,s)} \\ 
\pi_{(2,1,s)} & \pi_{(2,2,s)}
\end{bmatrix}|_{H}, \quad s \in G.
$$
It is easy to check\footnote{\thefootnote. For any $s,t \in G$, we have
$$
\bigg(\frac{1}{2}\begin{bmatrix} 
\pi_{(1,1,s)} & \pi_{(1,2,s)} \\ 
\pi_{(2,1,s)} & \pi_{(2,2,s)}
\end{bmatrix}\bigg)
\bigg(\frac{1}{2}\begin{bmatrix} 
\pi_{(1,1,t)} & \pi_{(1,2,t)} \\ 
\pi_{(2,1,t)} & \pi_{(2,2,t)}
\end{bmatrix}\bigg)
=\frac{1}{2}\begin{bmatrix} 
\pi_{(1,1,st)} & \pi_{(1,2,st)} \\ 
\pi_{(2,1,st)} & \pi_{(2,2,st)}
\end{bmatrix}.
$$
Thus we also have
$$
\bigg(\frac{1}{2}\begin{bmatrix} 
\pi_{(1,1,s)} & \pi_{(1,2,s)} \\ 
\pi_{(2,1,s)} & \pi_{(2,2,s)}
\end{bmatrix}\bigg)
\bigg(\frac{1}{2}\begin{bmatrix} 
\pi_{(1,1,s^{-1})} & \pi_{(1,2,s^{-1})} \\ 
\pi_{(2,1,s^{-1})} & \pi_{(2,2,s^{-1})}
\end{bmatrix}\bigg)
=\frac{1}{2}\begin{bmatrix} 
\pi_{(1,1,e)} & \pi_{(1,2,e)} \\ 
\pi_{(2,1,e)} & \pi_{(2,2,e)}
\end{bmatrix}
= P = \Id_H.
$$} that we have a continuous unitary representation of $G$ on the  Hilbert space $H$. Indeed, for any $(x,y) \in H$, we have 
\begin{align*}
\MoveEqLeft
\norm{\tilde{\pi}_{s}(x,y)}_H^2         
=\frac14 \left( \norm{\pi_{1,1,s}(x) + \pi_{1,2,s}(y)}_{\cal{H}_1}^2 + \norm{\pi_{2,1,s}(x) + \pi_{2,2,s}(y)}_{\cal{H}_2}^2 \right) \\
&=\frac14 \left( \norm{\pi_{1,1,s}\big( x+ \pi_{1,2,e}(y))}_{\cal{H}_1}^2 + \norm{\pi_{2,2,s}(\pi_{2,1,e}(x)+ y)}_{\cal{H}_2}^2 \right)\\
&=\frac14 \left( \norm{ x+ \pi_{1,2,e}(y)}_{\cal{H}_1}^2 + \norm{\pi_{2,1,e}(x)+ y}_{\cal{H}_2}^2 \right) 
=\frac14 \left( \norm{2x}_{\cal{H}_1}^2 + \norm{2y}_{\cal{H}_2}^2 \right) \\
&= \norm{x}_{\cal{H}_1}^2 +\norm{y}_{\cal{H}_2}^2 
=\norm{(x,y)}_H^2.
\end{align*}
We consider the vectors $\xi \ov{\mathrm{def}}{=} \sqrt{2} P(0,\zeta(2))$ and $\eta \ov{\mathrm{def}}{=} \sqrt{2} P(\zeta(1),0)$ in the space $H$. Now, for any $s \in G$ we observe that
\begin{align*}
\MoveEqLeft
\varphi(s)
=(\zeta,\zeta)(1,2,s)
=\big\la \pi_{(1,2,s)}\zeta(d(1,2,s)), \zeta(r(1,2,s)) \big\ra \\
&=\big\la \pi_{(1,2,s)}\zeta(2), \zeta(1) \big\ra \\
%&=\bigg\la\frac{1}{2}\begin{bmatrix} 
%\pi_{(1,1,s)} & \pi_{(1,2,s)} \\ 
%\pi_{(2,1,s)} & \pi_{(2,2,s)}
%\end{bmatrix}(0,\zeta(2)),(\zeta(1),0) \bigg\ra\\
&=\bigg\la\frac{1}{2}\frac{1}{2}\begin{bmatrix} 
\Id_{\cal{H}_1} & \pi_{(1,2,e)} \\ 
\pi_{(2,1,e)} & \Id_{\cal{H}_2}
\end{bmatrix}
\begin{bmatrix} 
\pi_{(1,1,s)} & \pi_{(1,2,s)} \\ 
\pi_{(2,1,s)} & \pi_{(2,2,s)}
\end{bmatrix}(0,\zeta(2)),(\zeta(1),0) \bigg\ra\\
&=2\la \tilde{\pi}_{s}P(0,\zeta(2)),P(\zeta(1),0) \ra_H
= \la \tilde{\pi}_s \sqrt{2} P(0,\zeta(2)),\sqrt{2}P(\zeta(1),0) \ra_H
= \la \tilde{\pi}_{s}\xi , \eta \ra_H.         
\end{align*}
Hence the function $\varphi$ belongs to the Fourier-Stieltjes algebra $\B(G)$. Furthermore, if $i \in \{1,2\}$ we have
\begin{align}
\MoveEqLeft
\label{inter-989}
\norm{\psi_i}_{\L^\infty(G)}
\ov{\eqref{norm-pos-def}}{=} \psi_i(e)
=(\zeta,\zeta)(i,i,e) 
=\big\la \pi_{(i,i,e)}\zeta(d(i,i,e)), \zeta(r(i,i,e)) \big\ra \\
&=\big\la \pi_{(i,i,e)}\zeta(i), \zeta(i) \big\ra
= \norm{\zeta(i)}_{\cal{H}_i}^2.   \nonumber      
\end{align}
Next, observe that 
\begin{align*}
\MoveEqLeft
\norm{\xi}_H^2 = \sqrt{2}^2 \norm{P(0,\zeta(2))}_H^2 = 2 \left( \frac14 \norm{\pi_{1,2,e}(\zeta(2))}_{\cal{H}_1}^2 + \frac14 \norm{\zeta(2)}_{\cal{H}_2}^2 \right) \\
& = \frac12 \left(\norm{\zeta(2)}_{\cal{H}_2}^2 + \norm{\zeta(2)}_{\cal{H}_2}^2\right) = \norm{\zeta(2)}_{\cal{H}_2}^2.
\end{align*}
In the same way, we have $\norm{\eta}_H^2 = \norm{\zeta(1)}_{\cal{H}_1}^2$.
Moreover, we have 
$$
\norm{\varphi}_{\B(G)}
\ov{\eqref{Norm-BG}}{\leq} \norm{\xi}_{H} \norm{\eta}_{H}
= \norm{\zeta(2)}_{\cal{H}_2} \norm{\zeta(1)}_{\cal{H}_1}
\ov{\eqref{inter-989}}{=} \norm{\psi_1}_{\L^\infty(G)}^{\frac{1}{2}} \norm{\psi_2}_{\L^\infty(G)}^{\frac{1}{2}}.
$$
\end{proof}

%%%%%%%%%%%%%%%%%%%%%%%%%%%%%%%%%%%%%%%%%%%%%%%%%%%%%%%%%%%%%%%%%%%%%%%%%%%%%%%%%%%%%%%%%%%%%
\subsection{Links between Fourier-Stieltjes algebras and decomposable multipliers}
\label{Sec-FS-and-dec}

Let $G$ be a locally compact group. Recall that by \cite[Corollary 1.8 (i) p.~465]{DCH85} or \cite[Corollary 5.4.11 p.~185]{KaL18} we have a contractive inclusion $\B(G) \subseteq \frak{M}^{\infty,\cb}(G)$. This is even a complete contraction by \cite[Corollary 4.3 p.~179]{Spr04}, where we equip the Fourier-Stieltjes algebra $\B(G)$ with the dual operator space structure induced by the equality $\mathrm{C}^*(G)^*=\B(G)$. In the next result, we strengthen this result by replacing the space $\frak{M}^{\infty,\cb}(G)$ of completely bounded Fourier multipliers on the von Neumann algebra $\VN(G)$ by the space $\frak{M}^{\infty,\dec}(G)$ of decomposable Fourier multipliers on $\VN(G)$. While the inclusion $\B(G) \subseteq \frak{M}^{\infty,\dec}(G)$ is straightforward\footnote{\thefootnote. Indeed, if $\varphi \in \B(G)$ then we can write $\varphi=\varphi_1-\varphi_2+\i \varphi_3-\varphi_4$, where each $\varphi_i$ is a continuous positive definite function. By \cite[Proposition 5.4.9 p.~184]{KaL18}, each Fourier multiplier $M_{\varphi_i} \co \VN(G)\to \VN(G)$ is completely positive. Then it is immediate that the Fourier multiplier
$$
M_\varphi
=M_{\varphi_1-\varphi_2+\i(\varphi_3-\varphi_4)}
=M_{\varphi_1}-M_{\varphi_2}+\i(M_{\varphi_3}-M_{\varphi_4})
$$
is decomposable.}, the \textit{contractivity} of the inclusion $\B(G) \subseteq \frak{M}^{\infty,\dec}(G)$ is new, even in the case where $G$ is discrete. Here, we equip the space $\frak{M}^{\infty,\dec}(G)$ with the operator space structure induced by the one of the operator space $\Dec(\VN(G))$. For the proof, we will use the notion of a quasi-complete locally convex space. Recall that a locally convex space $X$ is called quasi-complete if every bounded Cauchy net in $X$ converges \cite[Definition 4.23 p.~107]{Osb14}.

\begin{prop}
\label{prop-B(G)-inclus-dec}
Let $G$ be a locally compact group. The map $\B(G) \to \frak{M}^{\infty,\dec}(G)$, $\varphi \mapsto M_\varphi$ is a well-defined injective complete contraction from the Fourier-Stieltjes algebra $\B(G)$ into the space $\frak{M}^{\infty,\dec}(G)$ of decomposable Fourier multipliers.
\end{prop}

\begin{proof}
We will present two distinct proofs.

\noindent\textit{First proof.} We begin with a purely group-theoretic argument. Let $\varphi \in \B(G)$. By homogeneity, we can suppose that $\norm{\varphi}_{\B(G)}=1$. We will use the associated linear form $\omega_\varphi \co \C^*(G) \to \mathbb{C}$, $\int_G f(s) U_s \d\mu_G(s) \mapsto \int_G \varphi(s)f(s) \d\mu_G(s)$ defined in \eqref{Def-omega-varphi}. By \cite[Lemma A.2.2 p.~360]{BlM04}, we can consider the unique weak* continuous extension $\tilde{\omega}_\varphi \co \W^*(G) \to \mathbb{C}$ on the von Neumann algebra $\W^*(G)$, where we use here (and only in this step) the identification $\W^*(G)=\C^*(G)^{**}$ of \cite[p.~265]{Dix77}. We will prove that for any $s \in G$ the element $U_s$ belongs to the von Neumann algebra $\W^*(G)$ and the equality
\begin{equation}
\label{magic-equality-1}
\tilde{\omega}_\varphi(U(s))
=\varphi(s), \quad s \in G.
\end{equation}
Let $s \in G$ and let $\mathfrak{B}$ be a neighbourhood basis at $s$ constituted of compact neighbourhoods. For any $V \in \mathfrak{B}$, consider a positive continuous function $f_V \co G \to \R^+$ on $G$ such that $\int_G f_V \d \mu_G=1$ with support contained in $V$. Then by \cite[ Corollary 3, VIII.17]{Bou04b} the net $(\int_G f_V(t)U_t \d\mu_G(t))$ converges to $U_s$ in the strong operator topology, and therefore also in the weak operator topology. Moreover, by \cite[VIII.15]{Bou04b} for any $V$ we have the estimate
$$
\norm{\int_G f_V(t)U_t \d\mu_G(t)} 
\leq \int_G f_V \d\mu_G
=1.
$$
So the net $(\int_G f_V(t)U_t \d\mu_G(t))_V$ is bounded. We deduce that the net $(\int_G f_V(t)U_t \d\mu_G(t))_V$ converges to $U_s$ in the weak* topology by \cite[Lemma 2.5 p.~69]{Tak02}. In particular, we deduce that $U_s$ belongs to the von Neumann algebra $\W^*(G)$. Moreover, on the one hand, we infer by weak* continuity of the linear form $\tilde{\omega}_\varphi$ that the net $(\tilde{\omega}_\varphi(\int_G f_V(t)U_t \d\mu_G(t))_V)$ converges to $\tilde{\omega}_\varphi(U_s)$. On the other hand, using the continuity of the function $\varphi \co G \to \mathbb{C}$ and \cite[Corollary 2, VIII.17]{Bou04b} in the limit process, we obtain
$$
\tilde{\omega}_\varphi\bigg(\int_G f_V(t)U_t \d\mu_G(t)\bigg)
=\omega_\varphi\bigg(\int_G f_V(t)U_t \d\mu_G(t)\bigg)
\ov{\eqref{Def-omega-varphi}}{=} \int_G \varphi(t)f_V(t) \d\mu_G(t)
\xra[V]{} \varphi(s).
$$
By uniqueness of the limit, we conclude that \eqref{magic-equality-1} is true.

By Fell's absorption principle \eqref{Fell} applied to the representation $U \co G \to \mathcal{B}(H)$ instead of $\pi$, there exists a unitary $W \co \L^2(G,H) \to \L^2(G,H)$ such that for any $s \in G$
$$
W(\lambda_s \ot \Id_H)W^*
=\lambda_s \ot U_s.
$$
By \cite[p.~9 and p.~25]{Dix81}, we deduce that there exists a normal unital $*$-homomorphism $\Delta \co \VN(G) \to \VN(G) \otvn \W^*(G)$, $\lambda_s \mapsto \lambda_s \ot U_s$. Since any Banach space is barreled \cite[Theorem 4.5 p.~97]{Osb14}, we see by \cite[Corollary 4.25 (b) p.~107]{Osb14} that the weak* topology on the dual Banach space $\VN(G)$ is quasi-complete. Consequently, by \cite[Corollary 2, III p.~38]{Bou04a}, for any function $f \in \C_c(G)$, the integral $\int_G f(s)\lambda_s \d \mu_G(s)$ is a well-defined weak* integral. Using the weak* continuity of $\Delta$ together with \cite[Proposition 1, VI.3]{Bou04a} in the first equality, we deduce that
\begin{align}
\MoveEqLeft
\label{Eq-1356}
\Delta\bigg(\int_G f(s)\lambda_s \d \mu_G(s)\bigg)
=\int_G f(s)\Delta(\lambda_s) \d \mu_G(s)
=\int_G f(s) (\lambda_s \ot U_s) \d \mu_G(s). 
\end{align}  
Now, for any $s \in G$, we obtain again with \cite[Proposition 1, VI.3]{Bou04b} that
\begin{align*}
\MoveEqLeft
(\Id \ot \tilde{\omega}_\varphi) \circ \Delta\bigg(\int_G f(s)\lambda_s \d \mu_G(s)\bigg)         
\ov{\eqref{Eq-1356}}{=} (\Id \ot \tilde{\omega}_\varphi)\bigg(\int_G f(s) (\lambda_s \ot U_s) \d \mu_G(s)\bigg)\\
&=\bigg(\int_G f(s) (\Id \ot \tilde{\omega}_\varphi)(\lambda_s \ot U_s) \d \mu_G(s)\bigg)
\ov{\eqref{magic-equality-1}}{=}\int_G \varphi(s)f(s)\lambda_s \d \mu_G(s).
\end{align*} 
We conclude that the weak* continuous map $(\Id \ot \tilde{\omega}_\varphi) \circ \Delta$ is the Fourier multiplier $M_\varphi$ of symbol $\varphi$. 

Note that the $*$-homomorphism $\Delta$ is decomposable since it is completely positive. According to \cite[Lemma 5.4.3 p.~96]{EfR00}, the linear form $\tilde{\omega}_\varphi$ is equally decomposable with 
\begin{equation}
\label{inter-998}
\norm{\tilde{\omega}_\varphi}_{\dec}
\ov{\eqref{dec=cb}}{=} \norm{\tilde{\omega}_\varphi}_{\cb}
=\norm{\tilde{\omega}_\varphi}
=1,
\end{equation}
where we use \cite[Corollary 2.2.3 p.~24]{EfR00} in the second equality. By Lemma \ref{Lemma-tensor-dec-2}, we deduce that we have a well-defined weak* continuous decomposable map $\Id \ot \tilde{\omega}_\varphi \co \VN(G) \otvn \W^*(G) \to \VN(G)$. We conclude by composition that the linear map $M_\varphi=(\Id \ot \tilde{\omega}_\varphi) \circ \Delta \co \VN(G) \to \VN(G)$ is decomposable and that
\begin{align*} 
\MoveEqLeft
\norm{M_\varphi}_{\dec}
=\bnorm{(\Id \ot \tilde{\omega}_\varphi) \circ \Delta}_{\dec}
\ov{\eqref{Composition-dec}}{\leq} \norm{\Id \ot \tilde{\omega}_\varphi}_\dec \norm{\Delta}_\dec \\
&\ov{\eqref{dec-tensor-2}}{\leq} \norm{\tilde{\omega}_\varphi}_\dec \norm{\Delta}_\dec 
\ov{\eqref{inter-998}}{=} \norm{\Delta}_\dec
\ov{\eqref{dec-et-cp}}{\leq} 1.
\end{align*}
%Using a result of \cite[page 40]{BlM04}, we see by composition that the map $(\Id \ot \varpi_\xi) \circ \Delta \co \VN(G) \to \VN(G)$ is weak* continuous. 
Finally, it is easy to check that the map $\B(G) \to \frak{M}^{\infty,\dec}(G)$, $\varphi \mapsto M_\varphi$ is also injective. For the complete contractivity, the argument is similar. Consider a matrix $[\varphi_{ij}] \in \M_{n}(\B(G))$. We have a completely bounded map $[\omega_{\varphi_{ij}}] \co \C^*(G) \to \M_n$. Using \cite[1.4.8 p.~24]{BlM04}, its unique weak* continuous extension $[\tilde{\omega}_{\varphi_{ij}}] \co \W^*(G) \to \M_n$ is completely bounded with the same completely bounded norm. Note that this linear map is decomposable and its decomposable norm coincides with its completely bounded norm by \eqref{dec=cb}. Finally, we can write $[M_{\varphi_{ij}}]=(\Id \ot [\tilde{\omega}_{\varphi_{ij}}]) \circ (\Id_{\M_n} \ot \Delta)$.%, using the injectivity of the regular representation of $\L^1(G)$

\noindent\textit{Second proof of the contractivity.} Now, we give a second proof using groupoids. Let $\varphi \in \B(G)$ and $\epsi >0$. By Proposition \ref{Prop-carac-BG-2-2}, there exists continuous positive definite functions $\psi_1$ and $\psi_2$ (hence bounded by \cite[Proposition C.4.2 p.~351]{BHV08}) such that the matrix $\begin{bmatrix} 
\psi_1 & \varphi \\ 
\check{\ovl{\varphi}} & \psi_2
\end{bmatrix}$ defines a continuous positive definite function $F$ on the groupoid $\mathrm{P}_2\times G$ with
\begin{equation}
\label{ine-epsi}
\norm{\psi_1}_{\L^\infty(G)}^{\frac{1}{2}} \norm{\psi_2}_{\L^\infty(G)}^{\frac{1}{2}}
\leq \norm{\varphi}_{\B(G)}+\epsi.
\end{equation}
Note the identification $\VN(\mathrm{P}_2 \times G)=\VN(\mathrm{P}_2) \otvn \VN(G)=\M_2 \otvn \VN(G)=\M_2(\VN(G))$. By generalizing the very transparent argument of \cite[Proposition 5.6.16 p.~206]{BrO08} with \cite[Proposition 4.12 and Remark 4.12]{Arh24}, we see that $F$ induces a completely positive multiplier on the von Neumann algebra $\VN( \mathrm{P}_2 \times G)$. This completely positive multiplier identifies to the map $\begin{bmatrix} 
M_{\psi_1} & M_\varphi \\ 
M_{\check{\ovl{\varphi}}} &  M_{\psi_2}
\end{bmatrix}=\begin{bmatrix} 
M_{\psi_1} & M_\varphi \\ 
M_\varphi^\circ &  M_{\psi_2}
\end{bmatrix} \co \M_2(\VN(G)) \to \M_2(\VN(G))$. Note that the Fourier multipliers $M_{\psi_1}$ and $M_{\psi_2}$ are completely positive. We conclude that the Fourier multiplier $M_\varphi \co \VN(G) \to \VN(G)$ is decomposable with
\begin{align*}
\MoveEqLeft
\norm{M_\varphi}_{\dec,\VN(G) \to \VN(G)}
\ov{\eqref{Norm-dec-sqrt}}{\leq} \norm{M_{\psi_1}}_{\VN(G) \to \VN(G)}^{\frac{1}{2}} \norm{M_{\psi_2}}_{\VN(G) \to \VN(G)}^{\frac{1}{2}} \\
&=\norm{\psi_1}_{\L^\infty(G)}^{\frac{1}{2}} \norm{\psi_2}_{\L^\infty(G)}^{\frac{1}{2}}
\ov{\eqref{ine-epsi}}{\leq} \norm{\varphi}_{\B(G)}+\epsi.            
\end{align*} 
\end{proof}

Now, we study the converse of Proposition \ref{prop-B(G)-inclus-dec} in Proposition \ref{conj-1-1-correspondance} and in Theorem \ref{dec-vs-B(G)-discrete-group}. We need the following result, which gives a description of the norm of the Fourier-Stieltjes algebra $\B(G)$ for some suitable functions. Here, we denote by $\mathrm{P}(G)$ the set of continuous positive definite functions on $G$, following \cite[Definition 1.4.18 p.~23]{KaL18}. 

\begin{prop}
Let $G$ be a locally compact group. Let $\varphi \in \B(G)$ such that $\check{\varphi}=\ovl{\varphi}$. We have
\begin{equation}
\label{norm-B(G)}
\norm{\varphi}_{\B(G)}
=\inf \big\{\varphi_1(e)+\varphi_2(e): \varphi=\varphi_1-\varphi_2, \varphi_1,\varphi_2 \in \mathrm{P}(G) \big\}.
\end{equation}
\end{prop}

\begin{proof}
Suppose that $\varphi=\varphi_1-\varphi_2$ for some continuous positive definite functions $\varphi_1,\varphi_2 \co G \to \mathbb{C}$. We have
$$
\norm{\varphi}_{\B(G)}
=\norm{\varphi_1-\varphi_2}_{\B(G)}
\leq \norm{\varphi_1}_{\B(G)} + \norm{\varphi_2}_{\B(G)}
\ov{\eqref{norm-pos-def}}{=} \varphi_1(e)+\varphi_2(e).
$$
Passing to the infimum, we obtain that $\norm{\varphi}_{\B(G)} \leq \inf \big\{\varphi_1(e)+\varphi_2(e): \varphi=\varphi_1-\varphi_2, \varphi_1,\varphi_2 \in \mathrm{P}(G) \big\}$. Indeed, by \cite[(2.7) p.~193]{Eym64} (or \cite[p.~41]{KaL18}\footnote{\thefootnote. Note that in this reference, the assumption ``$\check{u}=\ovl{u}$'' is missing.}) we have an equality in this last inequality and the infimum is a minimum. 
\end{proof}

\begin{remark} \normalfont
Suppose that the locally compact $G$ is \textit{abelian}. For any real bounded regular Borel measure $\mu$ on the dual group $\hat{G}$, the previous result combined with Example \ref{norm-B(G)-commutatif} and \eqref{norm-pos-def} implies that
\begin{equation}
\label{norm-M(G)}
\norm{\mu}_{\M(\hat{G})}
=\inf \big\{\norm{\mu_1}_{\M(\hat{G})}+\norm{\mu_2}_{\M(\hat{G})}: \mu=\mu_1-\mu_2, \mu_1,\mu_2 \geq 0 \big\}.
\end{equation}
We can replace the group $\hat{G}$ by a locally compact space $X$. Indeed, for any \textit{real} bounded regular Borel measure $\mu$, we can decompose the measure $\mu$ with \cite[III \S1.~8 Corollary 2]{Bou04a} as $\mu=\mu^+ -\mu^-$ for some bounded positive regular Borel measures $\mu_+$ and $\mu_-$ on $X$ with $\norm{\mu}_{\M(X)}=\norm{\mu^+}_{\M(X)}+\norm{\mu^-}_{\M(X)}$ and we can use a similar reasoning.
%\begin{quest}
%A-t-on une formule similaire ou quel est le substitut pour une mesure complexe ?
%\end{quest}
\end{remark}

By adding property $(\kappa_\infty)$ to the group $G$, we obtain a partial converse to Proposition \ref{prop-B(G)-inclus-dec}. For the proof, we will use the folklore fact that says that the symbol of any bounded multiplier on the von Neumann algebra $\VN(G)$ of a locally compact group $G$ is almost everywhere equal to a continuous function. This follows from the <<regularity>> of the Fourier algebra, established in \cite[Theorem 2.3.8 p.~53]{KaL18}.

\begin{prop}
\label{conj-1-1-correspondance}
Let $G$ be a locally compact group. If $G$ has property $(\kappa_\infty)$, then the linear map $\B(G) \to \frak{M}^{\infty,\dec}(G)$, $\varphi \mapsto M_\varphi$ is a bijection from the Fourier-Stieltjes algebra $\B(G)$ onto the space $\frak{M}^{\infty,\dec}(G)$ of decomposable multipliers. Moreover, if $\kappa_\infty(G)=1$ and if the function $\varphi$ belongs to $\B(G)$ and satisfies $\check{\varphi}=\ovl{\varphi}$, we have $\norm{\varphi}_{\B(G)}=\norm{M_\varphi}_{\dec,\VN(G) \to \VN(G)}$.
\end{prop}

\begin{proof}
In Proposition \ref{prop-B(G)-inclus-dec}, we established a (completely) contractive inclusion $\B(G) \subseteq \frak{M}^{\infty,\dec}(G)$. We show the reverse inclusion. Suppose that $M_\varphi \co \VN(G) \to \VN(G)$ is a decomposable Fourier multiplier (hence weak* continuous) with continuous symbol $\varphi \co G \to \mathbb{C}$. We can write 
\begin{equation}
\label{in-543}
M_\varphi
=T_1 + T_2 + \i(T_3 - T_4)
\end{equation}
for some completely positive maps $T_1,T_2,T_3,T_4 \co \VN(G) \to \VN(G)$. By using the contractive projection $P_{\w^*} \co \cal{B}(\VN(G)) \to \cal{B}(\VN(G))$ of \cite[Proposition 3.1 p.~24]{ArK23}, which preserves the complete positivity, as in the proof of \cite[Proposition 3.4 p.~26]{ArK23}, we can suppose that these maps $T_1,T_2,T_3,T_4$ are weak* continuous since $P_{\w^*}(M_\varphi)=M_\varphi$. Using the bounded projection $P_{G}^\infty \co \CB_{\w^*}(\VN(G)) \to \CB_{\w^*}(\VN(G))$ provided by property $(\kappa_\infty)$, we obtain
$$
M_\varphi
=P_{G}^\infty(M_\varphi)
\ov{\eqref{in-543}}{=} P_{G}^\infty\big(T_1-T_2+\i(T_3-T_4)\big)
=P_{G}^\infty(T_1)-P_{G}^\infty(T_2)+\i(P_{G}^\infty(T_3)-P_{G}^\infty(T_4)),
$$
where each $P_{G}^\infty(T_i) \co \VN(G) \to \VN(G)$ is a completely positive Fourier multiplier for some symbol $\varphi_i \co G \to \mathbb{C}$, i.e.~$P_{G}^\infty(T_i)=M_{\varphi_i}$. By \cite[Proposition 4.2 p.~487]{DCH85}, the function $\varphi_i$ is continuous and positive definite. We deduce that
$$
M_\varphi
=M_{\varphi_1}-M_{\varphi_2}+\i(M_{\varphi_3}-M_{\varphi_4})
=M_{\varphi_1-\varphi_2+\i(\varphi_3-\varphi_4)}.
$$
%Recall that the map $\B(G) \to \frak{M}^{\infty}(G)$, $\psi \mapsto M_\psi$ is injective. 
We infer that $\varphi=\varphi_1-\varphi_2+\i \varphi_3-\varphi_4$. We conclude that the function $\varphi$ belongs to the Fourier-Stieltjes algebra $\B(G)$. Hence we have an inclusion $\frak{M}^{\infty,\dec}(G) \subset \B(G)$.

Now, we prove the second part of the statement assuming $\kappa_\infty(G)=1$. Suppose that the function $\varphi$ belongs to the Fourier-Stieltjes algebra $\B(G)$ and satisfies $\check{\varphi}=\ovl{\varphi}$. This last condition means that the Fourier multiplier $M_\varphi \co \VN(G) \to \VN(G)$ is adjoint preserving, i.e.~$M_\varphi(x^*)=(M_\varphi(x))^*$ for any $x \in \VN(G)$. Let $\epsi > 0$. By \cite[Proposition 1.3 (1) p.~177]{Haa85} and \cite[p.~184]{Haa85}, there exists some completely positive operators $T_1,T_2 \co \VN(G) \to \VN(G)$ such that 
\begin{equation}
\label{eqa-456}
M_\varphi=T_1-T_2
\quad \text{with} \quad 
\norm{T_1+T_2} = \norm{M_\varphi}_{\dec,\VN(G) \to \VN(G)}.
\end{equation}
By using the contractive projection $P_{\w^*} \co \cal{B}(\VN(G)) \to \cal{B}(\VN(G))$ of \cite[Proposition 3.1 p.~24]{ArK23}, which preserves the complete positivity, it is easy to check that we can suppose that the linear maps $T_1$ and $T_2$ are weak* continuous since $P_{\w^*}(M_\varphi)=M_\varphi$. Let $\epsi>0$. Since $\kappa_\infty(G)=1$, we can consider a bounded projection $P_{G}^\infty \co \CB_{\w^*}(\VN(G)) \to \CB_{\w^*}(\VN(G))$ of norm $\leq 1+\epsi$, preserving the complete positivity. We deduce that
\begin{equation}
\label{Decompo-magic}
M_\varphi
=P_{G}^\infty(M_\varphi)
\ov{\eqref{eqa-456}}{=} P_{G}^\infty(T_1-T_2)
=P_{G}^\infty(T_1)-P_{G}^\infty(T_2).
\end{equation}
We denote by $\varphi_1$ and $\varphi_2$ the continuous symbols of the completely positive Fourier multipliers $P_{G}^\infty(T_1)$ and $P_{G}^\infty(T_2)$. These functions are positive definite again by \cite[Proposition 4.2 p.~487]{DCH85}. The equality \eqref{Decompo-magic} gives $\varphi=\varphi_1-\varphi_2$. By using \cite[Proposition 4.3 p.~489]{DCH85} in the second equality, we see that
\begin{align*}
\MoveEqLeft
\norm{\varphi}_{\B(G)}
\ov{\eqref{norm-B(G)}}{\leq} \varphi_1(e)+\varphi_2(e)
=(\varphi_1+\varphi_2)(e)
=\norm{M_{\varphi_1+\varphi_2}}_{\VN(G) \to \VN(G)} \\
&=\norm{M_{\varphi_1} + M_{\varphi_2}}_{\VN(G) \to \VN(G)} 
=\norm{P_{G}^\infty(T_1)+P_{G}^\infty(T_2)} 
= \norm{P_{G}^\infty(T_1+T_2)} \\
&\leq (1+\epsi)\norm{T_1+T_2}
\ov{\eqref{eqa-456}}{=} (1+\epsi) \norm{M_\varphi}_{\dec,\VN(G) \to \VN(G)}.
\end{align*}
Since $\epsi>0$ is arbitrary, we deduce that $\norm{\varphi}_{\B(G)} \leq  \norm{M_\varphi}_{\dec,\VN(G) \to \VN(G)}$. Combining with Proposition \ref{prop-B(G)-inclus-dec}, we conclude that $\norm{\varphi}_{\B(G)}=\norm{M_\varphi}_{\dec,\VN(G) \to \VN(G)}$.  
\end{proof}

%Let $G$ be a locally compact group. By Proposition \ref{prop-B(G)-inclus-dec}, we can insert the space $\frak{M}^{\infty,\dec}(G)$ of decomposable Fourier multipliers on $\VN(G)$ in the classical contractive inclusions $\B(G) \subseteq \frak{M}^{\infty,\cb}(G) \subseteq \frak{M}^{\infty}(G)$:
%\begin{equation}
%\label{Inclusions}
%\B(G) 
%\subseteq \frak{M}^{\infty,\dec}(G)
%\subseteq \frak{M}^{\infty,\cb}(G) 
%\subseteq \frak{M}^{\infty}(G).
%\end{equation}

%For any non-amenable weakly amenable discrete group $G$, the second inclusion is strict by \cite[Proposition 3.32 (1)]{ArK23}. It is well-known that the norm and the completely norms of Fourier multipliers are different in general, see \cite[Cor. 4.9]{DCH85}. 

Now, we observe that the first inclusion in \eqref{Inclusions} can be strict.

\begin{prop}
\label{prop-groups-with-bad-multiplier}
Let $G$ be a non-amenable locally compact group such that the von Neumann algebra $\VN(G)$ is injective. Then there exists a decomposable Fourier multiplier $T \co \VN(G) \to \VN(G)$, which is not induced by an element $\varphi \in \B(G)$.
\end{prop}

\begin{proof}
Since the von Neumann algebra $\VN(G)$ is injective, we have by \cite[Theorem 1.6 p.~184]{Haa85} the equality $\frak{M}^{\infty,\cb}(G)=\frak{M}^{\infty,\dec}(G)$ isometrically. Since the group $G$ is not amenable, we know by an unpublished result of Ruan stated in \cite[p.~54]{Pis01} and \cite[p.~190]{Spr04} that $\B(G) \varsubsetneq \frak{M}^{\infty,\cb}(G)$. We conclude that $\B(G) \varsubsetneq \frak{M}^{\infty,\dec}(G)$.
%Suppose that $\VN(G)$ is approximately finite-dimensional and that $G$ has $(\widetilde{\kappa})$. In this case, Proposition \ref{prop-decomposer-mult} gives 
%$$
%\mathfrak{M}^{\infty,\dec}(G)
%=\mathrm{span} \ \ \ \textrm{cp multipliers}.
%$$
%We deduce that
%$$
%\mathfrak{M}^{\infty,\cb}(G)
%=\mathrm{span} \ \ \ \textrm{cp multipliers}.
%$$
%Using Theorem \ref{Th-Bozejko}, we conclude that $G$ is amenable
%
%Since 	Suppose that each decomposable Fourier multiplier $M_{\varphi} \co \VN(G) \to \VN(G)$ is induced by a normal $\omega$. We can write with obvious notations
	%\begin{align*}
	%\MoveEqLeft
	%M_\varphi
	%=\big(\omega \ot \Id_{\VN(G)}\big) \circ \Delta
	%=\big((\omega_1-\omega_2+\i(\omega_3-\omega_4)) \ot \Id_{\VN(G)}\big) \circ \Delta\\
	%&=\big((\omega_1 \ot \Id_{\VN(G)})-(\omega_2 \ot \Id_{\VN(G)})+\i((\omega_3 \ot \Id_{\VN(G)})-(\omega_4 \ot \Id_{\VN(G)}))\big) \circ \Delta\\
	%&=M_{\phi_1}-M_{\phi_2}+\i\big(M_{\phi_3}-M_{\phi_4}\big).   
	%\end{align*}
	%We infer that 
	%$$
	%\mathfrak{M}^{\infty,\dec}(G) 
	%=\mathrm{span} \ \ \ \textrm{cp multipliers}.
	%$$
	%By \cite[Corollary 2.8]{Haa85} we have $\norm{\cdot}_{\dec,\VN(G) \to \VN(G)}=\norm{\cdot}_{\cb,\VN(G) \to \VN(G)}$. We deduce that
	%$$
	%\mathfrak{M}^{\infty,\cb}(G)
	%=\mathrm{span} \ \ \ \textrm{cp multipliers}.
	%$$
	%
\end{proof}

\begin{example} \normalfont
\label{Example-SL2}
By \cite[Corollary 7 p.~75]{Con76}, the von Neumann algebra $\VN(G)$ of a second-countable connected locally compact group $G$ is injective. This result applies for example to the locally compact group $G=\SL_2(\R)$, which is non-amenable by \cite[Example G.2.4 (i) p.~426]{BHV08}. We conclude that $\B(G) \varsubsetneq \frak{M}^{\infty,\dec}(G)$ in this case.
\end{example}

For discrete groups, a matricial improvement of property $(\kappa_\infty)$ is available in \cite[Theorem 4.2 p.~62]{ArK23}. Consequently, we can establish the following isometric result.

\begin{thm}
\label{dec-vs-B(G)-discrete-group}
Let $G$ be a discrete group. The map $\B(G) \to \frak{M}^{\infty,\dec}(G)$, $\varphi \mapsto M_\varphi$ is an isometric isomorphism from the Fourier-Stieltjes $\B(G)$ onto the algebra $\frak{M}^{\infty,\dec}(G)$ of decomposable multipliers on the von Neumann algebra $\VN(G)$.
\end{thm}

\begin{proof}
In Proposition \ref{prop-B(G)-inclus-dec}, we have seen that we have a contractive inclusion $\B(G) \subseteq \frak{M}^{\infty,\dec}(G)$. It suffices to show the reverse inclusion. Suppose that the Fourier multiplier $M_\varphi \co \VN(G) \to \VN(G)$ is  decomposable. By \cite[Remark 1.5 p.~183]{Haa85}, there exist some linear maps $v_1,v_2 \co \VN(G) \to \VN(G)$ such that the linear map 
$ 
\begin{bmatrix} 
v_1 & M_\varphi \\ 
M_{\check{\ovl{\varphi}}} & v_2 
\end{bmatrix} 
\co \M_2(\VN(G)) \to \M_2(\VN(G))
$ 
is completely positive with $\max\{\norm{v_1},\norm{v_2}\}=\norm{M_\varphi}_{\dec,\VN(G) \to \VN(G)}$. We can suppose that the completely positive maps $v_1$ and $v_2$ are in addition weak* continuous by using \cite[Proposition 3.1 p.~24]{ArK23}. 

Now, we consider the projection $P_{\{1,2\},G}^\infty \co \CB_{\w^*}(\M_2(\VN(G))) \to \CB_{\w^*}(\M_2(\VN(G)))$, preserving the complete positivity and contractive, provided by \cite[Theorem 4.2 p.~62]{ArK23}. The proof shows that in case it is applied to an element of special structure as
$\begin{bmatrix} 
v_1 & M_\varphi \\ 
M_{\check{\ovl{\varphi}}} & v_2 
\end{bmatrix}$,
the mapping is $P_{\{1,2\},G}^\infty=\begin{bmatrix}
  P_{G}^\infty   &  P_{G}^\infty \\
  P_{G}^\infty   &  P_{G}^\infty \\
\end{bmatrix}$, where $P_{G}^\infty \co \CB_{\w^*}(\VN(G)) \to \CB_{\w^*}(\VN(G))$ is the contractive projection onto the space of completely bounded Fourier multipliers, provided by \cite[Theorem 4.2 p.~62]{ArK23}. We obtain that the map
\begin{align}
\MoveEqLeft
\label{Map-2x2-ttt-discrete}
\begin{bmatrix} 
   P_G^\infty(v_1)  & M_\varphi  \\
   M_{\check{\ovl{\varphi}}}  & P_G^\infty(v_2)  \\
\end{bmatrix}
=
\begin{bmatrix} 
   P_G^\infty(v_1)  & P_G^\infty(M_\varphi)  \\
  P_G^\infty(M_{\check{\ovl{\varphi}}})  & P_G^\infty(v_2)  \\
  \end{bmatrix}
	=P_{\{1,2\},G}^\infty\left(
\begin{bmatrix} 
v_1 & M_\varphi \\ 
M_{\check{\ovl{\varphi}}} & v_2 
\end{bmatrix} 
\right)
\end{align}
is completely positive. Moreover, we have 
\begin{align}
\label{Useful-estimation}
\MoveEqLeft
\max\big\{\norm{P_G^\infty(v_1)},\norm{P_G^\infty(v_2)}\big\}
\leq \norm{P_G^\infty} \max\big\{\norm{v_1},\norm{v_2}\big\}
=\norm{M_\varphi}_{\dec,\VN(G) \to \VN(G)}.           
\end{align}
We can write $P_G^\infty(v_1)=M_{\psi_1}$ and $P_G^\infty(v_2)=M_{\psi_2}$ for some continuous positive definite functions $\psi_1,\psi_2 \co G \to \mathbb{C}$. By \cite[Proposition 8.4 p.~166]{ArK23}, the condition \eqref{Condition-ArK} is satisfied with $\begin{bmatrix} 
\psi_1 & \varphi \\ 
\check{\ovl{\varphi}} &  \psi_2
\end{bmatrix}$ instead of $\begin{bmatrix} 
F_{11} & F_{12}\\ 
F_{21} & F_{22}
\end{bmatrix}$. 
By Lemma \ref{Lemma-Bloc-def-pos}, we conclude that $F \ov{\mathrm{def}}{=} \begin{bmatrix} 
\psi_1 & \varphi \\ 
\check{\ovl{\varphi}} &  \psi_2
\end{bmatrix}$ identifies to a continuous positive definite function on the groupoid $\mathrm{P}_2 \times G$. According to Proposition \ref{Prop-carac-BG-2-2}, we obtain that the function $\varphi$ belongs to the Fourier-Stieltjes algebra $\B(G)$. Moreover, using the well-known contractive inclusion $\frak{M}^\infty(G) \subseteq \L^\infty(G)$ of \cite[Proposition 5.1.2 p.~154]{KaL18} in the first inequality, we infer that
\begin{align*}
\MoveEqLeft
\norm{\varphi}_{\B(G)}
\ov{\eqref{Norm-B-G-utile}}{\leq} \norm{\psi_1}_{\L^\infty(G)}^{\frac{1}{2}} \norm{\psi_2}_{\L^\infty(G)}^{\frac{1}{2}} 
\leq \max\big\{\norm{M_{\psi_1}},\norm{M_{\psi_2}}\big\} \\
&= \max\big\{\norm{P_G^\infty(v_1)},\norm{P_G^\infty(v_2)}\big\} 
\ov{\eqref{Useful-estimation}}{\leq} \norm{M_\varphi}_{\dec,\VN(G) \to \VN(G)}.         
\end{align*}
%\textbf{OU} Using \cite[Corollary 1.1 p.~460]{Ren97} in the first inequality and \cite[Proposition 4 p.~1265]{Pat04} in the first equality, we see that
%\begin{align*}
%\MoveEqLeft
%\norm{\varphi}_{\B(G)}
%\ov{\eqref{Norm-B-G-utile}}{\leq} \norm{F}_{\L^\infty(\mathrm{P}_2 \times G)} 
%=\max\{ F_{11}(e) ,F_{22}(e)\}=\max\{ \psi_1(e) ,\psi_2(e)\} \\
%&= \max\big\{\norm{P_G^\infty(v_1)},\norm{P_G^\infty(v_2)}\big\} 
%\ov{\eqref{Useful-estimation}}{\leq} \norm{M_\varphi}_{\dec,\VN(G) \to \VN(G)}.         
%\end{align*}
%\textbf{Utiliser \cite[page 257]{BlM04}?}
\end{proof}

%\begin{remark} \normalfont
%\label{Rem-completely-isometric}

%\end{remark}

%\begin{thm}
%\label{th-completely-isometric}
%Let $G$ be a discrete group. The map $\B(G) \to \frak{M}^{\infty,\dec}(G)$, $\varphi \mapsto M_\varphi$ is a completely isometric isomorphism from the Fourier-Stieltjes algebra $\B(G)$ onto the algebra $\frak{M}^{\infty,\dec}(G)$ of decomposable multipliers.
%\end{thm}

Finally, we prove the second part of Conjecture \ref{conj} in the discrete case. This result improves \cite[Proposition 3.32 (1) p.~51]{ArK23} which says that the second inclusion of \eqref{Inclusions} is strict for any non-amenable weakly amenable discrete group $G$. This result can be seen as a new characterization of amenability for discrete groups.

\begin{thm}
\label{Thm-conj-discrete-case}
Let $G$ be a discrete group. The von Neumann algebra $\VN(G)$ is injective if and only if we have $\frak{M}^{\infty,\dec}(G)= \frak{M}^{\infty,\cb}(G)$.
\end{thm}

\begin{proof}
By Corollary \ref{dec-vs-B(G)-discrete-group}, we have an isometric isomorphism $\frak{M}^{\infty,\dec}(G)=\B(G)$. It suffices to use the result stated in \cite[p.~54]{Pis01}, which says that $\B(G) = \frak{M}^{\infty,\cb}(G)$ if and only if the group $G$ is amenable. For a discrete group $G$, the amenability is equivalent to the injectivity of the von Neumann algebra $\VN(G)$ by \cite[Theorem 3.8.2 p.~51]{SiS08} (or Theorem \ref{Th-Lau-Paterson}).
\end{proof}

%%%%%%%%%%%%%%%%%%%%%%%%%%%%%%%%%%%%%%%%%%%%%%%%%%%%%%%%%%%%%%%%%%%%%%%%%%%%%%%%%%%%%%%%%%%
\section{Inner amenability}
\label{Inner-amenability}
%%%%%%%%%%%%%%%%%%%%%%%%%%%%%%%%%%%%%%%%%%%%%%%%%%%%%%%%%%%%%%%%%%%%%%%%%%%%%%%%%%%%
\subsection{Background on inner amenability and amenability}
\label{Sec-prelim-inner}
%%%%%%%%%%%%%%%%%%%%%%%%%%%%%%%%%%%%%%%%%%%%%%%%%%%%%%%%%%%%%%%%%%%%%%%%%%%%%%%%%%%

We warn the reader that different notions of inner amenability coexist in the literature, see \cite[p.~84]{Pat88a} for more information. We say that a locally compact group $G$, equipped with a left Haar measure $\mu_G$, is inner amenable if there exists a state $m$ on the algebra $\L^\infty(G)$ such that
\begin{equation}
\label{inner-mean}
m(\inner_s f)
=m(f)
\end{equation}
for any $s \in G$, where
\begin{equation}
\label{def-conj-functions}
(\inner_sf)(t)
\ov{\mathrm{def}}{=} f\big(s^{-1}ts\big), \quad s,t \in G.
\end{equation}
It is worth noting that by \cite[Proposition 3.2 p.~2527]{CrT17}, a locally compact group $G$ is inner amenable if and only if there exists a state $m$ on the group von Neumann algebra $\VN(G)$ such that
\begin{equation*}
\label{state-G-invariant}
m(\lambda_s^* x\lambda_s)
=m(x),
\quad s \in G, x \in  \VN(G).	
\end{equation*}
Such a state is said to be $G$-invariant. According to \cite[Proposition 3.3 p.~2528]{CrT17}, any closed subgroup $H$ of an inner amenable locally compact group $G$ is inner amenable. If in addition $H$ is normal then the group $G/H$ is also inner amenable by \cite[Proposition 6.2 p.~168]{LaP91}.

\begin{example} \normalfont
\label{Ex-inner-2}
Every amenable locally compact group $G$ is inner amenable. Indeed, by \cite[Theorem 4.19 p.~36]{Pie84} there exists a state $m$ on $\L^\infty(G)$, which is two-sided invariant.
\end{example}

\begin{example} \normalfont
\label{Ex-inner-5}
Following \cite[p.~1273]{Pal01}, we say that a locally compact group $G$ is said to have an invariant neighborhood if there exists a compact neighbourhood $V$ of the identity $e$ in $G$ such that $V$ is stable under all inner automorphisms of $G$, i.e.~$s^{-1}Vs = V$ for all $s \in G$.  Such a group is said to be an IN-group. By \cite[Proposition 12.1.9 p.~1273]{Pal01}, any IN-group $G$ is unimodular. Note that if $\mu_G$ is a Haar measure on an IN-group $G$, it is clear using \cite[(31) and (33) VII.13]{Bou04b} that the state $m \co \L^\infty(G) \to \mathbb{C}$, $f \mapsto \frac{1}{\mu_G(V)}\int_V f$ satisfies the equation \eqref{inner-mean}. Hence an IN-group is inner amenable. 

By \cite[Proposition 6.36 p.~119]{ArK23}, a locally compact group $G$ is pro-discrete if and only if it admits a basis $(X_j)$ of neighborhoods of the identity $e$ consisting of open compact normal subgroups. Consequently, pro-discrete locally compact groups are IN-groups. Moreover, according to  \cite[Proposition 12.1.9 p.~1273]{Pal01}, compact groups, locally compact abelian groups and discrete groups groups are IN-groups. These groups are therefore all inner amenable. In particular, inner amenability is significantly weaker than amenability.
\end{example}

%In particular , by \cite[Proposition 12.1.9]{Pal01}, a SIN-group is IN. 
%Recall that a locally compact group $G$ has small invariant neighborhoods $\SIN$ if there exists a basis $(V_j)_j$ of compact neighborhoods of $e$ which are invariant under all inner automorphisms of $G$. Such a group is unimodular by \cite[Proposition 12.1.9]{Pal01} and obviously inner amenable by Theorem \ref{thm-inner-amenable-Folner}. 

%Recall the following characterization of amenability of \cite[Corollary 3.2]{LaP91} and \cite[page 85]{Pat88a} which we will use\footnote{\thefootnote. We added the assumption ``second-countable'' since this equivalence uses the equivalence between ``injective'' and property $P$ for von Neumann algebras with separable preduals.}.
%
%\begin{thm}
%\label{Th-Lau-Paterson}
%Let $G$ be a second-countable locally compact group. The following are equivalent:
%\begin{enumerate}
	%\item the group von Neumann algebra $\VN(G)$ is injective and $G$ is inner amenable.
	%\item $G$ is amenable. 
%\end{enumerate}
%\end{thm}

\begin{example} \normalfont 
\label{Contre-example}
Recall that a topological group $G$ is of type I \cite[Definition 6.D.1 p.~196 and Proposition 7.C.I p.~219]{BeH20} if for any continuous unitary representation $\pi$ of $G$, the von Neumann algebra $\pi(G)''$ is of type I, hence injective by \cite[Proposition 10.23 p.~144]{Str81}. In particular, by Theorem \ref{Th-Lau-Paterson} a second-countable locally compact group $G$ of type I is inner amenable if and only if it is amenable. We refer to \cite[Theorem 6.E.19 p.~208 and Theorem 6.E.20 p.~209]{BeH20} for an extensive list of locally compact groups of type I, including connected nilpotent  locally compact groups and linear algebraic groups over a local field of characteristic 0.
%In particular, the group $\mathrm{SL}_2(\Q_p)$ is a non-inner amenable totally disconnected locally compact group since it is non-amenable by \cite[Example G.3.6 (iii)]{BHV08}.
\end{example}

\begin{example} \normalfont 
\label{Example-almost}
If a locally compact group $G$ is almost connected, i.e.~$G/G_e$ is compact if $G_e$ is the connected component of the identity $e$, then its von Neumann algebra $\VN(G)$ is injective by \cite[p.~228]{Pat88b}. Again, by Theorem \ref{Th-Lau-Paterson} such a group is inner amenable if and only if it is amenable. This result in the connected case was first proved by Losert and Rindler in \cite[Theorem 1 p.~222]{LoR87} and proven again in \cite[Corollary 3.4 p.~161]{LaP91}.
\end{example}

If $A$ and $B$ are two subsets of a set $E$, the notation $A \Delta B\ov{\mathrm{def}}{=}(A-B) \cup (B-A)$ denotes here the symmetric difference of $A$ and $B$. Recall that 
\begin{equation}
\label{Indicator-formula}
|1_A-1_B|
=1_{A \Delta B}.
\end{equation}

We will use the following reformulation of \cite[Lemma 8.6 p.~43]{CPPR15}, which is actually and essentially a variant of a classical trick in amenability theory used in \cite[pp.~364-365]{EiW17},  \cite[pp.~441-442]{BHV08} and \cite[p.~410]{Fre13}. We give the two lines of calculus for the sake of completeness. %\textbf{(c'est la meme chose ? J'ai l'impression. On peut remplacer la somme par une integration sur un compact K, on pourrait etendre le th principal en remplacant les $F$ par des compacts $K$; a faire ou pas, je crois que pour les assertions ci-desous pour second-countable juste avant Rem \ref{Remark-3.8}, on en a besoin).}

\begin{lemma}
\label{Lemma-CPPR}
Let $G$ be a locally compact group equipped with a left Haar measure $\mu_G$. Let $\epsi >0$ and consider some positive functions $f,g_1,\ldots,g_n$ in the space $\L^1(G)$ satisfying the inequality $\sum_{k=1}^{n} \norm{f-g_k}_{\L^1(G)} < \epsi$ and $\norm{f}_{\L^1(G)}=1$. Then there exists $t > 0$ such that
\begin{equation}
\label{Equa-CPPR}
\sum_{k=1}^{n} \mu_G\big(\{f >t\} \Delta \{g_k >t\} \big) 
< \epsi \mu_G(\{f > t\}) .
\end{equation}
\end{lemma}

\begin{proof}
For any $s \in G$ and integer $1 \leq k \leq n$, we have by \cite[Lemma G.5.2 p.~441]{BHV08} and \eqref{Indicator-formula} the equalities
\begin{equation}
\label{BHV08}
\norm{f}_{\L^1(G)}
=\int_0^\infty \mu_G(\{f > t\}) \d t
\quad \text{and} \quad
\norm{f - g_k}_{\L^1(G)}
 =\int_0^\infty \mu_G\big(\{f >t\} \Delta \{g_k >t\}\big) \d t.
\end{equation}
%\begin{equation}
%\label{}
%|f(s)|
%=\int_0^\infty  1_{\{f(s) > t\}} \d t, 
%\quad
%| f(s) - g_k(s) |
 %=\int_0^\infty \big| 1_{\{f(s) > t\}} - 1_{\{g_k(s) > t\}} \big| \d t.
%\end{equation}
We deduce that
\begin{align*}
\MoveEqLeft
\int_0^\infty \sum_{k=1}^n \mu_G\big(\{f >t\} \Delta \{g_k >t\}\big) \d t
=\sum_{k=1}^{n}\int_0^\infty \mu_G\big(\{f >t\} \Delta \{g_k >t\}\big) \d t \\
&\ov{\eqref{BHV08}}{=} \sum_{k=1}^n \norm{ f - g_k }_{\L^1(G)} 
<  \epsi \norm{f}_{\L^1(G)} 
\ov{\eqref{BHV08}}{=} \epsi \int_0^\infty \mu_G(\{f > t\}) \d t .
\end{align*}
The conclusion is obvious.
\end{proof}

\paragraph{Convolution}
If $G$ is a \textsl{unimodular} locally compact group equipped with a Haar measure $\mu_G$,  recall that the convolution product of two functions $f$ and $g$ is given, when it exists, by
\begin{equation}
\label{Convolution-formulas}
(f*g)(s)
\ov{\mathrm{def}}{=} \int_G f(r)g(r^{-1}s) \d\mu_G(r)
=\int_G f(sr^{-1})g(r) \d\mu_G(r).
\end{equation}

%%%%%%%%%%%%%%%%%%%%%%%%%%%%%%%%%%%%%%%%%%%%%%%%%%%%%%%%%%%%%%%%%%%%%%%%%%%%%%%%%%%%%%%
\subsection{Some characterizations of inner amenability}
\label{subsec-inner-Folner}

%A discrete group is $[SAIN]$ by choosing $V_j = \{e\}$ and observing that $p V_j p^{-1} = p \{e\} p^{-1} = \{e\} = V_j$, so that in fact, $\mu(V_j \Delta p V_j p^{-1}) = 0$.

Now, we introduce the following definition which is an <<inner variant>> of the well-known definition of the notion of <<F\o{}lner net>> in amenability theory.

\begin{defi}
\label{def-IF}
A locally compact group $G$ is said to be inner F\o{}lner (in short $G \in \IF$) if for every finite subset $F$ of $G$ there exists a net $(V_j^F)_j$ of measurable subsets of $G$ such that $\mu(V_j^F) \in (0,\infty)$, with the property that for all $s \in F$,
\begin{equation}
\label{Inner-Folner}
\frac{\mu(V_j^F \Delta (s^{-1}V_j^Fs))}{\mu(V_j^F)} 
\xra[j \to \infty]{} 0.
\end{equation}
\end{defi}

Now, we give different characterizations of inner amenability for unimodular locally compact groups. The equivalence between the first and the second point is sketched in \cite[Proposition 1 p.~222]{LoR87}. For the sake of completeness, we give a complete proof. 

%Clearly a $\SAIN$ group is $\sIF$ and thus $\IF$.

%\textbf{I have assumed unimodular in order to have $\inner_s : \L^1(G) \to \L^1(G)$ isometric. Also convolution is easier with unimodular. Maybe unimodular superfluous, but should check carefully. For symmetric and SAIN, cela parait indispensable.}

\begin{thm}
\label{thm-inner-amenable-Folner}
Let $G$ be a unimodular locally compact group. The following are equivalent.
\begin{enumerate}
\item $G$ is inner amenable.
\item There exists an asymptotically central net $(f_j)$ of functions in the space $\L^1(G)$, i.e.~for any $s \in G$, we have 
\begin{equation}
\label{asymt-central}
\frac{\norm{f_j-\inner_sf_j}_{\L^1(G)}}{\norm{f_j}_{\L^1(G)}}
\xra[j ]{} 0.
\end{equation}

\item $G$ is inner F\o{}lner.

\item There exists a net $(f_j)$ of positive functions in the space $\L^1(G)$ with $\int_G f_j \d\mu = 1$ such that for all $s \in G$, we have $\norm{f_j - \inner_s f_j}_{\L^1(G)} \xra[j]{} 0$.

\item The same property as before, but the $f_j$'s belong in addition to the space $\C_c(G)$ and are positive definite.

\item $G$ is inner F\o{}lner and in addition the sets $V_j^F$ can be chosen to be symmetric, open and containing $e$.
\end{enumerate}
%Moreover, if $G$ is second-countable, then the nets $(f_j)$ can be chosen to be sequences.
Finally, the net $(V_j^F)_j$ in the previous definition of inner F\o{}lner can be chosen to be a sequence.
\end{thm}

\begin{proof}
1. $\Longrightarrow$ 4.: %\textbf{Meme genre d'argument dans \cite[$(\gamma)$ page 408]{Fre4I} \cite[page 454]{BHV08}\cite[page 367]{EiW17}}
Let $m \in \L^\infty(G)^*$ be an inner invariant mean. By \cite[Proposition 3.3 p.~25]{Pie84} (see also \cite[Lemma 10.16 p.~366]{EiW17}), we can  approximate $m$ in the weak* topology by a net $(f_j)$ of functions in $\L^1(G)$ with $f_j \geq 0$ and $\norm{f_j}_{\L^1(G)}=1$. For any $s \in G$ and any $g \in \L^\infty(G)$, we have
\begin{align*}
\MoveEqLeft           
\la \inner_s(f_j), g \ra_{\L^1(G),\L^\infty(G)}
= \la f_j, \inner_{s^{-1}}(g) \ra_{\L^1(G),\L^\infty(G)}
\xra[j \to \infty]{} \la m, \inner_{s^{-1}}(g) \ra
=\la m, g \ra.
\end{align*} 
and $\la f_j, g \ra_{\L^1(G),\L^\infty(G)} \to \la m, g \ra_{\L^\infty(G)^*,\L^\infty(G)}$. With an $\frac{\epsi}{2}$-argument, it follows that for any $s \in G$ we have $\w-\lim_j (\inner_sf_j-f_j)=0$. 

Since for convex sets the weak closure coincides with the norm closure \cite[Theorem 2.5.16 p.~216]{Meg98}, we can replace $f_j$ by some convex combinations to get $\lim_j \norm{\inner_s(f_j) -f_j}_{\L^1(G)}=0$.
This replacement can be seen in the following way.
Let $F \ov{\mathrm{def}}{=} \{ s_1, \ldots, s_n\}$ be a finite set of $G$.
According to the above, $(0,\ldots,0)$ belongs to the weak-closure of the convex hull of $\{(\inner_{s_1}(f_j) - f_j, \inner_{s_2}(f_j) - f_j, \ldots, \inner_{s_n}(f_j) - f_j) : j \}$, hence to the $\L^1(G)^n$ norm closure of this convex hull. 
Thus there exists a sequence $(g_k)_k$ in this convex hull converging to $(0,\ldots,0)$ in norm.
For any $k \in \N$, we can write 
\begin{align*}
\MoveEqLeft
g_k 
=\sum_{\ell=1}^L \lambda_\ell \big(\inner_{s_1}(f_{j_\ell}) - f_{j_\ell},\ldots, \inner_{s_n}(f_{j_\ell}) - f_{j_\ell}\big) \\
&= \bigg(\inner_{s_1}\bigg(\sum_{\ell=1}^L \lambda_\ell f_{j_\ell}\bigg) - \sum_{\ell=1}^L \lambda_\ell f_{j_\ell},\ldots, \inner_{s_n}\bigg(\sum_{\ell=1}^L \lambda_\ell f_{j_\ell}\bigg) - \sum_{\ell=1}^L \lambda_\ell f_{j_\ell}\bigg) \\
&= \big(\inner_{s_1}(h_k) -  h_k, \ldots, \inner_{s_n}( h_k) - h_k\big),
\end{align*}
where $\lambda_\ell \geq 0$, $\sum_{\ell=1}^L \lambda_\ell = 1$ and where $h_k \ov{\mathrm{def}}{=} \sum_{\ell=1}^L \lambda_\ell f_{j_\ell}$ still is a positive normalized element in $\L^1(G)$.
We can suppose that $\norm{\inner_{s}(h_k) - h_k}_1 \leq \frac{1}{k}$ for any $s \in F$.
Now write $h_k = h_{k,F}$, let $F$ vary in the set of finite subsets of $G$ directed by inclusion, so that $(h_{k,F})_{k,F}$ becomes a net in $\L^1(G)$ such that $\norm{\inner_s(h_{k,F}) - h_{k,F}}_1 \to 0$ as $(k,F) \to \infty$ for any $s \in G$.

4. $\Longrightarrow$ 1.: Note that we have an isometric inclusion $\L^1(G) \subseteq \L^\infty(G)^*$. Consider a cluster point $m \in \L^\infty(G)^*$ of this net for the weak* topology which is positive and satisfies clearly $m(1)=1$. For any $f \in \L^\infty(G)$ and any $s \in G$, we have
\begin{align*}
\MoveEqLeft           
\left|\la f_j, f \ra_{\L^1(G),\L^\infty(G)}-\big\la f_j,\inner_s f\big\ra_{\L^1(G),\L^\infty(G)}\right|
=\left|\la f_j, f \ra-\big\la \inner_{s^{-1}} f_j,f \big\ra\right|
=\left|\big\la f_j-\inner_{s^{-1}} f_j, f \big\ra \right| \\
&\leq \norm{f_j-\inner_{s^{-1}} f_j}_{\L^1(G)} \norm{f}_{\L^\infty(G)}\xra[j]{} 0. 
%&=\int_G \big(f_j-\inner_s f_j \big) \d\mu \norm{f}_{\L^\infty(G)} \\
%&=\\
%&=\frac{1}{\mu(V_j)} \left|\int_{V_j} \inner_s(f) \d \mu-\int_{V_j} f \d \mu\right|\\
%&=\frac{1}{\mu(V_j)} \left|\int_{s^{-1}V_js} f \d \mu-\int_{V_j} f \d \mu\right| 
%\leq\frac{1}{\mu(V_j)} \int_{G} |f| |1_{s^{-1}V_js}-1_{V_j}| \d \mu \\
%&\ov{\eqref{Indicator-formula}}{\leq} \frac{\norm{f}_{\L^\infty(G)}}{\mu(V_j)} \int_{G} 1_{(s^{-1}V_js) \Delta V_j} \d \mu
%=\norm{f}_{\L^\infty(G)}\frac{\mu((s^{-1} V_j s) \Delta V_j)}{\mu(V_j)} 
%\xra[j]{} 0.
\end{align*}
With a $\frac{\epsi}{3}$-argument we conclude that \eqref{inner-mean} is satisfied.

3. $\Longrightarrow$ 2.: Let $F$ be a finite subset of $G$. According to the assumption, there exists a subset $V_j^F$ such that $\frac{\mu(V_j^F \Delta (s^{-1}V_j^Fs))}{\mu(V_j^F)} \leq \frac{1}{j}$ for any $s \in F$.  Putting $f_j^F \ov{\mathrm{def}}{=} 1_{V_j^F}$ and using \eqref{Indicator-formula}, we obtain from Definition \ref{def-IF} that $\norm{f_j^F - \inner_s f_j^F}_1 / \norm{f_j^F}_1 \leq \frac{1}{j}$ for any $s \in F$. Directing the subsets $F$ by inclusion and the $j$ in the usual manner, we obtain a net of positive functions $f_j^F$ in $\L^1(G)$ as in \eqref{asymt-central}.

\noindent
2. $\Longrightarrow$ 4.: Using a normalization, it suffices to see that by the elementary inequality $\int_G \big||f|-|g|\big| \leq \int_G |f-g|$, the $f_j$'s in \eqref{asymt-central} can be chosen positive.

\noindent
4. $\Longrightarrow$ 5.: Let $(f_j)_j$ be the net as in 4. For $\epsi > 0$, choose some $f_{j,\epsi} \in \C_c(G)$ such that $f_{j,\epsi} \geq 0$, and $\norm{f_{j,\epsi} - f_j}_1 < \epsi$. Since the map $\inner_s \co \L^1(G) \to \L^1(G)$ is isometric,  for any $s \in G$ we have 
\[ 
\norm{f_{j,\epsi} - \inner_s f_{j,\epsi}}_1 
\leq \norm{f_{j,\epsi}-f_j}_1 + \norm{f_j - \inner_s f_j}_1 + \norm{\inner_s f_j - \inner_s f_{j,\epsi}}_1 
\leq \epsi + \norm{f_j - \inner_s f_j}_1 + \epsi. 
\]
We can suppose that $\int_G f_{j,\epsi} \d\mu = 1$. Replacing the index $j$ by $(j,\epsi)$ and equipping it with the suitable order, we obtain a net $(f_{j,\epsi})_{j,\epsi}$ of positive normalized continuous functions with compact support with the convergence property from 3. We may and do thus assume now that the net $(f_j)_j$ in the third point consists of continuous functions with compact support. 

For any $j$, put now $g_j \ov{\mathrm{def}}{=} f_j \ast \check{f}_j$ where $\check{f}_j (s) \ov{\mathrm{def}}{=} f_j(s^{-1})$. For any $ s \in G$ we have
\begin{equation}
\label{Def-de-g}
g_j(s)
=(f_j \ast \check{f}_j)(s)
\ov{\eqref{Convolution-formulas}}{=} \int_G f_j(st^{-1})\check{f}_j(t) \d \mu(t)
=\int_G f_j(sr)f_j(r) \d \mu(r)
\geq 0.
\end{equation}
Then for any $j$
\begin{align*}
\MoveEqLeft
\norm{g_j}_1 
=\int_G g_j(s) \d \mu(s)
\ov{\eqref{Def-de-g}}{=} 
\int_G \int_G f_j(sr) f_j(r) \d\mu(r) \d\mu(s)  
=\int_G \bigg(\int_G f_j(sr) \d\mu(s)\bigg) f_j(r) \d\mu(r) \\
&=\int_G \bigg(\int_G f_j(u) \d\mu(u)\bigg) f_j(r) \d\mu(r)
= \norm{f_j}_1^2 
= 1.
\end{align*} 
Moreover, for any $t \in G$, we have
\begin{equation}
\label{divers-33455}
(\inner_t g_j)(s) 
\ov{\eqref{def-conj-functions}}{=} g_j(t^{-1}st)
\ov{\eqref{Def-de-g}}{=} \int_G f_j(t^{-1}st r) f_j(r) \d\mu(r) 
= \int_G f_j(t^{-1}srt) f_j(t^{-1}rt) \d\mu(r).
\end{equation}
Thus we obtain
\begin{align*}
\MoveEqLeft
\norm{g_j - \inner_t g_j}_1 
=\int_G |g_j(s) - (\inner_t g_j)(s)| \d\mu(s) \\
&\ov{\eqref{Def-de-g} \eqref{divers-33455}}{=} \int_G \left| \int_G f_j(sr) f_j(r) - f_j(t^{-1}srt) f_j(t^{-1}rt) \d\mu(r) \right| \d\mu(s) \\
& \leq \int_G \left| \int_G f_j(sr) [f_j(r) - f_j(t^{-1}rt) ] \d\mu(r) \right| 
+ \left| \int_G [ f_j(sr) - f_j(t^{-1}srt) ] f_j(t^{-1}rt) \d\mu(r) \right| \d\mu(s) \\
&\leq \int_{G \times G} f_j(sr) |f_j(r) - f_j(t^{-1}rt)| \d\mu(r)\d\mu(s) \\
&+  \int_{G \times G} | f_j(sr) - f_j(t^{-1}srt)|f_j(t^{-1}rt) \d\mu(r) \d\mu(s) \\
&\leq \norm{f_j}_1 \norm{f_j - \inner_t f_j}_1 + \norm{f_j - \inner_t f_j}_1 \norm{\inner_t f_j}_1 = 2 \norm{f_j - \inner_t f_j}_1.
\end{align*}
This shows that the $g_j$'s have the same normalisation and convergence property than the $f_j$'s.
Moreover, by \cite[p.~281]{HeR70} the $g_j$'s are continuous positive definite functions with compact support.

\noindent
5. $\Longrightarrow$ 6.: Let $F = \{ s_1, \ldots, s_n \}$ be a finite subset of $G$ and $\epsi > 0$. According to the fifth point and \eqref{asymt-central}, choose some positive definite functions $f_j \in \C_c(G)$ such that $\norm{f_j - \inner_s f_j}_1 / \norm{f_j}_1 < \epsi / \card F$ for all $s \in F$. Using Lemma \ref{Lemma-CPPR} with $n = \card F$, $f = f_j$ and $g_k = \inner_{s_k^{-1}} f_j$ and the subset $V \ov{\mathrm{def}}{=} \{  f_j> t \}$ of $G$, we deduce that for some suitable $t > 0$ and $s \in F$,
\begin{align*}
\MoveEqLeft
\sum_{s \in F} \mu\big(V \Delta(s^{-1} V s)\big) 
%\ov{\eqref{Indicator-formula}}{=} \sum_{s \in F} \norm{1_{V} - 1_{\inner_s V}}_1 
= \sum_{s \in F} \mu\big(\{f_j > t\} \Delta \{\inner_{s^{-1}} f_j > t\}\big) 
\ov{\eqref{Equa-CPPR}}{<} \epsi \mu(\{f_j > t\}) 
= \epsi \mu(V).
\end{align*}
Therefore, the group $G$ is inner F\o{}lner. Moreover, since $f_j$ is continuous, $V$ is an open subset of $G$. Furthermore, since the function $f_j$ is positive definite, we have $\norm{f_j}_\infty = f_j(e_G)$ by \cite[p.~23]{KaL18}. We deduce that $e_G$ belongs to $V$ since otherwise we would have $V = \emptyset$ and the previous strict inequality could not hold. Finally, \cite[Proposition 1.4.16 (ii) p.~22]{KaL18}, we have $f_j=\check{f_j}$ since $f_j \geq 0$. We conclude that $V$ is symmetric. 

\noindent

6. $\Longrightarrow$ 3.: trivial. 

We turn to the last sentence of the statement.
So we assume that $G$ is an inner F\o{}lner group, such that for any finite subset $F$ of $G$ there exists a net $(V_\alpha^F)_\alpha$ of measurable subsets of $G$ such that $\mu(V_\alpha^F) \in (0,\infty)$, with the property that for all $s \in F$,
\begin{equation}
\label{equ-comment-11122024}
\frac{\mu(V_\alpha^F \Delta (s^{-1}V_\alpha^Fs))}{\mu(V_\alpha^F)} 
\xra[\alpha \to \infty]{} 0.
\end{equation}
We will construct a sequence $(W_j^F)_j$, indexed by $j \in \N$, that satisfies the same convergence property \eqref{equ-comment-11122024} as the $(V_\alpha^F)_\alpha$.
Start by putting $\epsi = 1$.
By \eqref{equ-comment-11122024}, for all $s \in F$, there exists some $\alpha(1,s)$ such that if $\alpha \geq \alpha(1,s)$, then
\[ 
\frac{\mu(V_\alpha^F \Delta (s^{-1}V_\alpha^Fs))}{\mu(V_\alpha^F)} 
\leq 1 .
\]
Choose some $\alpha(1) \geq \alpha(1,s)$ for all $s \in F$ (directed set property) and put $W_1^F \ov{\mathrm{def}}{=} V_{\alpha(1)}^F$. Now, let $\epsi = \frac12$. Again by \eqref{equ-comment-11122024}, for all $s \in F$, there exists some $\alpha(\frac12,s)$ such that if $\alpha \geq \alpha(\frac12,s)$, then
\[ 
\frac{\mu(V_\alpha^F \Delta (s^{-1}V_\alpha^Fs))}{\mu(V_\alpha^F)} 
\leq \frac12 .
\]
Choose some $\alpha(\frac12) \geq \alpha(\frac12,s)$ for all $s \in F$ and put $W_2^F \ov{\mathrm{def}}{=} V_{\alpha(\frac12)}^F$. Continue with $\epsi = \frac14,\frac18,\ldots$ and obtain a sequence of subsets $W_j^F \ov{\mathrm{def}}{=} V_{\alpha(\frac{1}{2^{j-1}})}^F$ such that for all $s \in F$ we have
\[ 
\frac{\mu(W_j^F \Delta (s^{-1}W_j^Fs))}{\mu(W_j^F)} 
= \frac{\mu(V_{\alpha(\frac{1}{2^{j-1}})}^F \Delta (s^{-1}V_{\alpha(\frac{1}{2^{j-1}})}^Fs))}{\mu(V_{\alpha(\frac{1}{2^{j-1}})}^F)} 
\leq \frac{1}{2^{j-1}}.
\]
For any $s \in F$, we infer that
\[ 
\frac{\mu(W_j^F \Delta (s^{-1}W_j^Fs))}{\mu(W_j^F)} 
\xra[j \to \infty]{} 0. 
\]
\end{proof}

%%%%%%%%%%%%%%%%%%%%%%%%%%%%%%%%%%%%%%%%%%%%%%%%%%%%%%%%%%%%%%%%%%%%%%%%%%%%%%%%%%%%%
\section{Projections on the space of completely bounded Fourier multipliers}
\label{Sec-complementation}
%%%%%%%%%%%%%%%%%%%%%%%%%%%%%%%%%%%%%%%%%%%%%%%%%%%%%%%%%%%%%%%%%%%%%%%%%%%%%%%%%%%%
\subsection{Preliminaries}
\label{Sec-prel-complet}

\paragraph{Hilbert-Schmidt operators} Let $\Omega$ be a $\sigma$-finite measure space. We will use the space $S^\infty_\Omega \ov{\mathrm{def}}{=} S^\infty(\L^2(\Omega))$ of compact operators, its dual $S^1_\Omega$ and the space $\cal{B}(\L^2(\Omega))$ of bounded operators on the complex Hilbert space $\L^2(\Omega)$. If $f \in \L^2(\Omega \times \Omega)$, we denote the associated Hilbert-Schmidt operator by
\begin{equation}
\label{Def-de-Kf}
\begin{array}{cccc}
  K_f  \co &  \L^2(\Omega)   &  \longrightarrow   & \L^2(\Omega)   \\
    &   \xi  &  \longmapsto       &  \int_{X} f(\cdot,y)\xi(y) \d y  \\
\end{array}.	
\end{equation}
Using the notation $\check{f}(x,y) \ov{\mathrm{def}}{=} f(y,x)$, we have $(K_f)^*=K_{\check{\ovl{f}}}$. Note that the linear map $\L^2(\Omega \times \Omega) \to S^2_\Omega$, $f \mapsto K_f$ is an isometry from the Hilbert space $\L^2(\Omega \times \Omega)$ onto the Hilbert space $S^2_\Omega$ of Hilbert-Schmidt operators acting on the Hilbert space $\L^2(\Omega)$. This means that
\begin{equation}
\label{dual-trace}
\tr(K_f K_g)
=\int_{\Omega \times \Omega} f \check{g}, \quad f,g \in \L^2(\Omega \times \Omega).
\end{equation}

\paragraph{Schur multipliers acting on $S^p_\Omega$}
Suppose that $1 \leq p \leq \infty$. We say that a measurable function $\varphi \co \Omega \times \Omega \to \mathbb{C}$ induces a bounded Schur multiplier on the Schatten class $S^p_\Omega \ov{\mathrm{def}}{=} S^p(\L^2(\Omega))$  if for any $f \in \L^2(\Omega \times \Omega)$ satisfying $K_f \in S^p_\Omega$ we have $K_{\varphi f} \in S^p_\Omega$ and if the map $S^2_\Omega \cap S^p_\Omega \to S^p_\Omega$, $K_f \mapsto K_{\varphi f}$ extends to a bounded map $M_\varphi$ from $S^p_\Omega$ into $S^p_\Omega$ called the Schur multiplier associated with $\varphi$. It is well-known \cite[Remark 1.4 p.~77]{LaS11} that in this case $\varphi \in \L^\infty(\Omega \times \Omega)$ and that
\begin{equation}
\label{ine-infty}
\norm{\varphi}_{\L^\infty(\Omega \times \Omega)} 
\leq \norm{M_\varphi}_{S^p_\Omega \to S^p_\Omega}.
\end{equation}
We denote by $\mathfrak{M}_{\Omega}^{p}$ the space of bounded Schur multipliers on $S^p_\Omega$ and by $\mathfrak{M}_{\Omega}^{p,\cb}$ the subspace of completely bounded ones.

\paragraph{Schur multipliers acting on $\cal{B}(\L^2(\Omega))$}
 We say that a function $\varphi \in \L^\infty(\Omega \times \Omega)$ induces a Schur multiplier on $\cal{B}(\L^2(\Omega))$ if the map $S^2_\Omega \mapsto \cal{B}(\L^2(\Omega))$, $K_{f} \mapsto K_{\varphi f}$ induces a bounded operator from $S^\infty_\Omega$ into $\cal{B}(\L^2(\Omega))$. In this case, the operator $S^\infty_\Omega \mapsto \cal{B}(\L^2(\Omega))$, $K_{f}\mapsto K_{\varphi f}$ admits by \cite[Lemma A.2.2 p.~360]{BlM04} a unique weak* extension $M_\varphi \co \cal{B}(\L^2(\Omega)) \to \cal{B}(\L^2(\Omega))$ called the Schur multiplier associated with $\varphi$. It is known that $M_\varphi$ induces a bounded map $M_\varphi \co S^p_\Omega \to S^p_\Omega$ for any $1 \leq p \leq \infty$. We refer to the surveys \cite{ToT10} and \cite{Tod15} for more information. See also the papers \cite{Arh24} and \cite{Spr04}.

\begin{example} \normalfont
If the set $\Omega=\{1,\ldots,n\}$ is equipped with the counting measure, we can identify the space $\cal{B}(\L^2(\Omega))$ with the matrix algebra $\M_n$. Then each operator $K_{f}$ identifies to the matrix $[f(i,j)]$. A Schur multiplier is given by a map $M_\varphi \co \M_n \to \M_n$, $[f(i,j)] \mapsto [\varphi(i,j)f(i,j)]$. 
\end{example}

By \cite[Proposition 4.3]{Arh24}, the map $S^2_\Omega \to S^2_\Omega$, $K_f \mapsto K_{\check{f}}$ extends to an involutive normal $*$-antiautomorphism $R \co \cal{B}(\L^2(\Omega)) \to \cal{B}(\L^2(\Omega))$. We introduce the following duality bracket
\begin{equation}
\label{Duality-bracket}
\langle z,y \rangle_{\cal{B}(\L^2(\Omega)), S^1_\Omega}
\ov{\mathrm{def}}{=} \tr(R(z)y), \quad z \in \cal{B}(\L^2(\Omega)), y \in S^1_\Omega,
\end{equation}
which is more suitable than the bracket $\langle z,y \rangle= \tr(zy)$ since we have
\begin{equation}
\label{auto-adjoint}
\big\langle M_{\varphi}(z),y \big\rangle_{\cal{B}(\L^2(\Omega)), S^1_\Omega}
=\big\langle z,M_{\varphi}(y) \big\rangle_{\cal{B}(\L^2(\Omega)), S^1_\Omega}, \quad z \in \cal{B}(\L^2(\Omega)), y \in S^1_\Omega
\end{equation}
for any Schur multiplier $M_\varphi$ and since the operator space duality requires taking the opposite structure into account.

%Let $\Omega$ be a locally compact space equipped with a Radon measure with full support. Let $\varphi \in \L^\infty(\Omega \times \Omega)$. By \cite[Proposition 4.7 and Remark 4.8 p.~784-785]{Arh24}, $\varphi$ induces a completely positive measurable Schur multiplier if and only if the function $\varphi$ is an integrally positive definite kernel, i.e.~we have for any $\xi \in \L^1(\Omega)$ the function $(x,y) \mapsto \varphi(x,y) \ovl{\xi(x)} \xi(y)$ is integrable on $\Omega \times \Omega$ and 
%\begin{equation}
%\label{Cartier-int-pos}
%\iint_{\Omega \times \Omega} \varphi(x,y) \ovl{\xi(x)} \xi(y) \d x \d y 
%\geq 0.
%\end{equation}

\paragraph{Herz-Schur multipliers}
Let $G$ be a (second-countable) unimodular locally compact group. Following \cite[p.~179]{Spr04}, a bounded Schur multiplier $M_\varphi \co \cal{B}(\L^2(G)) \to \cal{B}(\L^2(G))$ is a Herz-Schur multiplier if for any $r \in G$ we have $\varphi(sr,t)=\varphi(s,tr^{-1})$ for marginally almost all $(s,t)$ in $G \times G$. We define similarly the notion of Herz-Schur multiplier on $S^p_G$. We denote by $\frak{M}^{p,\cb,\HS}_G$ the subspace of $\frak{M}^{p,\cb}_G$ of completely bounded Herz-Schur multipliers. We define similarly $\frak{M}^{p,\HS}_G$. If $\varphi \co G \to \mathbb{C}$, we introduce the function $\varphi^\HS \co G \times G \to \mathbb{C}$, $(s,t) \mapsto \varphi(st^{-1})$. By \cite{BoF84} and \cite[Theorem 5.3 p.~181]{Spr04}, the linear map $\frak{M}^{\infty,\cb}(G) \to \frak{M}^{\infty,\cb,\HS}_G=\frak{M}^{\infty,\HS}_G$, $M_\varphi \mapsto M_{\varphi^\HS}$ is a surjective isometry. We let $M_{\varphi}^{\HS} \ov{\mathrm{def}}{=} M_{\varphi^\HS}$.

%The right regular representation $\rho \co G \to \cal{B}(\L^2(G))$ is given by $(\rho_t\xi)(s) \ov{\mathrm{def}}{=} \xi(st)$. Recall that $\rho$ is a strongly continuous unitary representation. Suppose that $1 \leq p \leq \infty$. For any $s \in G$, we will use the notation $\Ad_{\rho_s}^p \co S^p_G \to S^p_G$, $x \mapsto \rho_s x \rho_{s^{-1}}$. A bounded Schur multiplier $M_\varphi \co S^p_G \to S^p_G$ is a Herz-Schur multiplier if 
%$$
%M_\varphi\Ad_{\rho_s}^p =\Ad_{\rho_s}^p M_\varphi
%$$
%for any $s \in G$. 

%In this case, there exists a measurable function $\phi \co G \to \mathbb{C}$ such that $\varphi(r,s)=\phi(rs^{-1})$ for almost all $r,s \in G$ and 

%Let $G$ be a unimodular locally compact group.  The right regular representation $\rho \co G \to \cal{B}}(\L^2(G))$ is given by $(\rho_t\xi)(s) = \xi(st)$ and is a strongly continuous unitary representation. We will use the notation $\mathrm{Ad}_{\rho_s}^p \co S^p_G \to S^p_G$, $x \mapsto \rho_s x \rho_{s^{-1}}$. A bounded Schur multiplier $M_\phi \co S^p_G \to S^p_G$ is a Herz-Schur multiplier if $M_\phi\mathrm{Ad}_{\rho_s}^p =\mathrm{Ad}_{\rho_s}^p M_\phi$ for any $s \in G$. In this case, there exists a measurable function $\varphi \co G \to \mathbb{C}$ such that $\phi(r,s)=\varphi(rs^{-1})$ for almost every $r,s \in G$ and we let $M_{\varphi}^{\HS}=M_{\phi}$. 

\paragraph{Plancherel weights} Let $G$ be a locally compact group. A function $g \in \L^2(G)$ is called left bounded \cite[Definition 2.1]{Haa78b} if the convolution operator $\lambda(g) \co \C_c(G) \to \C_c(G)$, $f\mapsto g*f$ induces a bounded operator on the Hilbert space $\L^2(G)$. The Plancherel weight $\tau_G \co \VN(G)^+\to [0,\infty]$ is\footnote{\thefootnote. This is the natural weight associated with the left Hilbert algebra $\C_c(G)$.} defined by the formula
$$
\tau_G(x)
= \begin{cases}
\norm{g}^2_{\L^2(G)} & \text{if }x^{\frac{1}{2}}=\lambda(g) \text{ for some left bounded function } g \in \L^2(G)\\
+\infty & \text{otherwise}
\end{cases}.
$$

By \cite[Proposition 2.9 p.~129]{Haa78b} (see also \cite[Theorem 7.2.7 p.~236]{Ped79}), the canonical left ideal $\mathfrak{n}_{\tau_G}=\big\{x \in \VN(G)\ : \  \tau_G(x^*x) <\infty\big\}$ is given by
$$
\mathfrak{n}_{\tau_G}
=\big\{\lambda(g)\ :\ g \in \L^2(G)\text{ is left bounded}\big\}.
$$
%If $\xi \in L^2(G)$ is left bounded, by \cite[(2)]{Com}, we have the ``
%\begin{equation}
%\label{Formule-Plancherel}
%\tau_G\big(\lambda(\xi)^*\lambda(\xi)\big)=\norm{\xi}_{L^2(G)}^2.
%\end{equation}
Recall that $\mathfrak{m}_{\tau_G}^+$ denotes the set $\big\{x \in \VN(G)^+ : \tau_G(x)<\infty\big\}$ and that $\mathfrak{m}_{\tau_G}$ is the complex linear span of $\mathfrak{m}_{\tau_G}^+$, which is a two-sided ideal of the group von Neumann algebre $\VN(G)$. By \cite[Proposition 2.9 p.~129]{Haa78b} and \cite[Proposition p.~280]{Str81}, we have 
$$
\mathfrak{m}_{\tau_G}^+
=\big\{\lambda(g) : g \in \L^2(G) \text{ continuous and left bounded}, \ \lambda(g)\geq 0\big\}.
$$
\begin{comment}
and if $\lambda(\eta) \in \mathfrak{m}_{\tau_G}^+$ then $\eta$ is almost everywhere equal to a unique continuous function $\xi$ and $\tau_G(\lambda(f))=\xi(e)$ where $e$ denotes the identity element of $G$.
\end{comment} 

By \cite[Proposition 7.2.8 p.~237]{Ped79},
%ou\cite[p.~125]{Haa78b}
the Plancherel weight $\tau_G$ on the von Neumann algebra $\VN(G)$ is tracial if and only if the locally compact group $G$ is unimodular, which means that the left Haar measure of $G$ and the right Haar measure of $G$ coincide. Now, in the sequel, we suppose that the locally compact group $G$ is unimodular.

We will use the involution $f^*(t) \ov{\mathrm{def}}{=} \ovl{f(t^{-1})}$. By \cite[Theorem 4 p.~530]{Kun58}, if the functions $f,g \in \L^2(G)$ are left bounded then $f*g$ and $f^*$ are left bounded and we have 
\begin{equation}
\label{composition-et-lambda}
\lambda(f)\lambda(g)
=\lambda(f*g) 
\quad \text{and} \quad 
\lambda(f)^*=\lambda(f^*).
\end{equation}

If $f,g \in \L^2(G)$ it is well-known \cite[VIII pp.~39-40]{Bou04b} that the function $f*g$ is continuous and that we have $(f*g)(e)=(g*f)(e)=\int_G \check{g} f \d\mu_G$, where $e$ denotes the identity element of $G$ and where $\check{g}(s) \overset{\textrm{def}}= g(s^{-1})$. By \cite[(4) p.~282]{StZ75}, if $f,g \in \L^2(G)$ are left bounded, the operator $\lambda(g)^*\lambda(f)$ belongs to $\mathfrak{m}_{\tau_G}$ and we have the fundamental <<noncommutative Plancherel formula>>
\begin{equation}
\label{Formule-Plancherel}
\tau_G\big(\lambda(g)^*\lambda(f)\big)
=\langle g,f\rangle_{\L^2(G)},
%=(g^**f)(e_G).
\quad \text{which gives} \quad 
\tau_G\big(\lambda(g)\lambda(f)\big)
=\int_G \check{g} f \d\mu_G
=(g*f)(e).
\end{equation}
In particular, this formula can be used with any functions belonging to the space $\L^1(G) \cap \L^2(G)$.

If we introduce the subset $
%\begin{equation}
%\label{Ce-de-G}
\C_e(G)
\ov{\mathrm{def}}{=} \Span\big\{g^* * f : g,f \in \L^2(G)\text{ left bounded}\big\}
$ 
of the space $\C(G)$ considered in \cite[p.~238]{Ped79}, then we have
\begin{equation}
\label{Def-mtauG}
\mathfrak{m}_{\tau_G}
=\lambda(\C_e(G)).
\end{equation}
In this context, $\tau_G$ can be interpreted as the functional that evaluates functions of $\C_e(G)$ at the identity element $e_G$. While the formula $\tau_G(\lambda(h)) = h(e)$ appears meaningful for every function $h$ in $\C_c(G)$, we caution the reader that, in general, it is not true that $\lambda(\C_c(G)) \subset \mathfrak{m}_{\tau_G}$. Unfortunately, this misconception is frequently encountered in the literature.

\paragraph{Noncommutative $\L^p$-spaces}
In this paper, we focus on noncommutative $\L^{p}$-spaces associated to semifinite von Neumann algebras. Let $\cal{M}$ be a semifinite von Neumann algebra equipped with a normal semifinite faithful trace $\tau$. Let $\cal{S}^{+}$ be the set of all $x \in \cal{M}_{+}$ such that $\tau(\supp(x))<\infty$, where $\supp(x)$ denotes the support of $x$. Let $\cal{S}$ be the linear span of $\cal{S}^{+}$, then $\cal{S}$ is weak* dense $*$-subalgebra of $\cal{M}$. 

Suppose that $1 \leq p < \infty$. For any $x \in \cal{S}$, the operator $\vert x\vert^p$
belongs to $\cal{S}_+$ and we set
\begin{equation}
\label{Def-norm-Lp}
\norm{x}_{\L^p(\cal{M})}
\ov{\mathrm{def}}{=} \bigl(\tau(\vert x\vert^p)\bigr)^{\frac{1}{p}}.
\end{equation}
Here $\vert x \vert \ov{\mathrm{def}}{=}(x^*x)^{\frac{1}{2}}$ denotes the modulus of $x$. It turns out that $\norm{\cdot}_{\L^p(\cal{M})}$ is a norm on $\cal{S}$. By definition, the noncommutative $\L^p$-space $\L^p(\cal{M})$ associated with $(\cal{M},\tau)$ is the completion of $(\cal{S},\norm{\cdot}_{\L^p(\cal{M})})$. For convenience, we also set $\L^{\infty}(\cal{M}) \ov{\mathrm{def}}{=} \cal{M}$ equipped with its operator norm. Note that by definition, $\L^p(\cal{M}) \cap \cal{M}$ is dense in $\L^p(\cal{M})$ for any $1 \leq p < \infty$. See \cite{PiX03} for more information on noncommutative $\L^{p}$-spaces.

%Assume that $\cal{M} \subset \cal{B}(H)$ acts on some Hilbert space $H$. It will be fruitful to also have a description of the elements of $\L^p(\cal{M})$ as (possibly unbounded) operators on $H$. Let
%$\cal{M}'\subset \cal{B}(H)$ denote the commmutant of $\cal{M}$. We say that a closed and densely defined operator $x$ on $H$ is affiliated with $\cal{M}$ if $x$ commutes with any unitary of $\cal{M}'$. Then we say that an affiliated operator $x$ is measurable (with respect to the trace $\tau$) provided that there is a positive integer $n \geq 1$ such that $\tau(1-p_n)<\infty$, where $p_n = \chi_{[0,n]}(\vert x\vert)$ is the projection associated to the indicator function of $[0,n]$ in the Borel functional calculus of $\vert x \vert$. It turns out that the set $\L^0(\cal{M})$ of all measurable operators is a $*$-algebra (see e.g. \cite{Terp} for a precise definition of the sum and product on $\L^0(\cal{M})$). Indeed, this $*$-algebra has a lot of remarkable stability properties. First for any $x$ in $\L^0(\cal{M})$ and any $0<p<\infty$, the operator $\vert x\vert^p =(x^*x)^{\frac{p}{2}}$ belongs to $L^0(\cal{M})$. Second, let $\L^0(\cal{M})_+$ be the positive part of $\L^0(\cal{M})$, that is, the set of all selfadjoint positive operators in $\L^0(\cal{M})$. Then  the trace $\tau$ extends to a positive tracial functional on $L^0(\cal{M})_+$, still denoted by $\tau$, in such a way that for any $0<p<\infty$, we have
%$$
%L^p(\cal{M})
%=\bigl\{x \in \L^0(\cal{M})\, :\, \tau(\vert x \vert^p)<\infty\bigr\},
%$$
%equipped with $\norm{x}_p = (\tau (\vert x \vert^p))^{\frac{1}{p}}$. 

Furthermore, the trace $\tau$ uniquely extends to a bounded linear functional on the Banach space $\L^1(\cal{M})$, still denoted by $\tau$. Actually, we have 
\begin{equation}
\label{trace-continuity}
\vert\tau(x)\vert 
\leq \norm{x}_{\L^1(\cal{M})}, 
\quad x \in \L^1(\cal{M}).
\end{equation}
%For any $0 < p \leq \infty$ and any $x \in \L^p(\M)$, the adjoint operator $x^*$ belongs to $L^p(\cal{M})$ as well, with $\norm{x^*}_p = \norm{x}_p$. Clearly, we also have that $x^*x \in \L^{\frac{p}{2}}(\cal{M})$ and $\vert x\vert \in \L^p(\cal{M})$, with $\norm{\,\vert x\vert\,}_p = \norm{x}_p$. We let $L^p(\cal{M})_+ =\L^0(\cal{M})_+ \cap \L^p(\cal{M})$ denote the positive part of $L^p(\cal{M})$. The space $L^p(\M)$ is spanned by $L^p(\cal{M})_+$.
Recall the noncommutative H\"older's inequality. If $1 \leq p,q,r \leq \infty$ satisfy $\frac{1}{r}=\frac{1}{p}+\frac{1}{q}$ then
\begin{equation}
\label{Holder}
\norm{xy} _{\L^r(\cal{M})}
\leq \norm{x}_{\L^p(\cal{M})} \norm{y}_{\L^q(\cal{M})},\qquad x\in \L^p(\cal{M}), y \in \L^q(\cal{M}).
\end{equation}
%Conversely for any $z \in \L^r(\cal{M})$, by using the same ideas than the proof of \cite[Theorem 4.20]{Ray03} there exist $x \in \L^p(\cal{M})$ and $y \in \L^q(\cal{M})$ such that 
%\begin{equation}
%\label{inverse-Holder}
%z
%=xy
%\quad \text{with} \quad 
%\norm{z} _{\L^r(\cal{M})}
%=\norm{x}_{\L^p(\cal{M})} \norm{y}_{\L^q(\cal{M})}.
%\end{equation}
For any $1 \leq p < \infty$, let $p^* \ov{\mathrm{def}}{=} \frac{p}{p-1}$ be the conjugate number of $p$. Applying \eqref{Holder} with $q=p^*$ and $r=1$ together with \eqref{trace-continuity}, we obtain a linear map $\L^{p^*}(\cal{M}) \to (\L^p(\cal{M}))^*$, $y \mapsto \tau(xy)$, which induces an isometric isomorphism
\begin{equation}
%\label{2dual1}
(\L^p(\cal{M}))^* 
=\L^{p^*}(\cal{M}),\qquad 1 \leq p <\infty,\quad \frac{1}{p}
+\frac{1}{p^*}
=1.
\end{equation}
In particular, we may identify the Banach space $\L^1(\cal{M})$ with the unique predual $\cal{M}_*$ of the von Neumann algebra $\cal{M}$. 

%Finally, for any positive element in $\L^p(\cal{M})$, recall that there exists a positive element $y \in \L^{p^*}(\cal{M})$ such that
%\begin{equation}
%\label{dual-by-positive}
%\norm{y}_{\L^{p^*}(\cal{M})}=1
%\quad \text{and} \quad
%\norm{x}_{\L^p(\cal{M})}
%=\tau(xy).
%\end{equation}

\paragraph{Operator theory} Suppose that $1 \leq p < \infty$. Let $T \co \L^p(\cal{M}) \to \L^p(\cal{M})$ be any bounded operator. We will denote by $T^{*}$ the adjoint of $T$ defined by
$$
\tau(T(x)y) 
=\tau(xT^*(y)),\qquad x\in \L^p(\cal{M}), y\in \L^{p^*}(\cal{M}).
$$
For any $1 \leq p \leq \infty$ and any $T \co \L^p(\cal{M}) \to \L^p(\cal{M})$, we can consider the map $T^{\circ} \co \L^p(\cal{M}) \to \L^p(\cal{M})$ defined by
\begin{equation}
\label{2circ}
T^\circ(x) 
\ov{\mathrm{def}}{=}  T(x^{*})^{*},\qquad x \in \L^p(\cal{M}).
\end{equation}
If $p=2$ and if we denote by $T^{\dag} \co \L^2(\cal{M}) \to \L^2(\cal{M})$ the adjoint of $T\co \L^2(\cal{M}) \to \L^2(\cal{M})$ in the usual sense of Hilbertian operator theory, that is
$$
\tau\bigl(T(x)y^{*}\bigr) 
=\tau\bigl(x(T^{\dag}(y))^{*}\bigr),\qquad x, y \in \L^2(\cal{M}),
$$
we see that
\begin{equation}
\label{2dual4}
T^{\dag} 
= T^{*\circ}.
\end{equation}
%In particular, the selfadjointness of $T \co \L^2(\cal{M})\to \L^2(\cal{M})$ means that $T^*=T^\circ$.

%%%%%%%%%%%%%%%%%%%%%%%%%%%%%%%%%%%%%%%%%%%%%%%%%%%%%%%%%%%%%%%%%%%%%%%%%%%%%%%%%%%%
\subsection{Overview of the method}
\label{Sec-approach}

Suppose that $1 \leq p \leq \infty$ and let $G$ be a locally compact group. In this section, we present an approach for obtaining some bounded projections $P_G^p \co \CB(\L^p(\VN(G))) \to\CB(\L^p(\VN(G)))$ onto the subspace $\mathfrak{M}^{p,\cb}(G)$ of completely bounded Fourier multipliers on $\L^p(\VN(G))$, beyond the case of discrete groups, for a suitable locally compact group. The methods are different from the ones of \cite{ArK23} and complement the results of this paper. If $G$ is a locally compact group, we will use the fundamental unitary $W \co \L^2 (G \times G) \to \L^2(G \times G)$ in $\cal{B}(\L^2(G)) \otvn \VN(G)$ and its inverse $W^{-1}$ defined in \cite[Example 2.2.10 p.~26]{Vae01} (see also \cite[Remark 5.16 p.~150]{Kus05}) by
\begin{equation}
\label{Def-fund-unitary}
(W\xi)(s,t)
\ov{\mathrm{def}}{=}  \xi(s,s^{-1}t), \quad (W^{-1}\xi)(s,t)= \xi(s,st),
\quad s,t \in G, \xi \in \L^2(G \times G).
\end{equation}
Before going into the details, let us shortly present the roadmap of the proof of results of Section \ref{Sec-Th-complementation}.

Suppose that the group $G$ is discrete and recall the well-known construction. Consider the coproduct $\Delta \co \VN(G) \to \VN(G) \otvn \VN(G)$, $\lambda_s \mapsto \lambda_s \ot \lambda_s$. This \textit{trace preserving} normal unital injective $*$-homomorphism extends to a completely positive isometric map $\Delta_p \co \L^p(\VN(G)) \to \L^p(\VN(G) \otvn \VN(G))$ for any $1 \leq p \leq \infty$. With the adjoint $(\Delta_{p^*})^* \co \L^p(\VN(G) \otvn \VN(G)) \to \L^p(\VN(G))$, the map $P_G^p \co \CB(\L^p(\VN(G))) \to \CB(\VN(G))$ defined by
\begin{equation}
\label{Projection-discrete-case}
P_G^p(T)
=(\Delta_{p^*})^* (\Id_{\L^p(\VN(G))} \ot T)\Delta_p, \quad T \in \CB(\L^p(\VN(G))) 
\end{equation}
is a contractive projection from the Banach space $\CB(\L^p(\VN(G)))$ onto the subspace $\mathfrak{M}^{p,\cb}(G)$ of completely bounded Fourier multipliers acting on $\L^p(\VN(G))$ which preserves the complete positivity (in the case $p=\infty$ replace $\CB(\L^p(\VN(G)))$ by the space $\CB_{\w^*}(\VN(G))$).

By \cite[p.~26]{Vae01} and \cite[p.~267]{Str74}, we can factorize\footnote{\thefootnote. Indeed, this factorization is the definition of the coproduct.} the coproduct as
$$
\Delta(x)
=W(x \ot 1)W^{-1}, \quad x \in \VN(G).
$$
If $u,v \in \VN(G)$, we can therefore rewrite the formula \eqref{Projection-discrete-case} as 
$$
\big\langle P_G^p(T)u,v \big\rangle_{\L^{p}(\VN(G)),\L^{p^*}(\VN(G))}
=\big\langle (\Id \ot T)\Delta_p(u),\Delta_{p^*}(v) \big\rangle_{\L^p,\L^{p^*}}
$$ 
and finally 
\begin{equation}
\label{Magic-equa-1}
\big\langle P_G^p(T)u,v \big\rangle
=\big\langle (\Id \ot T)(W (u \ot 1)W^{-1}),W (v\ot 1)W^{-1} \big\rangle_{\L^p,\L^{p^*}}. 
\end{equation}
%\mathfrak{m}_{\tau_G}

Now, if $G$ is a (second-countable unimodular) locally compact group and if $T \co \L^p(\VN(G)) \to \L^p(\VN(G))$ is again a completely bounded map, we wish to replace one or both units 1 of the formula \eqref{Magic-equa-1} by suitable sequences $(x_j)$ and $(y_j)$ of elements which approximate 1 in some sense. Actually, we start by replacing in \eqref{Magic-equa-1} the elements 1 by elements $x,y \in \L^1(\VN(G)) \cap \VN(G)$ and $u,v$ by elements $u \in S^p_G$ and $v \in S^{p^*}_G$. We will show that there exists a completely bounded Schur multiplier $M_{x,y,T} \co S^p_G \to S^p_G$ (replace the Schatten class $S^p_G$ by the von Neumann algebra $\cal{B}(\L^2(G))$ if $p=\infty$) such that
\begin{equation}
\label{MxyT}
\big\langle M_{x,y,T}(u),v\big\rangle_{S^p_G, S^{p^*}_G}
=\big\langle (\Id \ot T)(W (u \ot x)W^{-1}),W (v \ot y)W^{-1} \big\rangle_{S^p_G(\L^p(\VN(G))),S^{p^*}_G(\L^{p^*}(\VN(G)))}
\end{equation}
%where the sense of the right-hand side will be specified in the paper.
for any suitable elements $u \in S^p_G$ and $v \in S^{p^*}_G$. Note that $x \in \L^p(\VN(G))$, $y \in \L^{p^*}(\VN(G))$ and that $W,W^{-1} \in \cal{B}(\L^2(G)) \otvn \VN(G)$. Moreover, we will compute the symbol $\varphi_{x,y,T}$, belonging to $\L^\infty(G \times G)$, of the Schur multiplier $M_{x,y,T}$ and we will get
\begin{equation}
\label{symbol-phixyT}
\varphi_{x,y,T}(s,t) 
=\tau_G\big(\lambda_ty \lambda_{s^{-1}} T(\lambda_s x \lambda_{t^{-1}}) \big) \quad s,t \in G.
\end{equation}
In the particular case of \textit{finite} groups, these assertions are straightforward and we refer to the end of this section for a short proof of \eqref{MxyT} and \eqref{symbol-phixyT}.

For the case of a locally compact group, this step unfortunately uses a painful approximation procedure described in Section \ref{Mappings} relying on a sequence $(M_{\phi_n})$ of completely bounded Fourier multipliers $M_{\phi_n} \co \L^p(\VN(G)) \to \L^2(\VN(G))$ which allows us to consider the completely bounded maps $M_{\phi_n}T \co \L^p(\VN(G)) \to \L^2(\VN(G))$ in order to reduce the problem to the level $p=2$.

We therefore obtain a map $P_{x,y} \co \CB(\L^p(\VN(G))) \to \CB(S^p_G)$, $T \mapsto M_{x,y,T}$ and it is easy to check that this map preserves the complete positivity. Introducing suitable sequences $(x_j)$ and $(y_j)$ of elements in $\L^1(\VN(G)) \cap \VN(G)$ which approximate the element 1 we obtain a sequence $(P_j)$ of linear maps $P_j \ov{\mathrm{def}}{=} P_{x_j,y_j} \co \CB(\L^p(\VN(G))) \to \CB(S^p_G)$. One of the difficulties in this area is to construct suitable sequences with the chosen assumptions on the group $G$.

%Note that one of the chosen elements $x_j$ and $y_j$ has its Fourier support close to $\{e\}$ where $e$ is the identity element of $G$.

Essentially, in the sequel we capture a cluster point of the bounded family $(P_j)$ and we obtain a bounded map $P^{(1)} \co \CB(\L^p(\VN(G))) \to \CB(S^p_G)$. Each map $P^{(1)}(T)$ is a completely bounded Schur multiplier.

%We will see that the symbol of $P_{j}(T)$ is 
%\[ 
%\phi_{j,T}(s,t) 
%= \tau_G\big(y_j \lambda_{s^{-1}} T(\lambda_s x_j \lambda_{t^{-1}}) \lambda_t\big), \quad s,t \in G. 
%\]

\paragraph{Case where $G$ is inner amenable and $p=\infty$}
With a \textit{suitable} choice of the sequences $(x_j)$ and $(y_j)$ provided by the inner amenability of $G$, the map $P^{(1)} \co \CB_{\w^*}(\VN(G)) \to \CB(\cal{B}(\L^2(G))$ is \textit{contractive} and the Schur multiplier $P^{(1)}(T) \co \cal{B}(\L^2(G)) \to \cal{B}(\L^2(G))$ is a \textit{Herz-Schur} multiplier for all weak* continuous completely bounded maps $T \co \VN(G) \to \VN(G)$. So, we can see the linear map $P^{(1)}$ as a map $P^{(1)} \co \CB_{\w^*}(\VN(G)) \to \mathfrak{M}^{\infty,\HS}_G=\mathfrak{M}^{\infty,\cb,\HS}_G$. Now, it suffices to identify (completely) bounded Herz-Schur multipliers acting on the space $\cal{B}(\L^2(G))$ isometrically with completely bounded Fourier multipliers acting on $\VN(G)$, while preserving the complete positivity. This step is well-known \cite{BoF84} and true for any locally compact group $G$ \textit{without} amenability assumption. Denoting $I \co \mathfrak{M}^{\infty,\HS}_G \to \CB(\VN(G))$ the associated isometry with range $\mathfrak{M}^{\infty,\cb}(G)$, the final contractive projection will be $P_G^\infty \ov{\mathrm{def}}{=} I \circ P^{(1)}$.  

Indeed, in the case where $T = M_\phi \co \VN(G) \to \VN(G)$ is a Fourier multiplier we will prove that the symbol $\phi_{j,T}$ of the Schur multiplier $P_{j}(M_\phi) \co \cal{B}(\L^2(G)) \to \cal{B}(\L^2(G))$ is equal to the symbol $\phi^\HS \co (s,t) \mapsto \phi(st^{-1})$ for any $j$. By passing the limit, $P^{(1)}(M_\phi) = M_{\phi^\HS}$ and finally
$$
P_G^\infty(M_\phi)
=I \circ P^{(1)}(M_\phi)
= I\big(M_{\phi^\HS}\big) 
= M_\phi.
$$
So we obtain the property $(\kappa_\infty)$ of Definition \ref{Defi-tilde-kappa} for these groups with constant $\kappa_\infty(G)=1$.

\paragraph{Case where $G$ is finite-dimensional and amenable and simultaneous cases $p=1$ and $p=\infty$}
In the case where the group $G$ is in addition finite-dimensional and amenable, replacing the sequences $(x_j)$ and $(y_j)$ of the proof of the last case by new ones, we obtain linear maps $P_p^{(1)} \co \CB(\L^p(\VN(G))) \to \CB(S^p_G)$ for $p=1$ and $p=\infty$ (replace $\CB(\L^\infty(\VN(G)))$ by $\CB_{\w^*}(\VN(G))$ here and in the sequel), that we see as maps $P_p^{(1)} \co \CB(\L^p(\VN(G))) \to \mathfrak{M}^{p,\cb}_G$. The cost of this replacement of sequences is the non-contractivity of $P_p^{(1)}$ but we obtain the compatibility of the maps $P_\infty^{(1)}(T)$ and $P_1^{(1)}(T)$. For the construction of the sequences $(x_j)$ and $(y_j)$, our approach relies on the structure of locally compact groups from the solution to Hilbert's fifth problem which makes appear connected Lie groups in this context and the use of Carnot-Carath\'eodory metrics on connected Lie groups.

Now, we construct and use a contractive map $Q \co \mathfrak{M}^{p,\cb}_G \to \mathfrak{M}^{p,\cb,\HS}_G$ from the space $\mathfrak{M}^{p,\cb}_G$ of Schur multipliers onto the subspace of Herz-Schur multipliers which preserves the complete positivity and the Herz-Schur multipliers\footnote{\thefootnote. We can see $Q$ as a contractive projection $Q \co \mathfrak{M}^{p,\cb}_G \to \mathfrak{M}^{p,\cb}_G$ onto the subspace $\mathfrak{M}^{p,\cb,\HS}_G$ of completely bounded Herz-Schur multipliers.}. In this essentially folklore step, we need the amenability of the group $G$ in sharp contrast with our previous work \cite{ArK23}. Then put $P^{(2)}_p \ov{\mathrm{def}}{=} Q \circ P_p^{(1)} \co \CB(\L^p(\VN(G))) \to \mathfrak{M}^{p,\cb,\HS}_G$.

At present, it suffices with \cite{CaS15} to identify completely bounded Herz-Schur multipliers isometrically with completely bounded Fourier multipliers, preserving the complete positivity.  Denoting $I \co \mathfrak{M}^{p,\cb,\HS}_G \to \CB(\L^p(\VN(G)))$ the associated isometry\footnote{\thefootnote. Actually, is it showed in \cite{CaS15} that the map $I$ is a contraction (when $G$ is amenable), which is an isometry on a large subspace.} with range $\mathfrak{M}^{p,\cb}(G)$, the final contractive projection will be 
$$
P_G^p 
\ov{\mathrm{def}}{=} I \circ P^{(2)}_p 
= I \circ Q \circ P^{(1)}_p.
$$

In the case where $T = M_\phi$ is a Fourier multiplier, we will prove that the symbol $\phi_{j,T}$, element in $\L^\infty(G \times G)$, of the completely bounded Schur multiplier $P_{j}(T)$ converges to the symbol $\phi^\HS \co (s,t) \mapsto \phi(st^{-1})$ for the weak* topology of of the dual Banach space $\L^\infty(G \times G)$. From this, we deduce that the limit $P^{(1)}_p(M_\phi)$ of the sequence $(P_j(T))$ also admits the symbol $\phi^\HS$. We conclude that
$$
P_G^p(M_\phi)
=I \circ Q \circ P^{(1)}_p(M_\phi)
=I \circ Q\big(M_{\phi^\HS}\big)
=I \big(M_{\phi^\HS}\big)
=M_\phi.
$$
We conclude that we obtain the property $(\kappa)$ of Definition \ref{Defi-complementation-G} for these groups. For totally disconnected groups, the method gives the sharp result $\kappa(G)=1$.

\paragraph{Case where $G$ is amenable and $1 < p < \infty$ with $\frac{p}{p^*}$ being  rational.}
In the case where the group $G$ is amenable, using some sequences $(x_j)$ and $(y_j)$, we obtain a \textit{contractive} linear map $P_p^{(1)} \co \CB(\L^p(\VN(G))) \to \CB(S^p_G)$ which is better than the \textit{boundedness} of the previous case, but only for \textit{one} value of $p$. The method is similar to the previous case but we use \cite{CaS15} (see also \cite{NeR11}) instead of \cite{BoF84} to identify completely bounded Herz-Schur multipliers isometrically with completely bounded Fourier multipliers (which require the amenability of $G$ once again). 
%Consequently, we only obtain a map $P_G^p \co \CB(\L^p(\VN(G))) \to \CB(\L^p(\VN(G)))$ whose range is included in $\mathfrak{M}^{p,\cb}_G$ such that $P_G^p(M_\phi)=M_\phi$ if the symbol $\phi$ is continuous. 

%However, at the time of the writing, we obtain the convergence of the symbols $\phi_{j,T}$ the symbol $\phi^\HS \co (s,t) \mapsto \phi(st^{-1})$ only for \textit{continuous} symbols $\phi$.

%, that we see as a map $P_p^{(1)} \co \CB(\L^p(\VN(G)) \to \mathfrak{M}^{p,\cb}_G$.

\paragraph{Particular case of finite groups: proof of \eqref{MxyT} and \eqref{symbol-phixyT}}
If the group $G$ is finite and if $(e_i)$ is an orthonormal basis of the Hilbert space $\ell^2_G$ then \eqref{Def-fund-unitary} translates to
\begin{equation}
\label{W-discret}
W(e_t \ot e_r)
=e_t \ot e_{tr}, \quad 
W^{-1}(e_t \ot e_r)
=e_t \ot e_{t^{-1}r}, \quad t,r \in G.
\end{equation}
For any $i,j,s,t,u \in G$, we have
\begin{align*}
\MoveEqLeft
W (e_{st} \ot \lambda_u)W^{-1}(e_i \ot e_j)            
\ov{\eqref{W-discret}}{=} W (e_{st} \ot \lambda_u)(e_i \ot e_{i^{-1}j}) 
=W\big(e_{st}e_i \ot \lambda_u(e_{i^{-1}j})\big) \\
&= \delta_{t=i}W\big(e_s \ot e_{ui^{-1}j})\big) 
\ov{\eqref{W-discret}}{=} \delta_{t=i} e_s \ot e_{sut^{-1}j}.
\end{align*} 
Hence in $\cal{B}(\ell^2(G)) \otvn \VN(G)$, we have
\begin{equation}
\label{calcul-890}
W (e_{st} \ot \lambda_u)W^{-1}
= e_{st} \ot \lambda_{sut^{-1}}.
\end{equation}
We deduce that 
\begin{equation}
\label{Equa-456}
(\Id \ot T)\big(W (e_{st} \ot \lambda_u)W^{-1}\big)            
=e_{st} \ot T(\lambda_{sut^{-1}}).
\end{equation}
%Moreover, we have
%\begin{equation}
%\label{Divers-33}
%(W (e_{ij} \ot \lambda_r)W^{-1})^{T \ot R}
%\ov{\eqref{calcul-890}}{=} (e_{ij} \ot \lambda_{irj^{-1}})^{T \ot R}
%=e_{ji} \ot \lambda_{jr^{-1}i^{-1}}
%\end{equation}
%where $x^T$ denotes the transpose of $x$. We infer that
%\begin{align*}
%\MoveEqLeft
%(\tr \ot \tau_G)\big[(\Id \ot T)\big(W (e_{st} \ot \lambda_u)W^{-1}\big)  (W (e_{ij} \ot \lambda_{r}W^{-1})^{T \ot R} \big] \\
%&\ov{\eqref{Equa-456}\eqref{Divers-33}}{=} (\tr \ot \tau_G)\big[(e_{st} \ot T(\lambda_{sut^{-1}}))  (e_{ji} \ot \lambda_{jr^{-1}i^{-1}})  \big]  \\
%&=\tr(e_{st}e_{ji}) \tau_G\big(T(\lambda_{sut^{-1}})\lambda_{jr^{-1}i^{-1}}  \big)
%=\delta_{s=i}\delta_{t=j}\tau_G\big(\lambda_{j}R(\lambda_{r})\lambda_{i^{-1}}T(\lambda_{sut^{-1}})  \big).
%\end{align*} 
%By linearity, we deduce on the one hand for any $x \in \L^p(\VN(G))$ and any $y \in \L^{p^*}(\VN(G))$
%$$
%(\tr \ot \tau_G)\big[(\Id \ot T)\big(W (e_{st} \ot x)W^{-1}\big)  (W (e_{ij} \ot y)W^{-1})^{T \ot R} \big]
%=\delta_{s=i}\delta_{t=j}\tau_G\big(\lambda_{j}R(y)\lambda_{i^{-1}}T(\lambda_{s} x \lambda_{t^{-1}})   \big)
%$$
We infer that
\begin{align*}
\MoveEqLeft
(\tr \ot \tau_G)\big[(\Id \ot T)\big(W (e_{st} \ot \lambda_u)W^{-1}\big)  (W (e_{ij} \ot \lambda_{r}W^{-1}) \big] \\
&\ov{\eqref{Equa-456}\eqref{calcul-890}}{=} (\tr \ot \tau_G)\big[(e_{st} \ot T(\lambda_{sut^{-1}}))  (e_{ij} \ot \lambda_{irj^{-1}})  \big]  \\
&=\tr(e_{st}e_{ij}) \tau_G\big(T(\lambda_{sut^{-1}})\lambda_{irj^{-1}}  \big)
=\delta_{t=i}\delta_{s=j}\tau_G\big(\lambda_{i}\lambda_{r}\lambda_{j^{-1}}T(\lambda_{sut^{-1}})  \big).
\end{align*} 
By linearity, we deduce on the one hand for any $x \in \L^p(\VN(G))$ and any $y \in \L^{p^*}(\VN(G))$
$$
(\tr \ot \tau_G)\big[(\Id \ot T)\big(W (e_{st} \ot x)W^{-1}\big)  (W (e_{ij} \ot y)W^{-1}) \big]
=\delta_{t=i}\delta_{s=j}\tau_G\big(\lambda_{i}y\lambda_{j^{-1}}T(\lambda_{s} x \lambda_{t^{-1}})   \big)
$$
On the other hand, if we consider the Schur multiplier $M_{x,y,T} \co S^p_G \to S^p_G$ with symbol \eqref{symbol-phixyT}, we have
\begin{align*}
\MoveEqLeft
\big\langle M_{x,y,T}(e_{st} ),e_{ij}\big\rangle_{S^p_G, S^{p^*}_G}            
\ov{\eqref{symbol-phixyT}}{=} \delta_{t=i}\delta_{s=j} \tau_G\big(\lambda_ty \lambda_{s^{-1}} T(\lambda_s x \lambda_{t^{-1}}) \big).
\end{align*}
%\textbf{AVEC ETOILE:}
%Moreover, we have
%\begin{equation}
%\label{Divers-33-bis}
%(W (e_{ij} \ot \lambda_r)W^{-1})^*
%\ov{\eqref{calcul-890}}{=} (e_{ij} \ot \lambda_{irj^{-1}})^*
%=e_{ji} \ot \lambda_{jr^{-1}i^{-1}}
%\end{equation}
%where $x^T$ denotes the transpose of $x$. We infer that
%\begin{align*}
%\MoveEqLeft
%(\tr \ot \tau_G)\big[(\Id \ot T)\big(W (e_{st} \ot \lambda_u)W^{-1}\big)  (W (e_{ij} \ot \lambda_{r}W^{-1})^* \big] \\
%&\ov{\eqref{Equa-456}\eqref{Divers-33-bis}}{=} (\tr \ot \tau_G)\big[(e_{st} \ot T(\lambda_{sut^{-1}}))  (e_{ji} \ot \lambda_{jr^{-1}i^{-1}})  \big]  \\
%&=\tr(e_{st}e_{ji}) \tau_G\big(T(\lambda_{sut^{-1}})\lambda_{jr^{-1}i^{-1}}  \big)
%=\delta_{s=i}\delta_{t=j}\tau_G\big(\lambda_{j}\lambda_{r}^*\lambda_{i^{-1}}T(\lambda_{sut^{-1}})  \big).
%\end{align*} 
%By linearity, we deduce on the one hand for any $x \in \L^p(\VN(G))$ and any $y \in \L^{p^*}(\VN(G))$
%$$
%(\tr \ot \tau_G)\big[(\Id \ot T)\big(W (e_{st} \ot x)W^{-1}\big)  (W (e_{ij} \ot y)W^{-1})^{T \ot R} \big]
%=\delta_{s=i}\delta_{t=j}\tau_G\big(\lambda_{j}y^*\lambda_{i^{-1}}T(\lambda_{s} x \lambda_{t^{-1}})   \big)
%$$
%\textbf{AVEC RIEN:}
%We denote by $R \co \VN(G) \to \VN(G)$, $x \mapsto x^R$ the antipode. Recall that $R(\lambda_s)=\lambda_{s^{-1}}$ for any $s \in G$. 

\begin{remark} \normalfont
Note that with $x=y=1$, the Schur multiplier $M_{x,y,T}$ is a Herz-Schur multiplier. See Section \ref{Sec-Herz-Schur} for a generalization of this crucial observation.
\end{remark}

%\begin{remark} \normalfont
%
%\end{remark}

%%%%%%%%%%%%%%%%%%%%%%%%%%%%%%%%%%%%%%%%%%%%%%%%%%%%%%%%%%%%%%%%%%%%%%%%%%%%%%%%%%%%%%%%%%%%%%%%%%
\subsection{Step 1: construction of the maps $P_j(T)$}
\label{Mappings}
% by \cite{Pis98}\cite[(19) page 147]{KaR97}

In this section, we establish \eqref{MxyT} and \eqref{symbol-phixyT}. We caution the reader that while this part is technically involved, the underlying idea is quite simple. Specifically, we reduce the computation to the case $p = 2$, where Parseval's identity can be applied.

Let $G$ be a unimodular locally compact group. We denote by $\tr_G$ and $\tau_G$ the canonical traces $\tr_G$ of the von Neumann algebras $\cal{B}(\L^2(G))$ and $\VN(G)$. Suppose that $G$ is second-countable and fix an orthonormal basis $(e_i)$ of the Hilbert space $\L^2(G)$ such that each function $e_i$ is continuous with compact support\footnote{\thefootnote. To demonstrate the existence of such an orthonormal basis, consider a sequence of continuous functions with compact support that is dense in $\L^2(G)$, and apply the Gram-Schmidt procedure.}. Note that by \cite[p.~40]{BlM04} we have a canonical identification $\cal{B}(\L^2(G)) \otvn \VN(G)=\M_\infty(\VN(G))$. That means that an element $X$ belonging to the von Neumann tensor product $\cal{B}(\L^2(G)) \otvn \VN(G)$ identifies to a matrix $[x_{ij}]$ with entries in the von Neumann algebra $\VN(G)$. 
%For any $i,j$, we have
%$$
%x_{ij}
%=(\tr_G \ot \Id)\big(X(e_{ij}^* \ot 1) \big).
%$$ 
For any $h \in \L^2(G)$ and any integer $k$, note that in $\L^2(G)$
\begin{equation}
\label{eval}
\big(x_{kk}(h)\big)(w)
=\int_G \big(X(e_k \ot h)\big)(s,w) \ovl{e_k(s)} \d \mu_G(s), \quad h \in \L^2(G), \text{ a.e. }w \in G.
%= \big[X(e_i \ot h)\big]_i\, , \quad \text{(``ieme coordonnee'')}.
\end{equation}
If $1 \leq p <\infty$, we have by \cite{Pis98} a similar isometry $\L^p(\cal{B}(\L^2(G)) \otvn \VN(G))=S^p_G(\L^p(\VN(G)))$. Moreover, if $X$ belongs to the intersection $\L^1(\cal{B}(\L^2(G)) \otvn \VN(G)) \cap \big[\cal{B}(\L^2(G)) \otvn \VN(G)\big]$ we have $x_{kk} \in \L^1(\VN(G)) \cap \VN(G)$ for any integer $k$ and 
\begin{equation}
\label{trace-Xn}
(\tr_G \ot \tau_G)(X) 
= \sum_{k=1}^{\infty} \tau_G(x_{kk}).
\end{equation}
In the next result, we use the operator $W$ in $\cal{B}(\L^2(G)) \otvn \VN(G)$ and its inverse from \eqref{Def-fund-unitary}.

\begin{lemma}
\label{lem-referee-proof-step-1-calcul-du-symbole-coefficients-L2}
Let $G$ be a second-countable unimodular locally compact group.
\begin{enumerate}
\item
Let $\phi \in \L^2(G \times G)$ such that $K_\phi$ belongs to $S^1_G$ and  $x \in \L^1(\VN(G)) \cap \VN(G)$. Then $W(K_\phi \ot x)W^{-1}$ belongs to $\L^1(\cal{B}(\L^2(G)) \otvn \VN(G))$ and to $\cal{B}(\L^2(G)) \otvn \VN(G)$. %If $x \in \mathfrak{m}_{\tau_G}$. Then $W(K_\phi \ot x)W^{-1}$ belongs to $\mathfrak{m}_{\tr \ot \tau_G}$. If

\item If $f$ belongs to the space $\C_c(G)$ and if $g$ belongs to the space $\C_c(G)*\C_c(G)$, we have for any integers $i,j$
\begin{align}
\label{equ-referee-proof-step-1-calcul-du-symbole-coefficients-L2}
\MoveEqLeft
(\tr_G \ot \tau_G)\big[W( K_\phi \ot \lambda(g) )W^{-1} \cdot ( e_{ij}^* \ot \lambda(f))\big] \\
& =\int_G \int_G \phi(s,t)  \tau_G\big[\lambda_s \lambda(g) \lambda_{t^{-1}} \lambda(f)\big]  \ovl{e_i(s)}e_j(t)  \d \mu_G(s)\d \mu_G(t). \nonumber
\end{align}
\end{enumerate}
\end{lemma}

\begin{proof}
1. The element $K_\phi \ot x$ belongs to $S^1_G \ot [ \L^1(\VN(G)) \cap \VN(G) ]$, hence to the space $\L^1(\cal{B}(\L^2(G)) \otvn \VN(G)) \cap \big[\cal{B}(\L^2(G)) \otvn \VN(G)\big]$. Then the claim follows since $W$ and $W^{-1}$ belong to the space $\cal{B}(\L^2(G)) \otvn \VN(G)$ and since $\L^1(\cal{B}(\L^2(G)) \otvn \VN(G)) \cap \big[\cal{B}(\L^2(G)) \otvn \VN(G)\big]$ is an ideal of the von Neumann algebra $\cal{B}(\L^2(G)) \otvn \VN(G)$.

2. By the first part, observe that the element $X \ov{\mathrm{def}}{=} W( K_\phi \ot \lambda(g) )  W^{-1} \cdot \big(e_{ij}^* \ot \lambda(f)\big)$ belongs to the intersection $\L^1(\cal{B}(\L^2(G)) \otvn \VN(G)) \cap \big[\cal{B}(\L^2(G)) \otvn \VN(G)\big]$. According to \eqref{trace-Xn}, we have 
$$
(\tr_G \ot \tau_G)\big[W( K_\phi \ot \lambda(g) )W^{-1} (e_{ij}^* \ot \lambda(f))\big] 
\ov{\eqref{trace-Xn}}{=} \sum_{k=1}^{\infty} \tau_G(x_{kk})
$$ 
Now, we want to compute $\tau_G(x_{kk})$ with \eqref{eval}. If $k \neq i$, we have 
$$
X(e_k \ot h) 
= W(K_\phi \ot \lambda(g))W^{-1}\big(e_{ji} \ot \lambda(f)\big)(e_k \ot h) 
= 0.
$$ 
Hence $x_{kk}=0$ in this case and therefore $\tau_G(x_{kk})=0$. Thus, we only need to consider $k = i$ in the sequel. Then replacing $r$ by $tv^{-1}s^{-1}r$ in the last equation for $h \in \C_c(G)$
\begin{align}
\MoveEqLeft
\label{Eq-10987}
\big(X(e_i \ot h)\big)(s,w) 
= \big(W(K_\phi \ot \lambda(g))W^{-1}(e_{ji} \ot \lambda(f) )(e_i \ot h)\big)(s,w) \\
&=\big(W(K_\phi \ot \lambda(g))W^{-1}(e_j \ot \lambda(f)(h))\big)(s,w) \nonumber \\
&\ov{\eqref{Def-fund-unitary}}{=} \big((K_\phi \ot \lambda(g))W^{-1}(e_j \ot \lambda(f)h)\big)(s,s^{-1}w) \nonumber\\
& \ov{\eqref{Def-de-Kf} \eqref{Convolution-formulas} }{=} \int_G \int_G \phi(s,t) g(v) W^{-1}(e_j \ot \lambda(f)h)(t,v^{-1}s^{-1}w) \d \mu_G(t) \d \mu_G(v) \nonumber\\
& \ov{\eqref{Def-fund-unitary}}{=} \int_G \int_G \phi(s,t) g(v)(e_j \ot \lambda(f)h)(t,tv^{-1}s^{-1}w) \d \mu_G(t) \d \mu_G(v) \nonumber\\
& \ov{\eqref{Convolution-formulas}}{=} \int_G \int_G \int_G \phi(s,t)e_j(t) g(v) f(r)  h(r^{-1} t v^{-1}s^{-1}w) \d \mu_G(t) \d \mu_G(v) \d \mu_G(r)\nonumber\\
&=\int_G \int_G \int_G \phi(s,t) g(v) f(tv^{-1}s^{-1}r)  h(r^{-1}w) e_j(t)\d \mu_G(t) \d \mu_G(v) \d \mu_G(r). \nonumber
\end{align}
Hence for almost all $t \in G$
\begin{align*}
\MoveEqLeft
\big(x_{ii}(h)\big)(w) 
\ov{\eqref{eval}}{=}\int_G \big(X(e_i \ot h)\big)(s,w) \ovl{e_i(s)} \d \mu_G(s) \\
%& = \int_G \int_G \int_G \int_G \phi(s,t) g(v) \ovl{f(r)} e_j(t) h(rtv^{-1}s^{-1}w) \ovl{e_i(s)}  \d \mu_G(t) \d \mu_G(v) \d \mu_G(r) \d \mu_G(s) \\
&\ov{\eqref{Eq-10987}}{=} \int_G \int_G \int_G \int_G \phi(s,t) g(v) f(tv^{-1}s^{-1}r)  h(r^{-1}w) \ovl{e_i(s)}  e_j(t)\d \mu_G(t) \d \mu_G(v) \d \mu_G(r) \d \mu_G(s).
\end{align*}
So we obtain
\[ 
x_{ii} 
=  \int_G \int_G \int_G \int_G \phi(s,t) g(v) f(tv^{-1}s^{-1}r)  \ovl{e_i(s)}e_j(t) \lambda_{r^{-1}}  \d \mu_G(t) \d \mu_G(v) \d \mu_G(r) \d \mu_G(s) . 
\]
which identifies to the convolution operator $\lambda(k)$ where $k$ is the function defined by
\begin{equation}
\label{Function-k}
k(r)
\ov{\mathrm{def}}{=} \int_G \int_G \int_G \phi(s,t) g(v) f(tv^{-1}s^{-1}r) \ovl{e_i(s)} e_j(t) \d \mu_G(t) \d \mu_G(v) \d \mu_G(s).
\end{equation} 

%\textbf{It is left to the reader (ou a enlever si ca parait faux...) to show that $k$ belongs to $\C_c(G)*\C_c(G)$ according to the fact that the functions $g, e_j$ and $\ovl{e_i}$ have compact support. Par \eqref{Convolution-formulas}} \textbf{on veut}
%$$
%k(r)
%=\int_G F(rt^{-1})G(t) \d\mu_G(t)
%$$
%On a
%$$
%=\int_G \bigg(\int_G \int_G \phi(s,t) g(v) \check{f}(rs vt^{-1}) \ovl{e_i(s)} e_j(t)  \d \mu_G(v) \d \mu_G(s)\bigg)\d \mu_G(t)
%$$
%$v \to s^{-1}v$
%$$
%=\int_G \bigg(\int_G \int_G \phi(s,t) g(s^{-1}v) \check{f}(rv t^{-1}) \ovl{e_i(s)} e_j(t)  \d \mu_G(v) \d \mu_G(s)\bigg)\d \mu_G(t)
%$$
%$$
%=\int_G \bigg[\bigg(e_j(t)\int_G  \phi(s,t)  \ovl{e_i(s)} \d \mu_G(s)\bigg)\bigg(\int_G \check{f}(rs vt^{-1})g(v)\d \mu_G(v)  \bigg)\bigg]\d \mu_G(t)
%$$
We can easily evaluate the trace of the diagonal entry $x_{ii}$ which is an element in the space $\L^1(\VN(G)) \cap \VN(G)$. Indeed, replacing $v$ by $s^{-1}v$ in the second equality, we have
\begin{align*}
\MoveEqLeft
\tau_G(x_{ii}) 
=k(e) 
\ov{\eqref{Function-k}}{=}  \int_G \int_G \int_G \phi(s,t) g(v) f(tv^{-1}s^{-1}) \ovl{e_i(s)}e_j(t)   \d \mu_G(t) \d \mu_G(v) \d \mu_G(s) \\
&= \int_G \int_G \int_G \phi(s,t) g(s^{-1}v) f(tv^{-1}) \ovl{e_i(s)} e_j(t)  \d \mu_G(t) \d \mu_G(v) \d \mu_G(s) \\
%&= \int_G \int_G \phi(s,t) \bigg(\int_G g(s^{-1}vt) f(v) \d \mu_G(v)\bigg) \ovl{e_i(s)} e_j(t)  \d \mu_G(t) \d \mu_G(s) \\
%&=\int_G \int_G \phi(s,t) \bigg(\int_G g(s^{-1}v) f(vt^{-1}) \d \mu_G(v)\bigg) \ovl{e_i(s)} e_j(t)  \d \mu_G(t) \d \mu_G(s) \\
&=\int_G \int_G \phi(s,t) \bigg(\int_G g(s^{-1}v) f(tv^{-1}) \d \mu_G(v)\bigg) \ovl{e_i(s)} e_j(t)  \d \mu_G(t) \d \mu_G(s)\\
&\ov{\eqref{Formule-Plancherel}}{=}  \int_G \int_G \phi(s,t) \tau_G\big[\lambda_s \lambda(g) \lambda_{t^{-1}} \lambda(f)\big] \ovl{e_i(s)} e_j(t)  \d \mu_G(t) \d \mu_G(s).
 \end{align*}
\end{proof}

In a similar way to Lemma \ref{lem-referee-proof-step-1-calcul-du-symbole-coefficients-L2}, we have the following result.

\begin{lemma}
\label{lem-referee-proof-step-1-calcul-du-symbole-coefficients-L2-petite-extension}
Let $G$ be a second-countable unimodular locally compact group. Let $\phi \in \L^2(G \times G)$ such that $K_\phi \in S^1_G$ and $x \in \L^1(\VN(G)) \cap \VN(G)$. Suppose that $1 \leq p \leq 2$. Let $T \co \L^p(\VN(G)) \to \L^2(\VN(G))$ be a completely bounded operator. Then $(\Id \ot T)(W(K_\phi \ot x)W^{-1})$ belongs to $\L^2(\cal{B}(\L^2(G)) \otvn \VN(G))$. If $g \in \C_c(G) * \C_c(G)$ we have for any integers $i,j$ and any $f \in \C_c(G)$
\begin{align}
\MoveEqLeft
\label{Lp-L2}
(\tr_G \ot \tau_G)\big[(\Id \ot T)(W( K_\phi \ot \lambda(g) )W^{-1}) \cdot (e_{ij}^* \ot \lambda(f) )\big] \\
%&= (\tr_G \ot \tau_G)(W(K_\phi \ot \lambda(g))W^{-1} (e_{ij}^* \ot (T^* \lambda(f) )^*)) \\
& =\int_G \int_G \phi(s,t) \tau_G\big[\lambda_s \lambda(g) \lambda_{t^{-1}} T^*(\lambda(f))\big] \ovl{e_i(s)}e_j(t)   \d \mu_G(s) \d \mu_G(t). \nonumber
\end{align}
\end{lemma}

\begin{proof}
According to Lemma \ref{lem-referee-proof-step-1-calcul-du-symbole-coefficients-L2}, the element $W (K_\phi \ot \lambda(g)) W^{-1}$ belongs to the space $\L^1(\cal{B}(\L^2(G)) \otvn \VN(G)) \cap \big[\cal{B}(\L^2(G)) \otvn \VN(G)\big]$, hence to the Banach space 
$$
\L^p(\cal{B}(\L^2(G)) \otvn \VN(G)) 
= S^p_G(\L^p(\VN(G))).
$$ 
By the complete boundedness of $T \co \L^p(\VN(G)) \to \L^2(\VN(G))$ and \cite[Lemma 1.7 p.~23]{Pis98}, we infer that the element $(\Id_{S^p_G} \ot T)(W (K_\phi \ot x)W^{-1})$ belongs to the space $S^p_G(\L^2(\VN(G)))$. Since $p \leq 2$, it belongs to the Banach space $S^2_G(\L^2(\VN(G))) = \L^2(\cal{B}(\L^2(G))\otvn \VN(G))$. We have immediately 
\begin{align*}
\MoveEqLeft
(\tr_G \ot \tau_G)\big[(\Id \ot T)(W( K_\phi \ot \lambda(g) )W^{-1}) \cdot (e_{ij}^* \ot \lambda(f))\big] \\
&= (\tr_G \ot \tau_G)\big[(W(K_\phi \ot \lambda(g))W^{-1}) \cdot (e_{ij}^* \ot T^* \lambda(f))\big].
\end{align*}
%where $R \co \VN(G) \to \VN(G)$, $\lambda_s\to \lambda_{s^{-1}}$ is the antipode.
Now, it suffices to show that \eqref{equ-referee-proof-step-1-calcul-du-symbole-coefficients-L2} holds for generic elements in $\L ^{p^*}(\VN(G))$ instead of $\lambda(f)$. This is indeed the case by density since both sides of \eqref{equ-referee-proof-step-1-calcul-du-symbole-coefficients-L2} are continuous as functions in $\lambda(f) \in \L^{p^*}(\VN(G))$.
\end{proof}

%\begin{remark} \normalfont
%The two previous lemmas are true without the assumption <<second-countable>>.
%\end{remark}

\begin{lemma}
\label{lem-referee-proof-step-1-cacul-du-symbol-approximation}
Let $G$ be a secound countable unimodular locally compact group. Suppose that $1 \leq p \leq 2$. Let $T \co \L^p(\VN(G)) \to \L^p(\VN(G))$ be a completely bounded map. There exists a sequence $(M_{\phi_n})$ of bounded Fourier multipliers $M_{\phi_n} \co \VN(G) \to \VN(G)$ such that $\phi_n \in \C_c(G)$, $\norm{\phi_n}_{\infty} \leq 1$, $M_{\phi_n} \co \L^p(\VN(G)) \to \L^2(\VN(G))$ is completely bounded, satisfying for any $g \in \C_c(G)$, any $\phi \in \C_c(G \times G)$ such that $K_\phi \in S^1_G$ and any sufficiently large $n$
$$
M_{\phi_n}(\lambda(g)) 
= \lambda(g) 
\quad \text{and} \quad
(\Id \ot M_{\phi_n})(W (K_\phi \ot \lambda(g)) W^{-1}) 
= W (K_\phi \ot \lambda(g) ) W^{-1}.
$$ 
%in $\L^1(\VN(G))$ and in $\L^1(\cal{B}(\L^2(G)) \otvn \VN(G))$.
%\item \textbf{Utilisé dans la suite ?} Suppose in addition that $G$ is amenable.
%Then there exists a sequence $(T_n)_n$ of operators such that $T_n \co \L^p(\VN(G)) \to \L^2(\VN(G))$ is completely bounded and $\langle T_n a , b \rangle \to \langle Ta , b \rangle $ as $n \to \infty$ for all $a,b \in \L^1(\VN(G)) \cap \VN(G)$, and
%$\langle( \Id \ot T_n )A , B \rangle \to \langle (\Id \ot T) A , B \rangle $ as $n \to \infty$ for all $A,B \in \L^1(\cal{B}(\L^2(G)) \otvn \VN(G)) \cap \cal{B}(\L^2(G)) \otvn \VN(G)$.
%\end{enumerate}
\end{lemma}

\begin{proof}
%1. 
Since the group $G$ is second-countable, we can consider a sequence $(K_n)$ of \textit{symmetric} compacts in $G$ such that for any compact $K$ of $G$, one has $K \subseteq K_n$ for sufficiently large enough $n$. By \cite[Proposition 2.3.2 p.~50]{KaL18}, for any $n$ there exists a function $\phi_n \co G \to \mathbb{C}$ which is a finite linear combination of continuous positive definite functions with compact support with
$0 \leq \phi_n \leq 1$ and $\phi_n(s) = 1$ for any $s \in K_n$. By essentially \cite[Proposition 5.4.9 p.~184]{KaL18} each function induces a completely bounded Fourier multiplier $M_{\phi_n} \co \VN(G) \to \VN(G)$. Furthermore, since $\phi_n \in \L^2(G)$ each map $M_{\phi_n} \co \L^1(\VN(G)) \to \L^2(\VN(G))$ is completely bounded by \cite[Remark 2.4 p.~899]{GJP17} and duality. Now, it suffices to interpolate with $M_{\phi_n} \co \L^2(\VN(G)) \to \L^2(\VN(G))$ to obtain a completely bounded Fourier multiplier $M_{\phi_n} \co \L^p(\VN(G)) \to \L^2(\VN(G))$.

Now, let $g \in \C_c(G)$ and $\phi \in \C_c(G \times G)$ such that $K_\phi \in S^1_G$.
Consider the compact $K \ov{\mathrm{def}}{=} \supp g$ and some compacts $L_1,L_2$ of $G$ such that $\supp \phi \subseteq L_1 \times L_2$ and let $L \ov{\mathrm{def}}{=} L_1 \cdot K \cdot L_2^{-1}$, which is also compact. 
Then for any sufficiently large enough $n$ such that $K \subseteq K_n$,
\[ 
M_{\phi_n}(\lambda(g)) 
= \lambda(\phi_n g) 
= \lambda(g).
\]
Moreover, consider some sufficiently large enough $n$ such that $L \subseteq K_n=\check{K}_n$. For any $s \in L_1$ and any $t \in L_2$, the element $\lambda_s \lambda(g) \lambda_{t^{-1}}$ has its Fourier support in $L_1 \cdot K \cdot L_2^{-1} = L$. Thus, $M_{\check{\phi}_n}(\lambda_s \lambda(g) \lambda_{t^{-1}}) = \lambda_s \lambda(g) \lambda_{t^{-1}}$. Then for any integers $i,j$ and any function $f \in \C_c(G)*\C_c(G)$ we have $M_{\phi_n}^*=M_{\check{\phi}_n}$. Hence
\begin{align*}
\MoveEqLeft
(\tr_G \ot \tau_G)\big[(\Id \ot M_{\phi_n})(W( K_\phi \ot \lambda(g) )W^{-1}) \cdot (e_{ij}^* \ot \lambda(f))\big]  \\
& \ov{\eqref{Lp-L2}}{=} \int_G \int_G \phi(s,t) \tau_G\big[\lambda_s \lambda(g) \lambda_{t^{-1}} M_{\phi_n}^*(\lambda(f))\big] \ovl{e_i(s)}e_j(t) \d \mu_G(s) \d \mu_G(t).  \\
& = \int_{L_1} \int_{L_2} \phi(s,t) \tau_G\big[ M_{\check{\phi}_n}(\lambda_s \lambda(g) \lambda_{t^{-1}}) \lambda(f) \big] \ovl{e_i(s)}  e_j(t) \d \mu_G(s)\d \mu_G(t)  \\
& = \int_{L_1} \int_{L_2} \phi(s,t) \tau_G\big[ \lambda_s \lambda(g) \lambda_{t^{-1}} \lambda(f) \big] \ovl{e_i(s)} e_j(t) \d \mu_G(s) \d \mu_G(t)  \\
& \ov{\eqref{equ-referee-proof-step-1-calcul-du-symbole-coefficients-L2}}{=} (\tr_G \ot \tau_G)\big[W( K_\phi \ot \lambda(g)  )W^{-1} \cdot(e_{ij}^* \ot \lambda(f))\big].
\end{align*}
By density of the $e_{ij}^* \ot \lambda(f)$'s, we infer that 
$$
(\Id \ot M_{\phi_n})(W( K_\phi \ot \lambda(g) )W^{-1}) 
= W( K_\phi \ot \lambda(g) )W^{-1}.
$$
\end{proof}

Recall that $(e_n)$ is an orthonormal basis of the Hilbert space $\L^2(G)$ such that each function $e_n$ is continuous with compact support. So the family $(\lambda(e_k))$ is an orthonormal basis of $\L^2(\VN(G))$ and $(e_{ij} \ot \lambda(e_k))_{i,j,k}$ is an orthonormal basis of the Hilbert space $\L^2(\cal{B}(\L^2(G)) \otvn \VN(G))$.

\begin{prop}
\label{prop-referee-proof-step-1-calcul-du-symbol-avec-coefficients-L2}
Let $G$ be a second-countable unimodular locally compact group. Suppose that $1 \leq p \leq \infty$.
Let $T \co \L^p(\VN(G)) \to \L^p(\VN(G))$ be a completely bounded operator (normal if $p = \infty$). Let $\phi,\psi \in \L^2(G \times G)$ such that $K_\phi,K_\psi \in S^1_G$, and $x,y \in \L^1(\VN(G)) \cap \VN(G)$. With the symbol
\begin{equation}
\label{Def-symbol-varphi-1}
\varphi_{x,y,T}(s,t) 
\ov{\mathrm{def}}{=} \tau_G\big(\lambda_ty \lambda_{s^{-1}} T(\lambda_s x \lambda_{t^{-1}}) \big).
\end{equation}
we have
\begin{equation}
\label{MxyT-bis}
\big\langle (\Id \ot T)(W (K_\phi \ot x)W^{-1}),W (K_\psi \ot y)W^{-1} \big\rangle_{S^p_G(\L^p(\VN(G))),S^{p^*}_G(\L^{p^*}(\VN(G)))}
=\big\langle M_{\varphi_{x,y,T}}(K_\phi),K_\psi\big\rangle_{S^p_G, S^{p^*}_G}.
\end{equation}
Finally, if $p = \infty$, the same holds for any $x \in \VN(G)$ and if $p = 1$, the same holds for any $y \in \VN(G)$.
\end{prop}

\begin{proof}
Note first that by a simple duality argument, we can suppose that $1 \leq p \leq 2$ and that the functions $\phi$ and $\psi$ belong to the space $\C_c(G \times G)$. 

We start with the case where the operator $T$ also induces a completely bounded map $T \co \L^p(\VN(G)) \to \L^2(\VN(G))$. Then by Lemma \ref{lem-referee-proof-step-1-calcul-du-symbole-coefficients-L2} and Lemma \ref{lem-referee-proof-step-1-calcul-du-symbole-coefficients-L2-petite-extension}, the elements $W (K_\psi \ot y)W^{-1}$ and $(\Id \ot T)(W (K_\phi \ot x)W^{-1})$ belong to the Hilbert space $\L^2(\cal{B}(\L^2(G)) \otvn \VN(G))$. So the left-hand side of \eqref{equ-prop-referee-proof-step-1-calcul-du-symbol-avec-coefficients-L2} is well-defined, and can be calculated with Parseval's formula and the orthonormal basis $(e_{ij} \ot \lambda(e_k))_{i,j,k}$. With Lemma \ref{lem-referee-proof-step-1-calcul-du-symbole-coefficients-L2-petite-extension}, we get
\begin{align*}
\MoveEqLeft
\big\langle (\Id \ot T)(W (K_\phi \ot x)W^{-1}),W (K_\psi \ot y)W^{-1} \big\rangle_{S^2_G(\L^2(\VN(G)))} \\
& = \sum_{i,j,k} (\tr_G \ot \tau_G)\big[(\Id \ot T)(W(K_\phi \ot x)W^{-1}) (e_{ij}^* \ot \lambda(e_k^*))\big] \\
&\times \ovl{(\tr_G \ot \tau_G)\big[W(K_\psi \ot y)W^{-1} (e_{ij}^* \ot \lambda(e_k^*))\big]} \\
&\ov{\eqref{Lp-L2}\eqref{equ-referee-proof-step-1-calcul-du-symbole-coefficients-L2} }{=} \sum_{i,j,k} \int_G \int_G \phi(s,t) \tau_G\big[ \lambda_s x \lambda_{t^{-1}}T^*(\lambda(e_k^*)) \big] \ovl{e_i(s)} e_j(t)  \d \mu_G(s) \d \mu_G(t) \\
& \times \ovl{\int_G \int_G \psi(s,t) \tau_G\big[\lambda_{s} y \lambda_{t^{-1}} \lambda(e_k^*) \big]  \ovl{e_i(s)} e_j(t)  \d \mu_G(s)\d \mu_G(t)}.
\end{align*}
Since the functions $(s,t) \mapsto \tau_G\big[ \lambda_s x \lambda_{t^{-1}}T^*(\lambda(e_k^*)) \big] \phi(s,t)$ and $(s,t) \mapsto \tau_G\big[\lambda_{s} y \lambda_{t^{-1}} \lambda(e_k^*) \big] \psi(s,t)$ belong to the Hilbert space $\L^2(G \times G)$, we can use the orthonormal basis $(\ovl{e_i} \ot e_j)_{i,j}$ of the space $\L^2(G \times G)$ and reduce in the previous expression the sum over $i,j$ and the integral over $s,t$, which then becomes
\[
\sum_k \int_G \int_G \phi(s,t) \tau_G\big[ \lambda_s x \lambda_{t^{-1}}T^*(\lambda(e_k^*)) \big] \ovl{\psi(s,t) \tau_G\big[\lambda_{s} y \lambda_{t^{-1}} \lambda(e_k^*) \big] } \d \mu_G(s) \d\mu_G(t).
\]
Recall that $T^{*\circ} \ov{\eqref{2dual4}}{=} T^{\dag}$ where $T^{\dag}$ is the hilbertian adjoint. We fix $s,t \in G$ and calculate using Parseval's identity in the fourth equality
\begin{align*}
\MoveEqLeft
\sum_k \tau_G\big[ \lambda_s x \lambda_{t^{-1}}T^*(\lambda(e_k^*)) \big] \ovl{\tau_G\big[\lambda_{s} y \lambda_{t^{-1}} \lambda(e_k^*) \big]}  
= \sum_k \big\langle \lambda_s x \lambda_{t^{-1}},(T^*(\lambda(e_k)^*))^* \big\rangle_{\L^2} \ovl{\big\langle  \lambda_{s} y \lambda_{t^{-1}}, \lambda(e_k)\big\rangle}_{\L^2} \\
&=\sum_k  \big\langle \lambda_s x \lambda_{t^{-1}},T^{\dag}(\lambda(e_k)) \big\rangle_{\L^2} 
  \ovl{\big\langle  \lambda_{s} y \lambda_{t^{-1}}, \lambda(e_k)\big\rangle_{\L^2}} 
	=\sum_k  \big\langle T(\lambda_s x \lambda_{t^{-1}}),\lambda(e_k) \big\rangle \
	\ovl{\big\langle  \lambda_{s} y \lambda_{t^{-1}}, \lambda(e_k)\big\rangle} \\
	&=\big\langle T(\lambda_s x \lambda_{t^{-1}}), \lambda_s y \lambda_{t^{-1}} \big\rangle_{\L^2(\VN(G))} 
	=\tau_G\big(\lambda_t y^* \lambda_{s^{-1}} T(\lambda_s x \lambda_{t^{-1}})\big).
\end{align*}
Note that $\norm{T(\lambda_s x \lambda_{t^{-1}})}_2 \leq C$ and $\norm{\lambda_s y \lambda_{t^{-1}}}_2 \leq C$ justifies that we can integrate over $s$ and $t$. Consequently, we obtain
\begin{align*}
\MoveEqLeft
\big\langle (\Id \ot T)(W (K_\phi \ot x)W^{-1}),W (K_\psi \ot y)W^{-1} \big\rangle_{S^2_G(\L^2(\VN(G))} \\
&=\int_{G \times G} \tau_G\big(\lambda_t y^* \lambda_{s^{-1}} T(\lambda_s x \lambda_{t^{-1}})\big)\phi(s,t) \ovl{\psi(s,t)} \d \mu_G(s) \d\mu_G(t) \\
&\ov{\eqref{dual-trace}}{=} \tr(K_{\varphi_{x,y,T}\phi} K_\psi^*) 
=\big\langle M_{\varphi_{x,y,T}}(K_\phi), K_\psi \big\rangle_{S^2_G}.   
\end{align*}   
Thus we have shown the formula
\begin{align}
\MoveEqLeft
\label{equ-prop-referee-proof-step-1-calcul-du-symbol-avec-coefficients-L2}
\big\langle (\Id \ot T)(W (K_\phi \ot x)W^{-1}),W (K_\psi \ot y)W^{-1} \big\rangle_{S^2_G(\L^2(\VN(G)))} 
=\big\langle M_{\varphi_{x,y,T}}(K_\phi), K_\psi \big\rangle_{S^2_G},
%& = \int_G \int_G \varphi(s,t) \phi(s,t) \ovl{\psi(s,t)} \d \mu_G(s) \d \mu_G(t) , \nonumber
\end{align}
under the additional assumption that the linear map $T$ defines a completely bounded operator $\L^p(\VN(G)) \to \L^2(\VN(G))$. Replacing both brackets, this formula translates to \eqref{MxyT-bis}.

For the general case, we use the first part of Lemma \ref{lem-referee-proof-step-1-cacul-du-symbol-approximation} with the approximation sequence $(M_{\phi_n})$ of Fourier multipliers $M_{\phi_n} \co \VN(G) \to \VN(G)$. By a density argument, we can assume that $y \in \lambda(\C_c(G))$.
According to Lemma \ref{lem-referee-proof-step-1-cacul-du-symbol-approximation}, the composition $M_{\phi_n} T \co \L^p(\VN(G)) \to \L^2(\VN(G))$ is completely bounded. So the first part of the proof applies to this operator.
We obtain 
\[ 
\big\langle (\Id \ot M_{\phi_n}T)(W(K_\phi \ot x)W^{-1}),W(K_\psi \ot y)W^{-1} \big\rangle 
=\big\langle M_{\varphi_n}(K_\phi), K_\psi \big\rangle_{S^2_G},
\]
with $\varphi_n(s,t) \ov{\mathrm{def}}{=} \tau_G\big(\lambda_ty^* \lambda_{s^{-1}} M_{\phi_n}T(\lambda_s x \lambda_{t^{-1}}) \big)$. By the approximation from Lemma \ref{lem-referee-proof-step-1-cacul-du-symbol-approximation}, we have
\begin{align*}
\MoveEqLeft
\big\langle (\Id \ot M_{\phi_n}T)(W(K_\phi \ot x)W^{-1}),W(K_\psi \ot y)W^{-1} \big\rangle \\
&= \big\langle (\Id \ot T)(W(K_\phi \ot x)W^{-1}, (\Id \ot M_{\phi_n}^*) (W(K_\psi \ot y)W^{-1} \big\rangle \\
& \xra[n \to \infty]{}
\big\langle (\Id \ot T)(W(K_\phi \ot x)W^{-1}),W(K_\psi \ot y)W^{-1} \big\rangle.
\end{align*}
Again with Lemma \ref{lem-referee-proof-step-1-cacul-du-symbol-approximation}, we have
\begin{align*}
\MoveEqLeft
\varphi_n(s,t) 
=\big\langle M_{\phi_n} T(\lambda_s x \lambda_{t^{-1}}) , \lambda_s y \lambda_{t^{-1}} \big\rangle \\
& = \big\langle T(\lambda_s x \lambda_{t^{-1}}) , M_{\phi_n}^*(\lambda_s y \lambda_{t^{-1}}) \big\rangle
\xra[n \to \infty]{}
\big\langle T(\lambda_s x \lambda_{t^{-1}}) , \lambda_s y \lambda_{t^{-1}} \big\rangle 
\ov{\eqref{Def-symbol-varphi-1}}{=} \varphi_{x,y,T}(s,t),
\end{align*}
pointwise in $s,t \in G$. Even stronger, if $s,t$ vary in given compacts, $\varphi_n(s,t) = \varphi(s,t)$ for $n$ sufficiently large. Since the functions $\phi$ and $\psi$ are assumed to belong to the space $\C_c(G\times G)$, we obtain then $\langle M_{\varphi_n}(K_\phi), K_\psi \rangle = \langle M_\varphi(K_\phi), K_\psi \rangle$ if $n$ sufficiently large. In summary, we have established the formula
\[ 
\big\langle (\Id \ot T)(W(K_\phi \ot x)W^{-1}),W(K_\psi \ot y)W^{-1} \big\rangle 
= \langle M_{\varphi_{x,y,T}}(K_\phi), K_\psi \rangle.
\]
\end{proof}

\begin{lemma}
\label{Lemma-estimation-cb}
Suppose that $1 \leq p \leq \infty$. Let $T \co \L^p(\VN(G)) \to \L^p(\VN(G))$ be a completely bounded operator (weak* continuous if $p = \infty$). For any elements $x$ and $y$ in the space $\L^1(\VN(G)) \cap \VN(G)$, we have the estimate
\begin{equation}
\label{div-987}
\norm{M_{\varphi_{x,y,T}}}_{\cb,S^p_G \to S^p_G}
\leq\norm{T}_{\cb, \L^p(\VN(G)) \to \L^p(\VN(G))} \norm{x}_{\L^p(\VN(G))} \norm{y}_{\L^{p^*}(\VN(G))}, 
\end{equation}
with the usual convention if $p=1$ or $p=\infty$.
If $p=\infty$ (resp. $p = 1$), we can also take $x \in \VN(G)$ (resp. $y \in \VN(G)$).
Moreover, if the the linear map $T$ is completely positive then the Schur multiplier $M_{\varphi_{x,y,T}}$ is also completely positive.
\end{lemma}

\begin{proof}
By \cite[Definition 2.1]{Pis95}, the duality \cite[Theorem 4.7 p.~49]{Pis98} and Plancherel formula \eqref{Formule-Plancherel}, we have according to Proposition \ref{prop-referee-proof-step-1-calcul-du-symbol-avec-coefficients-L2},
\begin{align*}
\MoveEqLeft
\norm{M_{\varphi_{x,y,T}}}_{\cb,S^p_G \to S^p_G} \\
&\leq \sup \left \{ \left| \sum_{ij} \big\langle (\Id \ot T)(W(K_{\phi_{ij}} \ot x)W^{-1}), W (K_{\psi_{ij}} \ot y) W^{-1} \big\rangle \right| \right. :\norm{[K_{\phi_{ij}}]}_p,\\
& \bigg. \norm{[K_{\psi_{ij}}]}_{p^*} \leq 1 \Bigg\} \\
& \leq \norm{\Id \ot \Id \ot T}_{\cal{B}(S^p(S^p_G(\L^p(\VN(G)))))} \norm{W}_\infty \norm{x}_p\norm{W^{-1}}_\infty \norm{W}_\infty \norm{y}_{p^*} \norm{W^{-1}}_\infty \\
& \leq \norm{T}_{\cb, \L^p(\VN(G)) \to \L^p(\VN(G))} \norm{x}_{\L^p(\VN(G))} \norm{y}_{\L^{p^*}(\VN(G))}.
\end{align*}
If the linear map $T \co \L^p(\VN(G)) \to \L^p(\VN(G))$ is completely positive (and weak* continuous if $p = \infty$), then for any positive elements $[K_{\phi_{ij}}]$ and $[K_{\psi_{ij}}]$ we have
\[ 
\sum_{ij} \big\langle M_{x,y,T} K_{\phi_{ij}}, K_{\psi_{ij}} \big\rangle 
= \sum_{ij} \big\langle (\Id \ot T)(W(K_{\phi_{ij}} \ot x)W^{-1}), W (K_{\psi_{ij}} \ot y) W^{-1} \big\rangle 
\geq 0 .
\]
Indeed, the map $\Id \ot \Id \ot T$ preserves the positivity and $[W(K_{\phi_{ij}} \ot x)W^{-1}]$ and $[W (K_{\psi_{ij}} \ot y) W^{-1}]$ are positive. We infer by \cite[Lemma 2.6 p.~13]{ArK23} that $[M_{\varphi_{x,y,T}} K_{\phi_{ij}}]$ is positive, hence the map $M_{\varphi_{x,y,T}}$ is completely positive.
\end{proof}

%%%%%%%%%%%%%%%%%%%%%%%%%%%%%%%%%%%%%%%%%%%%%%%%%%%%%%%%%%%%%%%%%%%%%%%%%%%%%%%%%%%%%%%%%%%%%%%
\subsection{Step 2: the symbol of $P_j(T)$ is Herz-Schur when $G$ is inner amenable}
\label{Sec-Herz-Schur}

In the following result, we show that if the group $G$ is inner amenable, we are able to make appear Herz-Schur multipliers. Recall that $\varphi_{x,y,T}$ is defined in \eqref{Def-symbol-varphi-1}.

\begin{lemma}
\label{lem-SAIN-Herz-Schur}
Let $G$ be a second-countable unimodular inner amenable locally compact group. Let $F$ be a finite subset of $G$ and let $(V_j^F)_j$ be a sequence of subsets of $G$ satisfying the last point of Theorem \ref{thm-inner-amenable-Folner}. Consider a weak* continuous completely bounded map $T \co \VN(G) \to \VN(G)$. With the notation \eqref{Def-symbol-varphi-1}, we let 
\begin{equation}
\label{Def-ds-inner}
y_j^F 
\ov{\mathrm{def}}{=} c_j^F |\lambda(1_{V_j^F})|^2
\quad \text{and} \quad
\phi_{j,T}^F 
\ov{\mathrm{def}}{=} \varphi_{1,y_j^F,T},
\end{equation}
where $c_j^F > 0$ is the normalisation to have $\norm{y_j^F}_{\L^1(\VN(G))} = 1$. Then any weak* cluster point $\phi_T^F$ of the sequence $(\phi_{j,T}^F)_j$ satisfies 
\[ 
\phi_{T}^F(sr,tr) 
= \phi_{T}^F(s,t), \quad s,t \in G, \: r \in F. 
\]
Moreover, any weak* cluster point of such $(\phi_T^F)_F$, where the finite subsets $F$ of $G$ are directed by inclusion, is also a Herz-Schur symbol.
\end{lemma}

\begin{proof}
For any $s,t \in G$ and any $r \in F$, we have %\textbf{Les $V_j^F$ doivent etre symetriques a cause de l'etoile dans \eqref{Def-symbol-varphi-1}: $\lambda(f)^*=\lambda(\check{\ovl{f}})$ voir DEC I}
\begin{align*}
\MoveEqLeft
\phi_{j,T}^F(sr,tr) - \phi_{j,T}^F(s,t) 
\ov{\eqref{Def-ds-inner}}{=} \varphi_{1,y_j^F,T}(sr,tr)-\varphi_{1,y_j^F,T}(s,t) \\
&\ov{\eqref{Def-symbol-varphi-1}}{=}\tau_G\big(\lambda_{tr}y_j^F \lambda_{(sr)^{-1}} T(\lambda_{sr}  \lambda_{(tr)^{-1}}) \big)-\tau_G\big(\lambda_{t}y_j^F \lambda_{s^{-1}} T(\lambda_s  \lambda_{t^{-1}}) \big)\\
&=\tau_G\big(\lambda_t \lambda_r y_j^F \lambda_{r^{-1}} \lambda_{s^{-1}} T(\lambda_{st^{-1}})\big) - \tau_G\big(\lambda_t y_j^F \lambda_{s^{-1}} T(\lambda_{st^{-1}})\big) \\
&= \tau_G\left(\lambda_t (\lambda_r y_j^F \lambda_{r^{-1}} - y_j^F) \lambda_{s^{-1}} T(\lambda_{st^{-1}}) \right).
\end{align*}
If we can show that
\begin{equation}
\label{equ-1-proof-lemma-SAIN-Herz-Schur}
\norm{\lambda_r y_j^F \lambda_{r^{-1}} - y_j^F}_{\L^1(\VN(G))}
\xra[j \to \infty]{} 0,
\end{equation}
then we will obtain the pointwise convergence $\phi_{j,T}^F(sr,tr) - \phi_{j,T}^F(s,t) \to 0$ as $j \to \infty$, for fixed $s,t \in G$ and $r \in F$.
Since $\phi_{j,T}^F(sr,tr) - \phi_{j,T}^F(s,t)$ is uniformly bounded in the Banach space $\L^\infty(G \times G)$, by dominated convergence, it follows that this sequence converges for the weak* topology to $0$ in the space $\L^\infty(G \times G)$. Thus, if $\phi_T^F$ is a cluster point of $(\phi_{j,T}^F)_{j}$, it is easy to check by a $\frac{\epsi}{3}$-argument, 
%\phi_{T}^F-\phi_{j,T}^F+\phi_{j,T}^F-\tau_r\phi_{j,T}^F+\tau_r\phi_{j,T}^F-\tau_r\phi_{T}^F
using the weak* continuity of translations on $\L^\infty$, that $\phi_T^F(sr,tr) = \phi_T^F(s,t)$ for any $s,t \in G$ and $r \in F$. It remains to show \eqref{equ-1-proof-lemma-SAIN-Herz-Schur}.

First, for any $j$ we have 
\begin{equation}
\label{equal-cjF}
(c_j^F)^{-1} 
= \bnorm{|\lambda(1_{V_j^F})|^2}_{1} 
= \bnorm{\lambda(1_{V_j^F})}_{\L^2(\VN(G))}^2 
= \bnorm{1_{V_j^F}}_{\L^2(G)}^2 
= \mu\big(V_j^F\big).
\end{equation} 
Now, observe by unimodularity in the second equality
\begin{align*}
\MoveEqLeft
\bnorm{\lambda_r |\lambda(1_{V_j^F})|^2 \lambda_{r^{-1}} - |\lambda(1_{V_j^F})|^2 }_1 
\ov{\eqref{composition-et-lambda}}{=} \bnorm{\lambda_r \lambda\big(1_{V_j^F} \ast 1_{V_j^F}\big) \lambda_{r^{-1}} - \lambda\big(1_{V_j^F} \ast 1_{V_j^F}\big) }_1 \\
&=\bnorm{\lambda\big( 1_{V_j^F} \ast 1_{V_j^F}(r^{-1} (\cdot) r) - 1_{V_j^F} \ast 1_{V_j^F} \big)}_1 \\
&\leq \bnorm{\lambda\big( 1_{V_j^F} \ast 1_{V_j^F}(r^{-1} (\cdot) r)-1_{V_j^F} \ast 1_{rV_j^Fr^{-1}}+1_{V_j^F} \ast 1_{rV_j^Fr^{-1}}- 1_{V_j^F} \ast 1_{V_j^F} \big)}_1 \\
&\leq  \bnorm{\lambda\big(1_{V_j^F} \ast 1_{rV_j^Fr^{-1}} - 1_{V_j^F} \ast 1_{V_j^F}(r^{-1}(\cdot)r)\big)}_1+\bnorm{\lambda\big(1_{V_j^F} \ast (1_{rV_j^Fr^{-1}} - 1_{V_j^F})\big)}_1.
\end{align*}
We estimate the second summand with unimodularity by
\begin{align*}
\MoveEqLeft
\bnorm{1_{V_j^F}}_{\L^2(G)} \bnorm{1_{rV_j^Fr^{-1}} - 1_{V_j^F}}_{\L^2(G)} 
\ov{\eqref{Indicator-formula}}{=} \mu\big(V_j^F\big)^{\frac12} \mu\big(rV_j^Fr^{-1} \Delta V_j^F\big)^{\frac12}.             
\end{align*}
Now, we manipulate the first summand. By remplacing $t$ by $tr^{-1}$ and $t$ by $r^{-1}t$ in the fourth equality, we obtain
\begin{align}
\MoveEqLeft
\label{Infinite-34}
1_{V_j^F} \ast 1_{rV_j^Fr^{-1}}(s) - 1_{V_j^F} \ast 1_{V_j^F}(r^{-1}sr) \\
&\ov{\eqref{Convolution-formulas}}{=} \int_G 1_{V_j^F}(t) 1_{rV_j^Fr^{-1}}(t^{-1}s) - 1_{V_j^F}(t) 1_{V_j^F}(t^{-1}r^{-1}sr) \d\mu_G(t) \nonumber\\
%& = \int_G 1_{V_j^F}(t) \big(1_{V_j^Fr^{-1}}(r^{-1}t^{-1}s) - 1_{V_j^Fr^{-1}}(t^{-1}r^{-1}s)\big) \d \mu_G(t)\nonumber\\
&=\int_G 1_{V_j^F}(t) 1_{V_j^Fr^{-1}}(r^{-1}t^{-1}s)\d \mu_G(t) -\int_G 1_{V_j^F}(t)  1_{V_j^Fr^{-1}}(t^{-1}r^{-1}s) \d \mu_G(t) \nonumber \\
&= \int_G 1_{V_j^F}(tr^{-1}) 1_{V_j^Fr^{-1}}(t^{-1}s) \d\mu_G(t)-\int_G  1_{V_j^F}(r^{-1}t) 1_{V_j^Fr^{-1}}(t^{-1}s) \d\mu_G(t) \nonumber\\
&= \int_G \big(1_{V_j^Fr} - 1_{r V_j^F}\big)(t) 1_{V_j^Fr^{-1}}(t^{-1}s) \d\mu_G(t) \nonumber
\ov{\eqref{Convolution-formulas}}{=} \big(1_{V_j^Fr} - 1_{r V_j^F}\big) \ast 1_{V_j^F r^{-1}}(s).\nonumber
\end{align}
Using the invariance of $\mu_G$, we therefore obtain the following estimate for the first summand:
\begin{align*}
\MoveEqLeft
\bnorm{\lambda\big(1_{V_j^F} \ast 1_{rV_j^Fr^{-1}} - 1_{V_j^F} \ast 1_{V_j^F}(r^{-1}(\cdot)r)\big)}_1 
\ov{\eqref{Infinite-34}}{\leq} \bnorm{1_{V_j^Fr} - 1_{rV_j^F}}_{\L^2(G)} \bnorm{1_{V_j^F r^{-1}}}_{\L^2(G)} \\
& \ov{\eqref{Indicator-formula}}{=}  \mu\big(V_j^Fr \Delta r V_j^F\big)^{\frac12} \mu\big(V_j^Fr^{-1}\big)^{\frac12} 
%= \mu\big(r^{-1} V_j^F \Delta V_j^F r^{-1}\big)^{\frac12} \mu\big(V_j^F\big)^{\frac12} \\
= \mu\big(V_j^F \Delta r V_j^F r^{-1}\big)^{\frac12} \mu\big(V_j^F\big)^{\frac12}. 
\end{align*}
Combining the two estimates, we obtain 
\begin{align*}
\MoveEqLeft
\norm{\lambda_r y_j \lambda_r^{-1} - y_j}_1 
\leq 2 c_j^F \mu\big(V_j^F \Delta r V_j^F r^{-1}\big)^{\frac12} \mu\big(V_j^F\big)^{\frac12} \\
&\ov{\eqref{equal-cjF}}{=} 2 \left[ \frac{\mu(V_j^F \Delta r V_j^F r^{-1})}{\mu(V_j^F)} \right]^{\frac12} 
\xra[j]{\eqref{Inner-Folner}} 0,
\end{align*}
according to the inner amenability assumption.

Now, for any finite subset $F$ of the group $G$, we fix a weak* cluster point $\phi_T^F$ of the net $(\phi_{j,T}^F)_j$. Let $\phi_T$ be a weak* cluster point of $(\phi_T^F)_F$. Then for any function $f \in \L^1(G \times G)$ and any $r \in G$, the function $f(\cdot\, r^{-1},\cdot\, r^{-1})$ belongs to the space $\L^1(G \times G)$. Moreover, using unimodularity in the first and third steps and the first part of the proof valid for $F$ containing $\{r\}$ in the forth step, we obtain
\begin{align*}
\MoveEqLeft
\big\langle \phi_T(\cdot\, r, \cdot\, r) , f \big\rangle_{\L^\infty(G \times G),\L^1(G \times G)} 
= \big\langle \phi_T, f(\cdot\, r^{-1}, \cdot\, r^{-1}) \big\rangle_{\L^\infty,\L^1}
= \lim_{F \to \infty} \big\langle \phi_T^F, f(\cdot\, r^{-1}, \cdot\, r^{-1}) \big\rangle_{\L^\infty,\L^1} \\
& = \lim_{F \to \infty} \big\langle \phi_T^F(\cdot\, r, \cdot\, r), f \big\rangle_{\L^\infty,\L^1} 
= \lim_{F \to \infty} \big\langle \phi_T^F, f \big\rangle_{\L^\infty,\L^1} 
= \langle \phi_T, f \rangle_{\L^\infty(G \times G),\L^1(G \times G)}.
\end{align*}
We deduce that the function $\phi_T$ is a Herz-Schur symbol.
\end{proof}

%%%%%%%%%%%%%%%%%%%%%%%%%%%%%%%%%%%%%%%%%%%%%%%%%%%%%%%%%%%%%%%%%%%%%%%%%%%%%%%%%%%%%%%%%%%%%%%
\subsection{Step 2: the symbol of $P_j(T)$ for a Fourier multiplier $T$ if $p=\infty$ or $p=1$}
\label{Section-p=1-p-infty}

We start with the case $p=\infty$. Let $T = M_\phi \co \VN(G) \to \VN(G)$ be a completely bounded Fourier multiplier. If $x \in \VN(G)$ and $y \in \L^1(\VN(G)) \cap \VN(G)$, recall that the symbol $\varphi_{x,y,T}$ is defined in \eqref{Def-symbol-varphi-1}. 

\begin{lemma}
\label{lemma-symbol-step-1-p=infty}
Let $G$ be a second-countable unimodular locally compact group. Consider a completely bounded Fourier multiplier $T = M_\phi \co \VN(G) \to \VN(G)$. Let $y$ be a positive element in the space $\L^1(\VN(G)) \cap \VN(G)$ such that $\tau_G(y) = 1$. We have
\begin{equation}
\label{}
\varphi_{1,y,T}(s,t)
=\phi(st^{-1}), \quad s,t \in G.
\end{equation}
\end{lemma}

\begin{proof}
For any $s,t \in G$, we have
\begin{align*}
\MoveEqLeft
\varphi_{1,y,T}(s,t)
\ov{\eqref{Def-symbol-varphi-1}}{=} \tau_G(y \lambda_{s^{-1}} M_\phi(\lambda_s \lambda_{t^{-1}}) \lambda_t)            
=\phi(st^{-1})\tau_G(\lambda_ty \lambda_{s^{-1}} \lambda_s \lambda_{t^{-1}} ) \\
&=\phi(st^{-1})\tau_G(y)
=\phi(st^{-1}).
\end{align*} 
\end{proof}

\begin{example} \normalfont
\label{example-p=infy}
Let $g$ be a continuous functions with compact support on $G$ with $\norm{g}_{\L^2(G)} = 1$. With $y \ov{\mathrm{def}}{=} \lambda(g^**g)$, the assumptions of Lemma \ref{lemma-symbol-step-1-p=infty} are satisfied by \eqref{composition-et-lambda} and since $
\tau_G(y)
\ov{\eqref{composition-et-lambda}}{=} \tau_G(\lambda(g)^*\lambda(g)) 
\ov{\eqref{Formule-Plancherel}}{=} \norm{g}_{\L^2(G)}^2
=1$.
%\cite[Theorem 32.4 (iii)]{HeR70}
%Let $(f_j)_j$ be a net of continuous functions with compact support on $G$ such that $\supp f_j \to \{e \}$\footnote{\thefootnote. That means that for any neighborhood $V$ of $e$ there exists $j_0$ such that for any $j > j_0$ we have $\supp g_j \subseteq V$).}.
%\textbf{il y a un probleme dans l'avant derniere; voir DEC I (6.3) ou changer le crochet de dualite avec l'antipode (comme pour les matrices et la transpositionsupposer que $f_j$ est symetrique;
\end{example}

We continue with the case $p=1$. We can prove the following similar result. %the same calculation as above and simply swap $T$ to its adjoint.

\begin{lemma}
\label{lemma-symbol-step-1-p=1}
Let $G$ be a second-countable unimodular locally compact group. Consider a completely bounded Fourier multiplier $T = M_\phi \co \L^1(\VN(G)) \to \L^1(\VN(G))$. Let $x$ be a positive element in the space $\L^1(\VN(G)) \cap \VN(G)$ such that $\tau_G(x) = 1$. We have
\begin{equation}
\label{}
\varphi_{x,1,T}(s,t)
=\phi(st^{-1}), \quad s,t \in G.
\end{equation}
\end{lemma}

\begin{proof}
For any $s,t \in G$, we have
\begin{align*}
\MoveEqLeft
\varphi_{x,1,T}(s,t)
\ov{\eqref{Def-symbol-varphi-1}}{=} \tau_G\big(\lambda_t\lambda_{s^{-1}} M_\phi(\lambda_s x\lambda_{t^{-1}}) \big) 
=\tau_G\big(\lambda_t\lambda_{s^{-1}} M_\phi(\lambda_s x\lambda_{t^{-1}}) \big) \\
&=\tau_G\big( M_{\check{\phi}}(\lambda_{ts^{-1}}) \lambda_s x\lambda_{t^{-1}}) \big)    
=\phi(st^{-1})\tau_G\big(\lambda_{ts^{-1}} \lambda_s x\lambda_{t^{-1}}) \big)
=\phi(st^{-1})\tau_G(x)
=\phi(st^{-1}).
\end{align*} 
\end{proof}

%\begin{example} \normalfont
%\label{example-p=1}
 %Let $(f_j)_j$ be a net of continuous functions with compact support such that $\supp f_j \to \{e \}$. Suppose that $\norm{f_j}_{\L^2(G)} = 1$ and that $\lambda(f_j) \geq 0$. With $x_j \ov{\mathrm{def}}{=} \lambda(f_j)^2$, the assumptions of Lemma \ref{lemma-symbol-step-1-p=1} are satisfied.
%\end{example}

%\begin{prop}
%\label{prop-referees-proof-step-1-weak-star-convergence-p=1}
%Let $G$ be a second-countable unimodular locally compact group. Let $T = M_\phi \co \VN(G) \to \VN(G)$ be a completely bounded Fourier multiplier. Let $(x_\alpha)_\alpha$ be a net of $\L^1(\VN(G)) \cap \VN(G)$ such that
%\begin{itemize}
%\item $x_\alpha \geq 0$ for all $\alpha$
%\item $\norm{x_\alpha}_{1} \leq C$ for all $\alpha$,
%
%\item $\tau_G(x_\alpha) = 1$ for all $\alpha$
%
%\end{itemize}
%Moreover, let
%\begin{equation}
%\label{def-symbol-phi-alpha}
%\phi_{\alpha,T}(s,t) 
%\ov{\mathrm{def}}{=} \tau_G \big(y_\alpha \lambda_{s^{-1}} T(\lambda_s \lambda_{t^{-1}}) \lambda_t\big), \quad s,t \in G.
%\end{equation}
%Then for any $\alpha$ we have $\phi_{\alpha,T}=\phi_T$ where $(s,t) \mapsto \phi_T(s,t) \ov{\mathrm{def}}{=} \phi(st^{-1})$.
%\end{prop}

%%%%%%%%%%%%%%%%%%%%%%%%%%%%%%%%%%%%%%%%%%%%%%%%%%%%%%%%%%%%%%%%%%%%%%%%%%%%%%%%%%%%%%%%%%%%%%%
\subsection{Step 2: convergence of the symbols for a multiplier $T$ with arbitrary symbol}
\label{Sec-convergence-continuous}

We show that for a suitable choice of sequences of functions, we obtain the convergence of symbols to the desired Herz-Schur symbol.

\begin{prop}
\label{th-convergence}
Let $G$ be a second-countable unimodular locally compact group. Suppose that $1 \leq p \leq \infty$. Consider some completely bounded Fourier multiplier $T = M_\phi \co \L^p(\VN(G)) \to \L^p(\VN(G))$.
%such that the symbol $\phi \co G \to \mathbb{C}$ is continuous.
Let $(f_j)$ and $(g_j)$ be nets of positive functions with compact support belonging to the space $\C_e(G)$ such that if $x_j \ov{\mathrm{def}}{=} \lambda(f_j)$, $y_j \ov{\mathrm{def}}{=} \lambda(g_j)$ we have
\begin{itemize}
\item $\norm{x_j}_{\L^p(\VN(G))} \norm{y_j}_{\L^{p^*}(\VN(G))} \leq C$ for all $j$ for some positive constant $C$,

\item $\tau_G(x_j y_j) = 1$ for all $j$,

\item $\supp f_j \to \{e\}$ or $\supp g_j \to \{e\}$. 
\end{itemize}
Moreover, let
\begin{equation}
\label{def-symbol-phi-alpha}
\phi_{j,T}(s,t) 
\ov{\mathrm{def}}{=} \varphi_{x_j,y_j,T}(s,t) 
\ov{\eqref{Def-symbol-varphi-1}}{=} \tau_G \big(\lambda_ty_j \lambda_{s^{-1}} T(\lambda_s x_j \lambda_{t^{-1}}) \big), \quad s,t \in G.
\end{equation}
Then the sequence $(\phi_{j,T})_j$ of elements in the space $\L^\infty(G \times G)$ converges for the weak* topology to the function $\phi^\HS \co (s,t) \mapsto \phi(st^{-1})$.
\end{prop}

\begin{proof}
For any $j$  and almost all $s,t \in G$, we have using a change of  variables in the last equality
\begin{align}
\MoveEqLeft
\label{Calcul-symbole}
\phi_{j,T}(s,t) 
\ov{\eqref{def-symbol-phi-alpha}}{=} \tau_G \big(y_j \lambda_{s^{-1}} T(\lambda_s x_j \lambda_{t^{-1}}) \lambda_t\big)            
=\tau_G \big(\lambda_t\lambda(g_j) \lambda_{s^{-1}} T(\lambda_s \lambda(f_j) \lambda_{t^{-1}}) \big)\\
&=\tau_G \big(\lambda(g_j(t^{-1}\cdot s))  M_\phi(\lambda(f_j(s^{-1}\cdot t) \big) 
=\tau_G \big(\lambda(g_j(t^{-1}\cdot s)) \lambda\big(\phi f_j(s^{-1}\cdot t) \big) \nonumber\\
&\ov{\eqref{Formule-Plancherel}}{=} \int_G g_j(t^{-1}u^{-1} s)) \phi(u) f_j(s^{-1}u t) \d\mu_G(u) 
=\int_G  \phi(sut^{-1}) g_j(u^{-1})f_j(u) \d\mu_G(u). \nonumber
\end{align} 
By Lemma \ref{Lemma-estimation-cb} and \eqref{symbol-phixyT}, we deduce that $\bnorm{M_{\phi_{j,T}}}_{\cb,S^p_G \to S^p_G} \leq C$ for any $j$. Using the inequality \eqref{ine-infty}, we see that the net $(\phi_{j,T})_j$ of functions is uniformly bounded in the Banach space $\L^\infty(G \times G)$.
Thus to check the claimed weak* convergence, it suffices by \cite[Proposition 1.21 p.~8]{Dou98} to test against a function $h \in \C_c(G \times G)$.
We suppose that $\supp f_j \to \{e\}$. Since $\int_G \check{g}_jf_j \d \mu_G
\ov{\eqref{Formule-Plancherel}}{=} \tau\big(\lambda(g_j)\lambda(f_j)\big)
=\tau(x_jy_j)=1$,
we have, using unimodularity in a change of variables, and with the notation $K_j = \supp h \cup \{(s,t):\: \exists \: u \in \supp f_j:\: (su^{-1},t) \in \supp h\}$
\begin{align*}
\MoveEqLeft
\left|\int_{G \times G}  \big(\phi_{j,T}(s,t) -\phi^\HS(s,t)\big) h(s,t) \d\mu_G(s)\d\mu_G(t)\right| \\
& \ov{\eqref{Calcul-symbole}}{=} \left|\int_G \int_G \int_G \big(\phi(sut^{-1}) -\phi(st^{-1})\big)\check{g}_j(u)f_j(u) h(s,t) \d\mu_G(u)\d\mu_G(s)\d\mu_G(t)\right|          
\\
&=\left|\int_G \int_G \int_G \phi(st^{-1}) \check{g}_j(u)f_j(u) \big( h(su^{-1},t) - h(s,t) \big)  \d\mu_G(u) \d\mu_G(s) \d\mu_G(t)  \right|  \\
&\leq \sup \left\{ |h(su^{-1},t) - h(s,t)| : \:u \in \supp f_j,(s,t) \in \supp h, (su^{-1},t) \in \supp h \right\} \cdot \\
& \cdot \int_{K_j} \int_G  |\phi(st^{-1})| f_j(u)\check{g}_j(u)\d\mu_G(u)\d\mu_G(s)\d\mu_G(t) \\
& =  \sup \left\{ |h(su^{-1},t) - h(s,t)| : \:u \in \supp f_j,(s,t) \in \supp h, (su^{-1},t) \in \supp h \right\} \cdot \\
& \cdot \int_{K_j} |\phi(st^{-1})| \d\mu_G(s)\d\mu_G(t)
 \xra[j]{} 0,
\end{align*} 
since $K_j$ is contained in a fixed compact, so that the last integral is uniformly bounded in $j$, and $h$ was supposed to be continuous.
We can use a similar reasoning if $\supp g_j \to \{e\}$.
\end{proof}

\begin{example} \normalfont
\label{Essai}
Let $G$ be a second-countable unimodular locally compact group. Consider some value $p \in (1,\infty)$ and assume that $\frac{p}{p^*}$ is rational. That is, $p = \frac{p}{p^*} + 1$ is rational, which implies that both $\frac{1}{p}$ and $\frac{1}{p^*}$ are also rational. Therefore, there exist integers $l,m,n \geq 1$ such that $\frac{1}{p}=\frac{m}{n}$ and $\frac{1}{p^*}=\frac{l}{n}$. Consequently, $\frac{n}{p}=m$ and $\frac{n}{p^*}=l$ are integers. Consider a sequence $(k_j)$ of positive functions belonging to the space $\C_c(G)$ with $\supp k_j \to \{e\}$. For each integer $j$, we define the function $h_j \ov{\mathrm{def}}{=} k_j^* \ast k_j$. We can suppose that $\norm{\lambda(h_j)}_{\L^{n}(\VN(G))} = 1$. We let 
\begin{equation}
\label{xj-yj}
x_j \ov{\mathrm{def}}{=} (\lambda(h_j))^m
\quad \text{and} \quad 
y_j \ov{\mathrm{def}}{=} (\lambda(h_j))^l.
\end{equation}
Note that by \eqref{composition-et-lambda} these elements belong to $\mathfrak{m}_{\tau_G}$, as defined in \eqref{Def-mtauG}. Then the sequences $(x_j)$ and $(y_j)$ satisfy the assumptions of Proposition \ref{th-convergence}. Indeed, the $x_j$ and $y_j$ are positive and we have
%$$
%\norm{x_j}_p 
%=\bnorm{(\lambda(h_j^* *h_j))^m}_p
%\ov{\eqref{composition-et-lambda}}{=} \bnorm{|\lambda(h_j)|^{\frac{2n}{p}}}_p 
%=\norm{\lambda(h_j)}_{2n}^{\frac{2n}{p}} 
%= 1
%$$ 
%ou
$$
\norm{x_j}_p 
\ov{\eqref{xj-yj}}{=}\bnorm{(\lambda(h_j))^m}_p
= \bnorm{\lambda(h_j)^{\frac{n}{p}}}_p 
=\norm{\lambda(h_j)}_{n}^{\frac{n}{p}} 
= 1
$$
and similarly 
%$$
%\norm{y_j}_{p^*} 
%= \bnorm{(\lambda(h_j^* *h_j))^{l}}_{p^*}
%\ov{\eqref{composition-et-lambda}}{=} \bnorm{|\lambda(h_j)|^{\frac{2n}{p^*}}}_{p^*}
%=\norm{\lambda(h_j)}_{2n}^{\frac{2n}{p^*}}
%= 1.
%$$
%ou
$$
\norm{y_j}_{p^*} 
\ov{\eqref{xj-yj}}{=} \bnorm{(\lambda(h_j))^{l}}_{p^*}
= \bnorm{\lambda(h_j)^{\frac{n}{p^*}}}_{p^*}
=\norm{\lambda(h_j)}_{n}^{\frac{n}{p^*}}
= 1.
$$
Finally, we observe that
%$$
%\tau_G(x_j y_j) 
%= \tau_G\big( |\lambda(h_j)|^{\frac{2n}{p}}|\lambda(h_j)|^{\frac{2n}{p^*}}\big)
%=\tau_G\big( |\lambda(h_j)|^{2n}\big)
%=\norm{\lambda(h_j)}_{2n}^{2n}
%= 1.
%$$
%ou
$$
\tau_G(x_j y_j) 
= \tau_G\big( \lambda(h_j)^{m}\lambda(h_j)^{l}\big)
=\tau_G\big( \lambda(h_j)^{n}\big)
=\norm{\lambda(h_j)}_{n}^{n}
= 1.
$$
Note that these sequences depend on $p$.
\end{example}

%%%%%%%%%%%%%%%%%%%%%%%%%%%%%%%%%%%%%%%%%%%%%%%%%%%%%%%%%%%%%%%%%%%%%%%%%%%%%%%%%%%%%%%%%%%%%
\subsection{Step 2: the case of totally disconnected and finite-dimensional groups}
\label{Sec-finite-dim}

In order to achieve a complementation that ensures the compatibility of the resulting projection $P^p_G$ for different values of $p$, one needs to select different sequences than those defined in Example \ref{Essai}. This will be achieved in Corollary \ref{Cor-38} in the case where the locally compact group $G$ is finite-dimensional. 

\paragraph{Dimensions of topological spaces} Recall that three notions of dimension of a \textit{suitable} topological space $X$ exist: the small inductive dimension, the large inductive dimension and the covering dimension. These dimensions are defined, for example, in \cite[Chapter 7]{Eng89}. Recall the definition of small inductive dimension. Let $X$ be a regular topological space. We say that $\ind X = -1$ if $X$ is empty. If $n$ is a natural number, then we say that $\ind X \leq n$ if for every point $x \in X$ and every neighborhood $V$ of $x$ in $X$ there exists a open set $U$ included in $V$ such that $x \in U$ and such that the boundary $\partial U$ satisfies $\ind \partial U \leq n-1$. We say that $\ind X=n$  if $\ind X \leq n$ and $\ind X \leq n-1$ does not hold. Finally, we say that $\ind X=\infty$ if the inequality $\ind X \leq n$ does not hold for any integer $n$. If two regular topological spaces $X$ and $Y$ are homeomorphic then $\ind X=\ind Y$. We refer to the book \cite{Eng89} for more information. %and \cite{Pea75}

By \cite[Theorem 7.3.3, p.~404]{Eng89}, these notions coincide when $X$ is metrizable and separable. Note that a second-countable locally compact group $G$ satisfies this property\footnote{\thefootnote. Such a group is metrizable by \cite[Theorem 2.B.2 p.~20]{CoH16} and second-countable topological spaces are separable by \cite[Corollary 1.3.8 p.~25]{Eng89}. See also \cite[Theorem 2.A.10 p.~15]{CoH16}, which presents a characterization of locally compact spaces which are second-countable.}. Indeed, Arhangel'skii and Pasynkov showed in \cite{Arh60} and \cite{Pas60} that these notions coincide for an \textit{arbitrary} locally compact group $G$. We refer to the survey \cite[p.~205]{ArM18} for more information.

\begin{example} \normalfont
\label{0-dim-space}
According to \cite[p.~360]{Eng89}, a topological space $X$ is called zero-dimensional if it is a non-empty $T_1$-space with a basis of open-and-closed subsets. If $X$ is locally compact and paracompact, it is equivalent, by \cite[Theorem 6.2.10, p.~362]{Eng89}, to say that $X$ is totally disconnected\footnote{\thefootnote. In \cite[p.~360]{Eng89}, the term <<hereditarily disconnected>> is used for this notion.}, meaning it contains no connected subspace with more than one point. Furthermore, \cite[Theorem 7.1.12, p.~388]{Eng89} shows that this is also equivalent to  $\ind X=0$. It is worth noting that every metrizable space is paracompact, as stated in \cite[Theorem 5.1.3, p.~300]{Eng89}.
\end{example}

\begin{example} \normalfont
\label{ex-finite-loc-compact}
By \cite[Remark 39.5 (d) p.~283]{Str06}, a finite-dimensional locally compact group $G$ is a Lie group if and only if it is locally connected\footnote{\thefootnote. This result is stronger than \cite[Exercise 1.6.9 p.~122]{Tao14}, which says without proof that a locally compact group $G$ is a Lie group if and only if it is first-countable, locally connected and finite-dimensional. Moreover, the notion of dimension of \cite[Exercise 1.6.9 p.~122]{Tao14} is different.}. See also \cite[Theorem 70, p.~337]{Pon66} for the compact case.
\end{example}

We need background on local isomorphisms since we will use Iwasawa's local splitting theorem, which provides some local isomorphism.

\paragraph{Local isomorphisms} Recall that two topological groups $G$ and $H$ are said to be locally isomorphic \cite[p.~224]{Bou98} if there exist open neighborhoods $V$ and $W$ of the identity elements $e_G$ and $e_H$ and a homeomorphism $f \co V \to W$ satisfying  $f(xy)=f(x)f(y)$ for all $x,y\in V$ such that $xy \in V$ and if $g$ is the mapping inverse to $f$, then for each pair of points $x', y'$ in $W$ such that $x'y' \in W$, we have $g(x'y') = g(x') g(y')$. We say that $f$ is a local isomorphism of $G$ with $H$. 

%As observed in \cite[p.~225]{Bou98} if $f$ is a local isomorphism of $G$ with $H$, then every restriction of $f$ to a neighbourhood $V_0$ of the identity element of $G$ such that $V_0^2 \subseteq V$ is again a local isomorphism of $G$ with $H$.

The following result from \cite[pp.~18-19]{Bou04b} describes the relationship between Haar measures and local isomorphisms.
%\cite[Theorem 3 p.22]{Dav}. 
%We think that there exists a simpler proof with the construction of the Haar measure of the proof \cite[Theorem 441C]{Fre4} but the details remains painful. \textbf{Beaucoup mieux :  a lire.}

\begin{lemma}
\label{Lemma-locally-isomorphic-1}
Let $G$ and $G'$ be locally isomorphic locally compact groups via a local homeomorphism $f \co V \to W$. Consider a left Haar measure $\mu_G$ of $G$ and its restriction $\mu_G^{V}$ on $V$. Then $f(\mu_G^V)$ is the restriction of a unique left Haar measure on $G'$.
\end{lemma}

We caution the reader that the property of being unimodular is not preserved under local isomorphisms. For an example of a non-unimodular locally compact group $G$ that is locally isomorphic to the unimodular locally compact group $\R$, see \cite[Exercise 5, VII.78]{Bou04b}.

\paragraph{Splitting theorem} 

We will use the following form \cite[Theorem B p.~92]{Glu60} of Iwasawa's local splitting theorem. See also \cite[Exercise 1.6.8 p.~122]{Tao14}  % see also \cite[Remark 39.5 (a) p.~283]{Str06}. 
and \cite[Theorem 70 p.~337]{Pon66} for a version for the particular case of compact groups.

\begin{thm}
\label{cor-spitting-2}
Every second-countable finite-dimensional locally compact group is locally isomorphic to the product of a totally disconnected compact group and a connected Lie group.
\end{thm}

\paragraph{Doubling metric measure spaces}
A Borel regular measure $\mu$ on a metric space $(X,\dist)$ is called a doubling measure \cite[p.~76]{HKST15} if every ball in $X$ has positive and finite measure and if there exists a constant $c \geq 1$ such that
\begin{equation}
\label{doubling-def}
\mu(B(x , 2r)) \leq 
c\,\mu(B(x,r)), \quad x \in X,\, r >0.
\end{equation}
Here $B(x,r) \ov{\mathrm{def}}{=} \{y \in X : \dist(x, y) < r\}$ is the open ball with radius $r$ centred at $x$. We call the triple $(X,\dist,\mu)$ a doubling metric measure space if $\mu$ is a doubling measure on $X$. Such a space $X$ is separable as a topological space by \cite[p.~76]{HKST15}.  We refer to the paper \cite{SoT19} for more information on the least doubling constant $\inf\{ c \text{ as in } \eqref{doubling-def}: \mu\text{ doubling measure on }(X,\dist)\}$ of a metric space $(X,\dist)$.

We introduce and will use the weaker notion of <<doubling measure for small balls>> replacing the inequality \eqref{doubling-def} by
\begin{equation}
\label{doubling-def-local}
\mu(B(x,2r)) 
\leq c\, \mu(B(x,r)), \quad x \in X,\, r \in (0,\tfrac{1}{2}].
\end{equation}

\paragraph{Carnot-Caratheodory distances} Consider a connected Lie group $G$ equipped with a left Haar measure $\mu_G$ and identity element $e$. We consider a finite sequence $X \ov{\mathrm{def}}{=}(X_1,\ldots,X_m)$ of left invariant vector fields, the generated Lie algebra of which is the Lie algebra $\frak{g}$ of the Lie group $G$ such that the vectors $X_1(e),\ldots, X_m(e)$ are linearly independent. We say that it is a family of left invariant H\"ormander vector fields. Let $\gamma \colon [0,1] \to G$ be an absolutely continuous path such that $\dot\gamma(t)$ belongs to the subspace $\Span \{ X_1|_{\gamma(t)}, \ldots, X_m|_{\gamma(t)} \}$ for almost every $t \in [0,1]$. If $\dot\gamma(t) = \sum_{k=1}^m \gamma_k(t) \, X_k|_{\gamma(t)}$ for almost every $t \in [0,1]$, where each $\dot\gamma_k$ is measurable, we can define the length of $\gamma$ by 
$$
\ell(\gamma) 
\ov{\mathrm{def}}{=} \int_0^1  \Big( \sum_{k=1}^m |\dot\gamma_k(t)|^2 \Big)^{1/2} \d t,
$$
which belongs to $[0,\infty]$. For any $s,s' \in G$, there exists such a path $\gamma \co [0,1] \to G$ with finite length such that $\gamma(0) = s$ and $\gamma(1) = s'$. If $s,s' \in G$ then we define the Carnot-Carath\'eodory distance
\begin{equation}
\label{distance-Carnot}
\dist_\CC(s,s')
\ov{\mathrm{def}}{=} \inf_{\gamma(0)=s,\gamma(1)=s'} \ell(\gamma)
\end{equation}
between $s$ and $s'$ to be the infimum of the length of all such paths with $\gamma(0) = s$ and $\gamma(1) = s'$. Then it is known that $\dist_\CC$ is a left invariant distance on $G$, inducing the same topology as the one of $G$, see \cite[Proposition III.4.1 p.~39]{VSCC92} and \cite[pp.~22-23]{DtER03}.

By \cite[p.~124]{VSCC92} there exist $c_1,c_2> 0$ and $d \in \N$ such that for all $r \in (0,1]$ we have
\begin{equation}
\label{Equivalence-measure-ball}
c_1 \, r^d 
\leq \mu_G(B(e,r)) 
\leq c_2 \, r^d.
\end{equation} 
The integer $d$ is called the local dimension of $(G,X)$. We infer that there exists $c > 0$ such that \eqref{doubling-def-local} is satisfied, i.e.~$\mu_G$ is a doubling measure for small balls. By \cite[Proposition 2.4 p.~199]{BEM13}, the metric measure space $(G,\dist_{\CC},\mu_G)$ is a doubling metric measure space if and only if the Lie group $G$ has polynomial growth.
%, which means that there exist $D \in \N$ and $c',C' > 0$ such that
%\begin{equation}
%\label{poly-growth}
%c' r^D 
%\leq V(r) 
%\leq 
%C' r^D, \quad r \geq 1.
%\end{equation}
Recall finally that the connected component of a Lie group is second-countable by \cite[Proposition 9.1.15 p.~293]{HiN12}.

\paragraph{Construction of some neighborhoods} We start with a technical result.

\begin{lemma}
\label{lem-Lie-group-estimate}
Let $G$ be a second-countable locally compact group equipped with a left invariant distance $\dist$ and a  doubling left Haar measure $\mu_G$ for small balls. There exists a sequence $(B_j)$ of open balls $B_j \ov{\mathrm{def}}{=} B(e,r_j)$ with decreasing radius $r_j \to 0$ satisfying
\begin{equation}
\label{Lie-estimate}
\mu_G(B_j)^3 
\leq c^3\int_{B_j} \mu_G(B_j \cap sB_j)^2 \d \mu_G(s), \quad j \geq 1,
\end{equation}
where $c$ is a constant satisfying \eqref{doubling-def-local}.
\end{lemma}

\begin{proof}
%We denote by $n$ the dimension of $L$. According to \cite[Theorem B]{Glu60}, \cite[Ex 1.6.8 p.122]{Tao14}, \cite[Remark 39.5 (a)]{Str06} and \cite[p.~237]{MoZ55}\footnote{\thefootnote. See also \cite[Th. 70, p. 337]{Pon66} for the case of compact groups.} a $n$-dimensional first countable locally compact group is locally isomorphic to the product of a compact totally disconnected group $K$ and a Lie group $L$ of dimension $n$. So there exists a neighborhood $V$ of $e_G$ which is homeomorphic the direct product $W \times U$ of a neighborhood $W$ of the neutral element $e_K$ and an open neighborhood $U$ of $e_L$. By restriction, we can suppose that $U$ is connected and,
Assume that $0 < \epsi < 1$. For any integer $j \geq 1$, we introduce the ball $B_j \ov{\mathrm{def}}{=} B \big(e,\frac{\epsi}{j}\big)$. We have
\begin{equation}
\label{mesure-Vj}
\mu_G(B_{j})
= \mu_G\big(B \big(e,\tfrac{\epsi}{j}\big)\big) 
\ov{\eqref{doubling-def-local}}{\leq} c\,\mu_G\big(B \big(e,\tfrac{\epsi}{2j}\big)\big)
=c \,\mu_G(B_{2j}).
\end{equation}
For any element $s$ in the open ball $B_{2j}=B(e,\frac{\epsi}{2j})$, we will show that 
\begin{equation}
\label{inclusion-balls}
B\big(e,\tfrac{\epsi}{2j}\big) 
\subset B\big(e,\tfrac{\epsi}{j}\big) \cap sB\big(e,\tfrac{\epsi}{j}\big), \quad \text{i.e.} \quad B_{2j}
\subset B_j \cap sB_j.
\end{equation}
Indeed, if $r \in G$ satisfies $\dist(e,r) < \frac{\epsi}{2j}$ we have obviously $r \in B(e,\frac{\epsi}{j})$ and using left invariance of the distance, we obtain
$$
\dist(e,s^{-1}r) 
=\dist(s,r) 
\leq \dist(s,e) + \dist(e,r) 
< \frac{\epsi}{2j} + \frac{\epsi}{2j}
=\frac{\epsi}{j}.
$$ 
So $s^{-1}r \in B(e,\frac{\epsi}{j})$ and consequently $r \in sB(e,\frac{\epsi}{j})$. So the claim \eqref{inclusion-balls} is proved. For any integer $j \geq 1$, we deduce that 
\begin{align*}
\MoveEqLeft
c^3\int_{B_j} \mu_G(B_j \cap sB_j)^2 \d \mu_G(s) 
\geq c^3\int_{B_{2j}} \mu_G(B_j \cap sB_j)^2 \d \mu_G(s) \\
&\ov{\eqref{inclusion-balls}}{\geq} c^3\int_{B_{2j}} \mu_G(B_{2j})^2 \d\mu_G(s) 
= c^3\,\mu_G(B_{2j})^3 
\ov{\eqref{mesure-Vj}}{\geq} \mu_G(B_j)^3.
\end{align*}
\end{proof}

%\paragraph{Carnot-Carath\'eodory metrics on connected Lie groups} 

%\begin{prop} 
%\begin{enumerate}
%\item \label{psubkato203-0.5}
%\item \label{psubkato203-1}
%The metric space $(G,d)$ is complete and the closed balls $\overline{B(x,r)}$ are compact.
%\item \label{psubkato203-2}
%
%\item \label{psubkato203-3}
%There exist $C > 0$ and $\lambda \geq 0$ such that 
%$\mu(B(r)) \leq C \, e^{\lambda r}$ for all $r \geq 1$.
%\end{enumerate}
%\end{prop}
%
%\textbf{For a proof of Statement~\ref{psubkato203-0.5}, 
%Statement~\ref{psubkato203-1} follows from the discussion in Section~III.4 
%in \cite{VSCC92} and the fact that every locally compact metric space is complete.
%Statement~\ref{psubkato203-2} is a consequence of \cite[Theorems~1 and~4]{NSW1} .}

We continue by proving another technical result for totally disconnected groups.

\begin{lemma}
\label{lem-disconnected-group-estimate}
Let $G$ be a second-countable totally disconnected locally compact group equipped with a left Haar measure $\mu_G$. Then there exists a basis $(K_j)$ of symmetric open compact neighborhoods $K_j$ of $e$ such that
\begin{equation}
\label{disco-estim}
\mu_G(K_j)^3 
=\int_{K_j} \mu_G(K_j \cap s K_j)^2 \d \mu_G(s), \quad j \geq 1. 
\end{equation}
\end{lemma}

\begin{proof}
According to Van Dantzig's theorem \cite[(7.7) Theorem p.~62]{HeR79} or \cite[Theorem 2.E.6 p.~44]{CoH16}, $G$ admits a basis $(K_j)$ of open compact subgroups. Clearly, each $K_j$ is symmetric, being a group, is a neighborhood of $e$, being an open subset, and of finite measure due to its compactness. Since we have assumed $G$ to be second-countable, it follows from the proof of \cite[Theorem 2.E.6 p.~44]{CoH16} that the basis can be chosen as a sequence. Since $K_j$ is a subgroup, we have $K_j \cap sK_j = \emptyset$ for $s \not\in K_j$ and $K_j \cap sK_j = K_j$ for $s \in K_j$.
Thus,
\[ 
\int_{K_j} \mu_G\left(K_j \cap sK_j \right)^2 \d\mu_G(s) 
= \int_{K_j} \mu_G(K_j)^2 \d\mu_G(s) 
= \mu_G(K_j)^3 . 
\]
\end{proof}

Using a version of Iwasawa's local splitting theorem, we are now able to obtain a result for \textit{finite-dimensional} locally compact groups.

\begin{lemma}
\label{lem-finite-dimensional-group-estimate}
Let $G$ be a second-countable finite-dimensional locally compact group equipped with a left Haar measure $\mu_G$. Then there exists a basis $(V_j)$ of symmetric open neighborhoods $V_j$ of $e$ and a constant $c > 0$ such that
\begin{equation}
\label{subtil-estimate}
 \mu_G(V_j)^3 
\leq c^3\int_{V_j} \mu_G(V_j \cap sV_j)^2 \d \mu_G(s), \quad j \geq 1. 
\end{equation}
\end{lemma}

\begin{proof}
We denote by $n$ the dimension of $G$. According to Theorem \ref{cor-spitting-2}, %\cite[Theorem B p.~92]{Glu60}, \cite[Ex 1.6.8 p.~122]{Tao14}, \cite[Remark 39.5 (a) p.~283]{Str06} and \cite[p.~237]{MoZ55}\footnote{\thefootnote. See also \cite[Theorem 70 p.~337]{Pon66} for the case of compact groups.}, 
$G$ is locally isomorphic to the product of a totally disconnected compact group $K$ and a Lie group $L$ of dimension $n$. So there exists a neighborhood $V$ of $e_G$ which is homeomorphic to the direct product $W \times U$ of a neighborhood $W$ of the neutral element $e_K$ and an open neighborhood $U$ of the neutral element $e_L$.
%We denote by $n$ the dimension of $G$. 
%According to Corollary \ref{cor-spitting-4}, the group $G$ is locally isomorphic to the product $K \times L$ of a compact group $K$ and a connected Lie group $L$. So there exists a neighborhood $V$ of $e_G$ which is homeomorphic the direct product $W \times U$ of a neighborhood $W$ of the identity element $e_K$ and an open neighborhood $U$ of the identity element $e_L$. 
%By restriction, we can suppose that $U$ is connected and, by considering the connected component of $L$, that the Lie group $L$ is connected. 
We identify $V$ with $W \times U$. By Lemma \ref{Lemma-locally-isomorphic-1}, we can choose left Haar measures $\mu_K$ and $\mu_L$ on the groups $K$ and $L$ such that
%there exists an open neighborhood $W_0 \subset W$ of $e_K$ and an open neighborhood $R_0 \subset R$ of $e_L$ such that \textbf{on doit pouvoir eviter de restreindre avec \cite[page 18-19]{Bou04b}}
\begin{equation}
\label{equ-1-proof-lem-finite-dimensional-group-estimate}
\mu_G(A \times B) 
= \mu_K(A) \mu_L(B), \quad A \subseteq W, \: B \subseteq U.
\end{equation}

Next, we consider the left invariant metric on the connected Lie group $L$ given by the Carnot-Carath\'eodory distance \eqref{distance-Carnot} with respect to some fixed sequence of left invariant vector fields. 
We consider a neighborhood basis sequence $(K_j)$ of open compact subgroups of $K$ whose existence is guaranteed by Van Dantzig's theorem, see the proof of Lemma \ref{lem-disconnected-group-estimate}.
%We consider a word metric on the compact group $K$. Let $(K_j)$  be a basis of neighbourhoods of $e_K$ as in Lemma \ref{lem-Lie-group-estimate}
Furthermore, we let $(B_j)$ be a sequence as in Lemma \ref{lem-Lie-group-estimate} for the Lie group $L$. For any integer $j \geq 1$, we put 
\begin{equation}
\label{Def-V_j}
V_j 
\ov{\mathrm{def}}{=} K_j \times B_j
\end{equation}
and we can suppose that $K_j \subseteq W$ and that $B_j \subseteq U$ for any integer $j \geq 1$. Recall that each $K_j$ resp. each ball $B_j$ is symmetric, being a subgroup resp. a ball with respect to a left invariant metric\footnote{\thefootnote. Note that $\dist_\CC(e,s)=\dist_\CC(s^{-1},e)$.}. We conclude that $V_j$ is symmetric as well.

In view of the previous local product structure of the Haar measure described in \eqref{equ-1-proof-lem-finite-dimensional-group-estimate}, we deduce that for any integer $j$
\begin{align*}
\MoveEqLeft
c^3\int_{V_j} \mu_G(V_j \cap sV_j)^2 \d \mu_G(s) 
\ov{\eqref{Def-V_j}}{=} c^3\int_{K_j \times B_j} \mu_G(V_j \cap sV_j)^2 \d\mu_{K \times L}(s) \\
&\ov{\eqref{equ-1-proof-lem-finite-dimensional-group-estimate}, \eqref{Def-V_j}}{=} c^3\int_{K_j \times B_j} \mu_G\left(\big(K_j \times B_j\big) \cap (r,t)\big(K_j \times B_j\big)\right)^2 \d\mu_K(r) \d \mu_L(t) \\
&\ov{\eqref{equ-1-proof-lem-finite-dimensional-group-estimate}}{=} c^3 \int_{K_j \times B_j} \mu_K(K_j \cap rK_j)^2 \mu_L\left(B_j \cap t B_j\right)^2 \d\mu_K(r) \d\mu_L(t) \\
&= c^3\int_{K_j} \mu_K(K_j\cap rK_j)^2 \d\mu_K(r) \int_{B_j}\mu_L\left(B_j \cap t B_j\right)^2 \d\mu_L(t) \\
&\ov{\eqref{Lie-estimate} }{\geq} \mu_K(K_j)^3 \mu_G(B_j)^3 
\ov{\eqref{equ-1-proof-lem-finite-dimensional-group-estimate}}{=} \mu_G\left(K_j \times B_j\right)^3 
\ov{\eqref{Def-V_j}}{=} \mu_G(V_j)^3.
\end{align*}
%We denote by $n$ the dimension of $G$. According to Theorem \ref{cor-spitting-2}, %\cite[Theorem B p.~92]{Glu60}, \cite[Ex 1.6.8 p.~122]{Tao14}, \cite[Remark 39.5 (a) p.~283]{Str06} and \cite[p.~237]{MoZ55}\footnote{\thefootnote. See also \cite[Theorem 70 p.~337]{Pon66} for the case of compact groups.}, 
%$G$ is locally isomorphic to the product of a totally disconnected compact group $K$ and a Lie group $L$ of dimension $n$. So there exists a neighborhood $V$ of $e_G$ which is homeomorphic the direct product $W \times U$ of a neighborhood $W$ of the neutral element $e_K$ and an open neighborhood $U$ of the neutral element $e_L$. %By restriction, we can suppose that $U$ is connected and, by considering the connected component of $L$, that the Lie group $L$ is connected. 
%%there exists an open neighborhood $W_0 \subset W$ of $e_K$ and an open neighborhood $R_0 \subset R$ of $e_L$ such that \textbf{on doit pouvoir eviter de restreindre avec \cite[page 18-19]{Bou04b}}
\end{proof}

Now, we show the interest of the previous lemmas. 

\begin{prop}
\label{cor-2-referees-proof-step-1-weak-star-convergence-bis}
Let $G$ be a second-countable unimodular locally compact group. Suppose that $1 \leq p \leq \infty$. Let $(V_j)$ be a basis of symmetric neighborhoods of $e$ and a constant $c > 0$ such that
\begin{equation}
\label{subtil-estimate-bis}
 \mu_G(V_j)^3 
\leq c^3\int_{V_j} \mu_G(V_j \cap sV_j)^2 \d \mu_G(s), \quad j \geq 1. 
\end{equation}
Moreover, we put 
\begin{equation}
\label{def-fj}
f_j \ov{\mathrm{def}}{=} 1_{V_j} \ast 1_{V_j},\quad
x_j \ov{\mathrm{def}}{=} a_j \lambda(f_j) 
\quad \text{and} \quad y_j \ov{\mathrm{def}}{=} b_j \lambda(f_j)
\end{equation}
with
\begin{equation}
\label{aj}
a_j 
\ov{\mathrm{def}}{=} \norm{\lambda(f_j)}_{p^*} \norm{\lambda(f_j)}_2^{-2} 
\quad \text{and} \quad 
b_j \ov{\mathrm{def}}{=} \norm{\lambda(f_j)}_{p^*}^{-1} .
\end{equation}
Then the sequences $(x_j)$ and $(y_j)$ satisfy the assumptions from Proposition \ref{th-convergence}. More precisely, we have $\norm{x_j}_p \leq c^3$ and $\norm{y_j}_{p^*} =1$.
\end{prop}

\begin{proof}
We denote by $\mu_G$ a Haar measure on the group $G$. Consider any measurable subset $V$ of $G$ of measure $\mu_G(V) \in (0,\infty)$. Recall the isometric complex interpolation formula $\L^{p}(\VN(G))=(\L^\infty(\VN(G)), \L^2(\VN(G)))_{\frac{2}{p}}$ of \cite[(2.1) p.~1466]{PiX03}. In particular, we have by \cite[Corollary 2.8 p.~53]{Lun18}, the inequality $\norm{\cdot}_{\L^{p}(\VN(G))} \leq \norm{\cdot}_{\L^\infty(\VN(G))}^{1-\frac{2}{p}} \norm{\cdot}_{\L^2(\VN(G))}^{\frac{2}{p}} $. Using this inequality and Young's inequality \cite[Corollary 20.14 p.~293]{HeR79} in the second inequality, we obtain
%\footnote{\thefootnote. In the first inequality, we use $\norm{\cdot}_{(X,Y)_\theta} \leq \norm{\cdot}_{X}^{1-\theta}\norm{\cdot}_{Y}^\theta$ and the complex interpolation formula $\L^{2p}=(\L^\infty,\L^2)_{\frac{2}{p}}$.} 
\begin{align*}
\MoveEqLeft
\norm{\lambda(1_{V^{-1}} \ast 1_{V})}_p
\ov{\eqref{composition-et-lambda}}{=} 
\norm{\lambda(1_{V})^*\lambda(1_{V})}_p
=\norm{|\lambda(1_{V})|^2}_p
= \norm{ \lambda(1_V)}_{2p}^2 
\leq \norm{\lambda(1_V)}_\infty^{\frac{2}{p^*}}\norm{\lambda(1_V)}_2^{\frac{2}{p}}  \\
&=\norm{\lambda(1_V)}_{\L^2(G) \to \L^2(G)}^{\frac{2}{p^*}} \norm{\lambda(1_V)}_2^{\frac{2}{p}}
\leq \mu_G(V)^{\frac{2}{p^*}} \cdot \mu_G(V)^{\frac{1}{p}} 
= \mu_G(V)^{1 + \frac{1}{p^*}}.
\end{align*}
Taking $V = V_j$ and using the previous calculation also for $p^*$ in place of $p$, we deduce with Lemma \ref{lem-finite-dimensional-group-estimate} that
\begin{align}
\label{equ-1-proof-cor-2-referees-proof-step-1-weak-star-convergence}
\MoveEqLeft
\norm{\lambda(1_{V_j} \ast 1_{V_j})}_p \cdot \norm{\lambda(1_{V_j} \ast 1_{V_j})}_{p^*} 
\leq \mu_G(V_j)^{1 + \frac{1}{p^*} + 1 + \frac{1}{p}} \\
&= \mu_G(V_j)^{3} 
\ov{\eqref{subtil-estimate-bis}}{\leq} c^3\int_{V_j} \mu_G(V_j \cap sV_j)^2 \d \mu_G(s).\nonumber 
\end{align}
On the other hand, using $V_j = V_j^{-1}$, we obtain
\begin{align}
\MoveEqLeft
\label{equ-1-proof-cor-2-referees-proof-step-1-weak-star}
\norm{\lambda(1_{V_j} \ast 1_{V_j})}_2^2 
\ov{\eqref{Convolution-formulas}}{=} \int_G \left| \int_G 1_{V_j}(t)1_{V_j}(t^{-1}s) \d \mu_G(t) \right|^2 \d \mu_G(s) \nonumber\\
&= \int_G \left| \int_G 1_{V_j}(t)1_{V_js^{-1}}(t^{-1}) \d \mu_G(t) \right|^2 \d \mu_G(s) 
=\int_G \left| \int_G 1_{V_j}(t)1_{sV_j}(t) \d \mu_G(t) \right|^2 \d \mu_G(s) \nonumber \\
&= \int_G \mu_G^2(V_j  \cap sV_j) \d\mu_G(s).
\end{align}
Combining \eqref{equ-1-proof-cor-2-referees-proof-step-1-weak-star-convergence} and \eqref{equ-1-proof-cor-2-referees-proof-step-1-weak-star}, we see that
\begin{equation}
\label{divers-500}
\norm{\lambda(1_{V_j} \ast 1_{V_j})}_p  \cdot \norm{\lambda(1_{V_j} \ast 1_{V_j})}_{p^*} 
\leq c^3 \norm{\lambda(1_{V_j} \ast 1_{V_j})}_2^2. 
\end{equation}
Note that with the choice of $a_j$ and $b_j$, we finally obtain
\[ 
\norm{x_j}_p 
\ov{\eqref{def-fj}}{=} \norm{a_j \lambda(f_j)}_p 
\ov{\eqref{aj}}{=} \norm{\lambda(f_j)}_{p^*} \norm{\lambda(f_j)}_2^{-2}\norm{\lambda(f_j)}_p 
\ov{\eqref{def-fj} \eqref{divers-500}}{\leq} c^3 
\]
and $\norm{y_j}_{p^*} \ov{\eqref{def-fj}}{=} \norm{b_j \lambda(f_j)}_{p^*} = 1$, as well as $\tau_G(x_j y_j) \ov{\eqref{aj}}{=} \norm{\lambda(f_j)}_2^{-2} \tau_G(\lambda(f_j)^2 ) = 1$.
\end{proof}

%\begin{remark} \normalfont
%\label{rem-2-referees-proof-step-1-weak-star-convergence}
%Suppose that $1 \leq p \leq \infty$ and consider a second-countable totally disconnected locally compact group $G$. Then the sequences $(x_j)$ and $(y_j)$ from Corollary \ref{cor-2-referees-proof-step-1-weak-star-convergence-bis}, satisfy $\norm{x_j}_p \leq 1$ and $ \norm{y_j}_{p^*} = 1$, with respect to the choice of $V_j=K_j$, where $K_j$ is defined in Lemma \ref{lem-disconnected-group-estimate}. 
%%Indeed, we have for any $j$
%%\begin{align*}
%%\MoveEqLeft
%%\norm{x_j}_{p} 
%%\ov{\eqref{def-fj}}{=} \bnorm{a_j \lambda(f_j)}_p 
%%\ov{\eqref{def-fj} \eqref{aj}}{=} \norm{\lambda(1_{V_j} \ast 1_{V_j})}_{p^*} \norm{\lambda(1_{V_j} \ast 1_{V_j})}_2^{-2} \norm{\lambda(1_{V_j}\ast 1_{V_j})}_p \\
%%&  \overset{\eqref{equ-1-proof-cor-2-referees-proof-step-1-weak-star-convergence} \eqref{equ-2-proof-cor-2-referees-proof-step-1-weak-star-convergence}}{\leq} \mu_G(V_j)^3 \left( \int_G \mu_G^2\left(V_j \cap sV_j \right) \d\mu_G(s) \right)^{-2} 
%%?\ov{\eqref{disco-estim}}{=}? 1.
%%\end{align*}
%\end{remark}

A combination of Proposition \ref{cor-2-referees-proof-step-1-weak-star-convergence-bis} and the previous lemmas gives the next result, which is the main result of this section.

\begin{cor}
\label{Cor-38}
\begin{enumerate}
\item Let $G$ be a second-countable finite-dimensional unimodular locally compact group equipped with a Haar measure $\mu_G$. Then there exist sequences $(x_j)$ and $(y_j)$ satisfying the assumptions from Proposition \ref{th-convergence}. More precisely, we have $\norm{x_j}_p \leq c^3$ and $\norm{y_j}_{p^*} =1$ for any integer $j$ for some constant $c>0$.
\item Let $G$ be a second-countable totally disconnected unimodular locally compact group equipped with a Haar measure $\mu_G$. Then there exist sequences $(x_j)$ and $(y_j)$ satisfying the assumptions from Proposition \ref{th-convergence} with $\norm{x_j}_p \leq 1$ and $\norm{y_j}_{p^*} =1$ for any integer $j$.
\end{enumerate}
\end{cor}

%\begin{remark} \normalfont
%If $G$ is a second-countable finite-dimensional locally compact group, we can use Corollary \ref{cor-spitting-2} instead of Corollary \ref{cor-spitting-4}, the obtained constant $c$ is better.
%\end{remark}
%
%\begin{remark} \normalfont
%In some cases, we could estimate the doubling constant.
%\end{remark}

%%%%%%%%%%%%%%%%%%%%%%%%%%%%%%%%%%%%%%%%%%%%%%%%%%%%%%%%%%%%%%%%%%%%%%%%%%%%%%%%%%%%%%
\subsection{Step 3: the projection on the space of Herz-Schur multipliers}
\label{Section-Schur}

The first part of the following result says that the unit ball of the space $\mathfrak{M}^{p}_\Omega$ of measurable Schur multipliers is closed for the weak* topology of the dual Banach space $\L^\infty(\Omega \times \Omega)$.

\begin{lemma}
\label{Lemma-symbol-weak}
Let $\Omega$ be a $\sigma$-finite measure space. Suppose that $1 \leq p \leq \infty$. Let $(M_{\phi_j})$ be a bounded net of bounded Schur multipliers on the Schatten class $S^p_\Omega$ and suppose that $\phi$ is an element in $\L^\infty(\Omega \times \Omega)$ such that the net $(\phi_j)$ converges to $\phi$ for the weak* topology of $\L^\infty(\Omega \times \Omega)$. Then the function $\phi$ induces a bounded Schur multiplier on $S^p_\Omega$. Moreover, the net $(M_{\phi_j})$ converges to the operator $M_{\phi}$ for the weak operator topology of the space $\cal{B}(S^p_\Omega)$ (point weak* topology if $p = 1$) and
\begin{equation}
\label{estim-divers-35}
\norm{M_{\phi}}_{S^p_\Omega \to S^p_\Omega}
\leq \liminf_{j \to \infty} \norm{M_{\phi_j}}_{S^p_\Omega \to S^p_\Omega}.
\end{equation}
%In addition if $1<p<\infty$ 
%$$
%If $p=\infty$, for any functions $f \in ??$ and $g \i ???$, we have
%$$
%\la M_{\phi_j}() , \ra_{} \xra[\ j \ ]{} \la M_{\phi}() , \ra_{}.
%$$
A similar statement is true upon replacing <<bounded>> by <<completely bounded>> and the norm $\norm{\cdot}_{S^p_\Omega \to S^p_\Omega}$ by the norm $\norm{\cdot }_{\cb, S^p_\Omega \to S^p_\Omega}$.
\end{lemma}

\begin{proof}
Consider some functions $f,g \in \L^2(\Omega \times \Omega)$ such that $K_{f} \in S^p_\Omega$ and $K_{g} \in S^{p^*}_\Omega$. Note that we have $f\check{g} \in \L^1(\Omega \times \Omega)$. For any $j$, we have 
\begin{align*}
\MoveEqLeft
\left| \int_{\Omega \times \Omega} \phi_j f \check{g} \right| 
\ov{\eqref{dual-trace}}{=}\left| \big\langle M_{\phi_j}(K_f) , K_g \big\rangle_{S^p_\Omega,S^{p^*}_\Omega} \right| 
\leq \bnorm{M_{\phi_j}(K_f)}_{S^p_\Omega} \bnorm{K_g}_{S^{p^*}_\Omega} \\
&\leq \norm{M_{\phi_j}}_{S^p_\Omega \to S^p_\Omega} \bnorm{K_f}_{S^p_\Omega} \bnorm{K_g}_{S^{p^*}_\Omega}.
\end{align*}
Passing to the limit, we obtain 
\begin{align*}
\MoveEqLeft
\left| \big\langle  K_{\phi f} ,K_g\big\rangle_{S^p_\Omega,S^{p^*}_\Omega} \right|\ov{\eqref{dual-trace}}{=}  \left| \int_{\Omega \times \Omega} \phi f \check{g} \right| 
\leq \liminf_{j \to \infty} \norm{M_{\phi_j}}_{S^p_\Omega \to S^p_\Omega} \bnorm{K_f}_{S^p_\Omega} \bnorm{K_g}_{S^{p^*}_\Omega}.
\end{align*}
By density, we conclude that the function $\phi$ induces a bounded Schur multiplier on $S^p_\Omega$ with the estimate \eqref{estim-divers-35} on the norm of this operator. Using again the weak* convergence of the net $(\phi_j)$, we see  that for any functions $f, g \in \L^2(\Omega \times \Omega)$ such that $K_f \in S^p_\Omega$ and $K_g \in S^{p^*}_\Omega$
\begin{align*}
\MoveEqLeft
\tr\big((M_{\phi}-M_{\phi_j})(K_f)K_g\big)  
=\tr\big(K_{(\phi-\phi_j)f} K_g)\big) 
\ov{\eqref{dual-trace}}{=} \iint_{\Omega \times \Omega} (\phi-\phi_j) f\check{g} \\
&=\big\langle \phi-\phi_j,f\check{g} \big\rangle_{\L^\infty(\Omega \times \Omega),\L^1(\Omega \times \Omega)} 
\xra[\ j \ ]{} 0.            
\end{align*} 
By density, using an $\frac{\epsi}{4}$-argument and the boundedness of the net, we conclude\footnote{\thefootnote. More precisely, if $X$ is a Banach space, if $E_1$ is dense subset of $X$, if $E_2$ is a dense subset of $X^*$ and if $(T_j)$ is a bounded net of $\cal{B}(X)$ with an element $T$ of $\cal{B}(X)$ such that $\langle T_j(x),x^*\rangle \xra[i \to +\infty]{}  \langle T(x), x^*\rangle$ for any $x \in E_1$ and any $x^* \in E_2$, then the net $(T_j)$ converges to $T$ for the weak operator topology of the space $\cal{B}(X)$.} that the net $(M_{\phi_j})$ converges to the operator $M_{\phi}$ for the weak operator topology of $\cal{B}(S^p_\Omega)$ (point weak* topology if $p = 1$).

Now, we prove the last sentence. For any functions $f_{kl},g_{kl} \in \L^2(\Omega \times \Omega)$ where $1 \leq k,l \leq N$, we have $f_{kl} \check{g}_{kl}\in \L^1(\Omega \times \Omega)$. For any $j$, we infer by \cite[Theorem 4.7 p.~49]{Pis98} that
\begin{align*}
\MoveEqLeft
\left| \big\langle \big[ M_{\phi_{j}}(K_{f_{kl}})\big] , \big[K_{g_{kl}}\big] \big\rangle_{\M_N(S^p_\Omega),S^1_N(S^{p^*}_\Omega)} \right|  
\leq \norm{M_{\phi_j}}_{\cb,S^p_\Omega \to S^p_\Omega} \bnorm{\big[K_{f_{kl}}\big]}_{\M_N(S^p_\Omega)} \bnorm{\big[K_{g_{kl}}\big]}_{S^1_N(S^{p^*}_\Omega)},
\end{align*}
that is,
\begin{align*}
\MoveEqLeft
\left|\sum_{k,l=1}^N \int_{\Omega \times \Omega} \phi_j f_{kl} \check{g}_{kl} \right| 
\leq \norm{M_{\phi_j}}_{\cb,S^p_\Omega \to S^p_\Omega} \bnorm{\big[K_{f_{kl}}\big]}_{\M_N(S^p_\Omega)} \bnorm{\big[K_{g_{kl}}\big]}_{S^1_N(S^{p^*}_\Omega)}.
\end{align*}
Passing to the limit, we obtain
\begin{align*}
\MoveEqLeft
\left|\sum_{k,l=1}^N \int_{\Omega \times \Omega} \phi f_{kl} \check{g}_{kl} \right| 
\leq \liminf_{j \to \infty} \norm{M_{\phi_j}}_{\cb,S^p_\Omega \to S^p_\Omega} \bnorm{\big[K_{f_{kl}}\big]}_{\M_N(S^p_\Omega)} \bnorm{\big[K_{g_{kl}}\big]}_{S^1_N(S^{p^*}_\Omega)}.            
\end{align*}  
We deduce that the function $\phi$ induces a completely bounded Schur multiplier on the Schatten space $S^p_\Omega$ with the suitable estimate on the completely bounded norm.
\end{proof}

If $1 \leq p < \infty$, note that the Schatten space $S^p_\Omega$ is a dual Banach space. So the Banach space $\CB(S^p_\Omega)$ is also a dual space with predual $S^p_\Omega \widehat{\ot} S^{p^*}_\Omega$, where $\widehat{\ot}$ denotes the operator space projective tensor product and the duality bracket is given by
\begin{equation}
\label{Belle-dualite}
\langle T, x \ot y \rangle_{\CB(S^p_\Omega),S^p_\Omega \widehat{\ot} S^{p^*}_\Omega} 
=\big\langle T(x), y \big\rangle_{S^p_\Omega, S^{p^*}_\Omega}.
\end{equation}

\begin{lemma}
\label{lem-Schur-weak-star-closed}
Let $\Omega$ be a $\sigma$-finite measure space.
\begin{enumerate}
	\item Let $1 \leq p < \infty$. Then the space $\mathfrak{M}^{p,\cb}_\Omega$ of completely bounded Schur multipliers is weak* closed in $\CB(S^p_\Omega)$ and the space $\mathfrak{M}^{p}_\Omega$ of bounded Schur multipliers is weak* closed in $\cal{B}(S^p_\Omega)$.

\item The space $\mathfrak{M}^{\infty,\cb}_\Omega=\mathfrak{M}^{\infty}_\Omega$ of (completely) bounded Schur multipliers is weak* closed in the space $\CB(S^\infty_\Omega,\cal{B}(\L^2(\Omega)))$.% and $\cal{B}(S^\infty_\Omega,\cal{B}(\L^2(\Omega)))$.
%nd $\mathfrak{M}^{\infty}_\Omega$ are
\end{enumerate}
\end{lemma}

\begin{proof}
We start by proving the first assertion. By the Banach-Dieudonn\'e theorem \cite[p.~154]{Hol75}, it suffices to show that the closed unit ball of the space $\mathfrak{M}^{p,\cb}_\Omega$ is weak* closed in the space $\CB(S^p_\Omega)$. Let $(M_{\phi_j})$ be a net in that unit ball converging for the weak* topology to some completely bounded map $T \co S^p_\Omega \to S^p_\Omega$. We have for any $j$ the inequality 
$$
\norm{\phi_j}_{\L^\infty(\Omega \times \Omega)} 
\ov{\eqref{ine-infty}}{\leq}  \norm{M_{\phi_j}}_{\cb,S^p_\Omega \to S^p_\Omega} 
\leq 1.
$$ 
By Banach-Alaoglu's theorem, there exists a subnet of $(\phi_j)$ converging for the weak* topology to some function $\phi \in \L^\infty(\Omega \times \Omega)$. It remains to show that $T=M_\phi$. By \eqref{Belle-dualite}, we have $\big\langle M_{\phi_j}(x), y\big\rangle \xra[j]{} \langle T(x),y \rangle$ for any $x \in S^p_\Omega$ and any $y \in S^{p^*}_\Omega$. That means that the net $(M_{\phi_j})$ converges to $T$ for the weak operator topology (point weak* topology if $p = 1$). By Lemma \ref{Lemma-symbol-weak},  the net $(M_{\phi_j})$ converges to $M_\phi$. We conclude by uniqueness of the limit that $T = M_\phi$.
%For any $1 \leq k,l \leq N$ consider some $f_{kl},g_{kl} \in \L^2(\Omega \times \Omega)$ such that $K_{f_{kl}} \in S^p_\Omega$, $K_{g_{kl}} \in S^{p^*}_\Omega$. For any $j$, we have 
%\begin{align*}
%\MoveEqLeft
%\left|\sum_{k,l=1}^N \int_{\Omega \times \Omega} \phi_{j} f_{kl} \check{g}_{kl} \right| 
%\ov{\eqref{dual-trace}}{=}\left| \big\langle \big[ M_{\phi_{j}}(K_{f_{kl}})\big] , \big[K_{g_{kl}}\big] \big\rangle_{\M_N(S^p_\Omega),S^1_N(S^{p^*}_\Omega)} \right| \\  
%&\leq \bnorm{\big[M_{\phi_{j}}(K_{f_{kl}})\big]}_{\M_N(S^p_\Omega)} \bnorm{\big[K_{g_{kl}}\big]}_{S^1_N(S^{p^*}_\Omega)}\\
%& \leq \Big(\sup_j \norm{M_{\phi_j}}_{\cb,S^p_\Omega \to S^p_\Omega}\Big) \bnorm{\big[K_{f_{kl}}\big]}_{\M_N(S^p_\Omega)} \bnorm{\big[K_{g_{kl}}\big]}_{S^1_N( S^{p^*}_\Omega)} \\
%&\leq \bnorm{\big[K_{f_{kl}}\big]}_{\M_N(S^p_\Omega)} \bnorm{\big[K_{g_{kl}}\big]}_{S^1_N( S^{p^*}_\Omega)}.
%\end{align*}
%Passing to the limit, we obtain 
%\begin{align*}
%\MoveEqLeft
%\left| \big\langle \big[ K_{\phi f_{kl}}\big] , \big[K_{g_{kl}}\big] \big\rangle_{\M_N(S^p_\Omega),S^1_N(S^{p^*}_\Omega)} \right|\ov{\eqref{dual-trace}}{=}\left|\sum_{k,l=1}^N \int_{\Omega \times \Omega} \phi f_{kl} \check{g}_{kl} \right| 
%\leq \bnorm{\big[K_{f_{kl}}\big]}_{\M_N(S^p_\Omega)} \bnorm{\big[K_{g_{kl}}\big]}_{S^1_N( S^{p^*}_\Omega)}.
%\end{align*}
%By density, we conclude that $\phi$ induces a completely bounded measurable Schur multiplier on $S^p_\Omega$. 

The statement on the space $\mathfrak{M}^{p}_\Omega$ can be proved in a similar manner, using the predual $S^p_\Omega \hat{\ot} S^{p^*}_\Omega$ of the dual Banach space $\cal{B}(S^p_\Omega)$, where $\hat{\ot}$ denotes the Banach space projective tensor product. The second point is also similar.
%\textbf{Old:}
%
%Let $(T_\alpha)_\alpha$ be a net in $\mathfrak{M}^{p,\cb}_\Omega$ converging to $T \in \CB(S^p_\Omega)$ in the weak$^*$ topology.
%
%We check that $T$ is a Schur multiplier using Lemma \ref{lem-Schur-disjoint-supports}.
%Let $K_\phi,K_\psi \in S^1_\Omega)$ such that $\supp \phi \cap \supp \psi = \emptyset$.
%Then
%\[ 
%\langle T K_\phi, K_\psi \rangle 
%= \lim_\alpha \langle T_\alpha K_\phi, K_\psi \rangle 
%= \lim_\alpha 0 = 0 ,
%\]
%since the $T_\alpha$ are Schur multipliers, hence satisfy the criterion from Lemma \ref{lem-Schur-disjoint-supports}.
%
%The proof for the bounded case and the case $p = \infty$ is the same.
\end{proof}

The following is essentially folklore. The case $p=\infty$ is explicitly proved in \cite[Proposition 5.2 p.~375]{SpT02} and \cite[Corollary 5.4 p.~183]{Spr04} with a slightly different method relying on the use of an invariant mean. We sketch a proof since it is important for us.

\begin{prop}
\label{prop-referee-step-2}
Let $G$ be an amenable unimodular locally compact group. Suppose that $1 \leq p \leq \infty$. Then there exists a contractive projection $Q \co \mathfrak{M}^{p,\cb}_G \to \mathfrak{M}^{p,\cb}_{G}$ onto the space $\mathfrak{M}^{p,\cb,\HS}_{G}$ of completely bounded Herz-Schur multipliers acting on the Schatten space $S^p_G$ ($\cal{B}(\L^2(G))$ if $p=\infty$), preserving the complete positivity. Moreover, the obtained projections are compatible for all different values of $1 \leq p \leq \infty$.
\end{prop}

\begin{proof}
Let $(F_j)$ be a F\o{}lner net in $G$ provided by the amenability of the group $G$. For any $K_f \in S^2_G \cap S^p_G$, the map $G \to S^p_G$, $r \mapsto \Ad(\rho_r) (K_f)$ is continuous since the composition of operators is strongly continuous on bounded sets by \cite[Proposition C.19 p.~517]{EFHN15} (recall the notation $\Ad(\rho_s)(x)=\rho_s x \rho_{s^{-1}}$). Similarly for any $M_\phi \in \mathfrak{M}^{p,\cb}_G$, the map $G \to S^p_G$, $r \mapsto \big[\Ad(\rho_r^*) M_\phi \Ad(\rho_r) \big](K_f)$ is also continuous, hence Bochner integrable on the compact $F_j$. Now, for any $K_f \in S^2_G \cap S^p_G$ and any $M_\phi \in \mathfrak{M}^{p,\cb}_G$, put
\begin{equation}
\label{Equa33}
Q_j(M_\phi)(K_f) 
=\frac{1}{\mu_G(F_j)} \int_{F_j} \big[\Ad(\rho_r^*) M_\phi \Ad(\rho_r) \big](K_f) \d\mu_G(r).	
\end{equation}

For any $K_f \in S^2_G \cap S^p_G$ and any completely bounded Schur multiplier $M_\phi \co S^p_G \to S^p_G$, we have
\begin{align*}
\MoveEqLeft
  \bnorm{Q_j(M_\phi)(K_f)}_{S^p_G} 
		=\frac{1}{\mu_G(F_j)}\norm{\int_{F_j} \big[\Ad(\rho_r^*) M_\phi \Ad(\rho_r) \big](K_f)\d\mu_G(r)}_{S^p_G}\\
		&\leq \frac{1}{\mu_G(F_j)}\int_{F_j} \norm{\big[\Ad(\rho_r^*) M_\phi \Ad(\rho_r) \big](K_f)}_{S^p_G} \d\mu_G(r)
		\leq  \norm{M_\phi}_{S^p_G \to S^p_G}\norm{K_f}_{S^p_G}.
\end{align*}
A similar argument shows that $\norm{Q_j(M_\phi)}_{\cb, S^p_G \to S^p_G}\leq \norm{M_\phi}_{\cb, S^p_G \to S^p_G}$. Consequently, we have a well-defined contractive map $Q_j \co \mathfrak{M}^{p,\cb}_G \to \CB(S^p_G)$, $M_\phi \mapsto \frac{1}{\mu_G(F_j)} \int_{F_j} \Ad(\rho_r^*) M_\phi \Ad(\rho_r) \d\mu_G(r)$. If the linear map $M_\phi$ is completely positive then observe that the map $\big[\Ad(\rho_r^*) M_\phi \Ad(\rho_r)\big]$ is also completely positive. Thus the map $Q_j$ preserves the complete positivity. It is easy to check that $Q_j(M_\phi)$ is a Schur multiplier with symbol
\begin{equation}
\label{Divers-234}
\phi_j(s,t) 
\ov{\mathrm{def}}{=} \frac{1}{\mu_G(F_j)} \int_{F_j}\phi(sr,tr) \d\mu_G(r) 
\end{equation}
(Gelfand integral in $\L^\infty(G \times G)$).
%\textbf{As essentially observed (and well-known) in} \cite[page 27]{Tod}, for fixed $r \in G$ and any measurable Schur multiplier $M_\phi \co S^p_G \to S^p_G$ we can write
%$$
%M_{\phi(sr,tr)} 
%= \Ad(\rho_r^*) M_\phi \Ad(\rho_r)
%$$ 
%where $\Ad(\rho_r) \co S^p_G \to S^p_G$ is a conjugation with a unitary.Moreover, $\norm{M_{\phi(sr,tr)}}_{\cb,p \to p}  = \norm{M_\phi}_{\cb,p \to p}$.

We continue with the case $1 \leq p < \infty$. Since the space $\mathfrak{M}^{p,\cb}_G$ is weak* closed in $\CB(S^p_\Omega)$, the space $\mathfrak{M}^{p,\cb}_G$ is a dual Banach space. Hence $\cal{B}( \mathfrak{M}^{p,\cb}_G,\mathfrak{M}^{p,\cb}_G )$ is a dual space. By Banach-Alaoglu's theorem, the uniformly bounded net $(Q_j)$ admits a weak* accumulation point that we denote by $Q$ which is obviously a contraction. So we can suppose that $Q_j \to Q$ for the weak* topology. So, for each completely bounded Schur multiplier $M_\phi \co S^p_G \to S^p_G$ this implies that $Q(M_\phi) = \lim_{j} Q_{j}(M_\phi)$ in the weak operator topology. Recall that the weak* topology on $\CB(S^p_G)$ coincides on bounded subsets with the point weak* topology. Since $Q_{j}(M_\phi)$ belongs to the space $\mathfrak{M}^{p,\cb}_G$ and since the latter space is weak* closed in $\CB(S^p_G)$ according to Lemma \ref{lem-Schur-weak-star-closed}, we obtain that $Q(M_\phi)$ also belongs to the space $\mathfrak{M}^{p,\cb}_G$. Since each map $Q_j$ preserves complete positivity, by \cite[Lemma 2.10 2.~p.~15]{ArK23}, the map $Q$ also preserves complete positivity.

For any completely bounded Schur multiplier $M_\phi \co S^p_G \to S^p_G$, it remains to show that $Q(M_\phi)$ is in addition a Herz-Schur multiplier. That is, for any $r_0 \in G$ we have to show that $\lim_{j} M_{\phi_{j}(sr_0,tr_0)} = \lim_{j} M_{\phi_{j}(s,t)}$. Fix some $r_0 \in G$ and some $j$. Using the F\o{}lner condition in the last line, we  have
\begin{align*}
\MoveEqLeft
\norm{M_{\phi_{j}(s,t)} - M_{\phi_{j}(sr_0,tr_0)} }_{\cb,S^p_G \to S^p_G} \\
&\ov{\eqref{Divers-234}}{=} \norm{\frac{1}{\mu_G(F_j)} \int_{F_j} M_{\phi(sr,tr)} \d\mu_G(r) - \frac{1}{\mu(F_j)} \int_{F_j} M_{\phi(sr_0u,tr_0u)} \d\mu_G(u) }_{\cb,S^p_G \to S^p_G} \\
& = \frac{1}{\mu_G(F_j)} \norm{\int_{F_j} M_{\phi(sr,tr)} \d\mu_G(r) - \int_{r_0 F_j} M_{\phi(sr,tr)} \d\mu_G(r)}_{\cb,S^p_G \to S^p_G} \\
& \leq \frac{1}{\mu_G(F_j)} \int_{F_j \bigtriangleup r_0 F_j} \norm{M_{\phi(sr,tr)}}_{\cb,S^p_G \to S^p_G} \d\mu_G(r) \\
&= \frac{\mu_G(F_j \bigtriangleup r_0 F_j)}{\mu_G(F_j)} \norm{M_\phi}_{\cb,S^p_G \to S^p_G}
\xra[j \to \infty]{} 0.
\end{align*}
Using the weak* lower semicontinuity of the norm \cite[Theorem 2.6.14 p.~227]{Meg98}, we infer that
\[ 
\norm{Q(M_\phi)-Q(M_{\phi_{(\cdot r_0, \cdot r_0)}})}_{\cb,S^p_G \to S^p_G} 
\leq \liminf_{j} \norm{M_{\phi_{j}(s,t)} - M_{\phi_{j}(sr_0,tr_0)}}_{\cb,S^p_G \to S^p_G} 
= 0. 
\]
Finally, it is easy to see that $Q(M_{\phi_{(\cdot r_0, \cdot r_0)}})=Q(M_{\phi})_{(\cdot r_0, \cdot r_0)}$. The case $p = \infty$ is similar.

 %then consider the operator $Q^{(\infty)} = Q^{(2)} \circ J \co \mathfrak{M}^{\infty,\cb}_G \to \mathfrak{M}^{2,\cb}_G \to \mathfrak{M}^{2,\cb,\HS}_G$, where $Q^{(2)}$ is the above mapping obtained for $p = 2$ and $J$ is the contractive natural injection.
%We still have 
%
%
%\[ 
%\norm{Q^{(\infty)}}_{\mathfrak{M}^{\infty,\cb}_G \to \mathfrak{M}^{\infty,\cb}_G} 
%\leq \sup_\alpha \norm{Q_\alpha}_{\mathfrak{M}^{\infty,\cb}_G \to \mathfrak{M}^{\infty,\cb}_G} 
%\leq 1, 
%\]
%so that $Q^{(\infty)}$ actually maps into $\mathfrak{M}^{\infty,\cb,\HS}_G$.
%It is easy to see that $Q^{(\infty)}$ is a projection and preserves complete positivity (now a consequence of the weak operator topology convergence $Q_{\alpha'} \to Q^{(\infty)}$ and \cite[Lemma 2.10 1.]{ArK23}) as well.

In order to obtain that the mappings $Q^{(p)} \co \mathfrak{M}^{p,\cb}_G \to \mathfrak{M}^{p,\cb}_G$ are compatible for different values of $1 \leq p \leq \infty$, it suffices to observe that we can choose the indices $j'$ in the converging subnet $Q_{j'}^{(p)}$ independent of $p$, in the same manner as done in the proof of Corollary \ref{cor-the-compatible-complementation} below by means of an argument relying on Tychonoff's theorem. The proof is complete.
\end{proof}

\begin{remark} \normalfont
\label{Remark-Herz-Schur-amenability}
We have a similar result for spaces of bounded Schur multipliers.
\end{remark}

\begin{remark} \normalfont
\label{Remark-Herz-Schur-amenability-bis}
If $G$ is compact, the proof is simpler. We do not need to use an approximation procedure. See \cite[Proposition 2.3 p.~365]{SpT02} for the case $p=\infty$.
\end{remark}

%%%%%%%%%%%%%%%%%%%%%%%%%%%%%%%%%%%%%%%%%%%%%%%%%%%%%%%%%%%%%%%%%%%%%%%%%%%%%%%%%%%%%%%%%%%%%%%%
\subsection{Combining Steps 1-3: Complementation theorems}
\label{Sec-Th-complementation}

Let $G$ be a locally compact group. Note that the space $\cal{B}(\CB(\VN(G)),\CB(S^\infty_G,\cal{B}(\L^2(G))))$ is a dual space and admits the predual
\begin{equation*}
\label{equ-predual-bracket}
\CB(\VN(G)) \hat \ot \big(S^\infty_G \widehat{\ot} S^1_G \big),
\end{equation*}
where $\hat{\ot}$ denotes the Banach space projective tensor product and where $\widehat{\ot}$ denotes the operator space projective tensor product. The duality bracket is given by
\begin{equation}
\label{Duality-bracket-gros}
\big\langle P , T \ot (x \ot y) \big\rangle
=\big\langle P(T) x, y \big\rangle_{\cal{B}(\L^2(G)),S^1_G}.
\end{equation}

Now, we prove one of our main results.

\begin{thm}
\label{thm-SAIN-tilde-kappa}
Let $G$ be a second-countable unimodular inner amenable locally compact group. Then $G$ has property $(\kappa_\infty)$ with $\kappa_\infty(G) = 1$. More precisely, there exists a contractive projection $P_{G}^\infty \co \CB_{\w^*}(\VN(G)) \to \CB_{\w^*}(\VN(G))$ preserving the complete positivity onto the space $\mathfrak{M}^{\infty,\cb}(G)$ of completely bounded Fourier multipliers on the group von Neumann algebra $\VN(G)$.
\end{thm}

\begin{proof}
Fix some finite subset $F$ of the group $G$. We can consider a sequence $(V_j^F)_j$ of subsets of $G$ satisfying the last point of Theorem \ref{thm-inner-amenable-Folner}. As in Lemma \ref{lem-SAIN-Herz-Schur}, if $T \co \VN(G) \to \VN(G)$ is a weak* continuous completely bounded map then we consider the elements $y_j^F \ov{\eqref{Def-ds-inner}}{=} c_j^F |\lambda(1_{V_j^F}) |^2$ in $\L^1(\VN(G)) \cap \VN(G)$ and the symbol $\phi_{j,T}^F \ov{\eqref{Def-ds-inner}}{=} \varphi_{1,y_j^F,T}$.
Recall that $\norm{y_j^F}_{\L^1(\VN(G))} = 1$.

\paragraph{Step 1} Consider the mapping $P_j^F \co \CB(\VN(G)) \to \CB(S^\infty_G,\cal{B}(\L^2(G)))$, $T \mapsto M_{\phi_{j,P_{\w^*}(T)}^F}$, where the projection $P_{\w^*} \co \CB(\VN(G)) \to \CB(\VN(G))$, preserving the complete positivity, is defined in \cite[Proposition 3.1 p.~24]{ArK23}. By Lemma \ref{Lemma-estimation-cb}, we have the estimate
$$
\bnorm{M_{\phi_{j,T}^F}}_{\cb,S^\infty_G \to \cal{B}(\L^2(G))} 
=\bnorm{M_{\varphi_{1,y_j^F,T}}}_{\cb,S^\infty_G \to \cal{B}(\L^2(G))}
\ov{\eqref{div-987}}{\leq} \norm{T}_{\cb,\VN(G) \to \VN(G)}.
$$ 
Hence the maps $P_j^F$ belong to the unit ball of the space $\cal{B}(\CB(\VN(G)),\CB(S^\infty_G,\cal{B}(\L^2(G))))$. By Banach-Alaoglu's theorem, we can introduce a weak* accumulation point $P^{F} \co \CB(\VN(G)) \to \CB(S^\infty_G,\cal{B}(\L^2(G)))$. So, we have a net $(P_{j(k)}^F)$ which converges to $P^F$ in the weak* topology. Taking into account \eqref{Duality-bracket-gros}, this implies that the bounded net $(P_{j(k)}^F(T))$, that is $\big(M_{\phi_{j(k),P_{\w^*}(T)}^F}\big)$, converges in the point weak* topology of the space $\CB(S^\infty_G,\cal{B}(\L^2(G)))$ to $P^F(T)$. Since the weak* topology on the space $\CB(S^\infty_G,\cal{B}(\L^2(G)))$ coincides, essentially by the same proof as the one of \cite[Lemma 7.2 p.~85]{Pau02}, on bounded subsets with the point weak* topology, we conclude by the second part of Lemma \ref{lem-Schur-weak-star-closed} that the map $P^F(T) \co S^\infty_G \to \cal{B}(\L^2(G))$ is itself a Schur multiplier. %Hence we can consider the restriction $P^F=P^F|_{\CB_{\w^*}(\VN(G))} \co \CB_{\w^*}(\VN(G)) \to \CB(S^\infty_G,\cal{B}(\L^2(G)))$ with the same notation. 
Note that by the weak* lower semicontinuity of the norm \cite[Theorem 2.6.14 p.~227]{Meg98}, we have 
$$
\bnorm{P^F}_{\CB(\VN(G)) \to \CB(S^\infty_G,\cal{B}(\L^2(G)))} 
\leq \liminf_{k \to \infty} \bnorm{P_{j(k)}^F}_{\CB(\VN(G)) \to \CB(S^\infty_G,\cal{B}(\L^2(G)))} 
\leq 1.
$$ 
We next show that the map $P^F$ preserves the complete positivity. Suppose that the map $T$ is completely positive. Using Lemma \ref{Lemma-estimation-cb}, we see that each map $M_{\phi_{j,P_{\w^*}(T)}^F}$ is completely positive. Since $P^F(T)$ is the limit in the point weak* topology of the $M_{\phi_{j(k),P_{\w^*}(T)}^F}$'s, the complete positivity of $M_{\phi_{j,P_{\w^*}(T)}^F}$ carries over to that of $P^F(T)$ by \cite[Lemma 2.10 p.~15]{ArK23}.

Now, we consider a weak* accumulation point $P^{(1)} \co \CB(\VN(G)) \to \CB(S^\infty_G,\cal{B}(\L^2(G)))$ of the net $(P^F)_F$ and by the same reasoning as before, the map $P^{(1)}(T) \co S^\infty_G \to \cal{B}(\L^2(G))$ is again a completely bounded Schur multiplier and preserves the complete positivity. The map $P^{(1)}$ is contractive.

\paragraph{Step 2} For any weak* continuous completely bounded map $T \co \VN(G) \to \VN(G)$, we claim that the map $P^{(1)}(T) \co S^\infty_G \to \cal{B}(\L^2(G))$ is in fact a Herz-Schur multiplier. It is easy to check that the weak* convergence of a subnet of $(M_{\phi_{j,P_{\w^*}(T)}^F})_j$ to $M_{\varphi^F}\ov{\mathrm{def}}{=} P^{F}(T)$ implies that $\varphi^F$ is a cluster point of $(\phi_{j,P_{\w^*}(T)}^F)_j$ for the weak* topology of the dual space $\L^\infty(G \times G)$. In the same manner, the symbol $\varphi$ of the Schur multiplier $P^{(1)}(T)$ is a cluster point of $(\varphi^F)_F$ for the weak* topology of $\L^\infty(G \times G)$. Thus according to Lemma \ref{lem-SAIN-Herz-Schur}, the function $\varphi$ is a Herz-Schur symbol. This is the step where we use the assumption of inner amenability on the group $G$.

\paragraph{Step 3} By \cite{BoF84} and \cite[Theorem 5.3 p.~181]{Spr04}, we have an isometric map $I \co \mathfrak{M}^{\infty,\cb,\HS}_G \to \CB_{\w^*}(\VN(G))$, $M_{\varphi}^\HS \mapsto M_\varphi$ with range $\mathfrak{M}^{\infty,\cb}(G)$, preserving the complete positivity. Indeed, on the one hand, recall that by \cite[Proposition 6.11 p.~90]{ArK23}, a Fourier multiplier $M_\varphi \co \VN(G) \to \VN(G)$ is completely positive if and only if $\varphi$ is equal almost everywhere to a continuous positive definite function, i.e.~the kernel $(s,t) \mapsto \varphi(st^{-1})$ is of positive type \cite[p.~351]{BHV08} and that a bounded Fourier multiplier necessarily has a symbol that is equal almost everywhere to a continuous function (see the discussion \cite[p.~85]{ArK23}). On the other hand, the Herz-Schur multiplier $M_\varphi^\HS$ is completely positive if and only if $\varphi^\HS \co (s,t) \mapsto \varphi(st^{-1})$ is equal almost everywhere to a bounded and measurable function of positive type 
(see \cite[Proposition 3.3 p.~781 and Remark 4.8 p.~785]{Arh24}).

%if for any $\xi \in \L^1(G)$ we have the inequality
%$$
%\iint_{G \times G} \varphi(st^{-1}) \ovl{\xi(s)} \xi(t) \d s \d t
%\ov{\eqref{Cartier-int-pos}}{\geq} 0,
%$$
%which is equivalent to say that the function $\varphi$ is positive definite.

We introduce the linear map $P \ov{\mathrm{def}}{=} I \circ P^{(1)} \co \CB_{\w^*}(\VN(G)) \to \CB_{\w^*}(\VN(G))$ with values in the space $\mathfrak{M}^{\infty,\cb}(G)$ of completely bounded Fourier multipliers. By composition, this map is contractive and preserves the complete positivity. 

Finally, if $T = M_\varphi \co \VN(G) \to \VN(G)$ is a completely bounded Fourier multiplier, then for any $j$ the symbol $\phi_{j,T}^F$ of the Schur multiplier $P_j^F(T) \co S^\infty_G \to \cal{B}(\L^2(G))$ is given by
\begin{align*}
\MoveEqLeft
\phi_{j,T}^F(s,t) 
\ov{\eqref{Def-ds-inner}}{=} \varphi_{1,y_j^F}(s,t)
\ov{\eqref{Def-symbol-varphi-1}}{=} \tau_G\big( y_j^F \lambda_{s^{-1}} T(\lambda_{st^{-1}})\lambda_t\big) 
= \varphi(st^{-1}) \tau_G\big( y_j^F \lambda_{s^{-1}}  \lambda_{st^{-1}}\lambda_t\big) \\
& = \varphi(st^{-1}) \tau_G\big(y_j^F\big) 
= \varphi(st^{-1}).
\end{align*}
Thus $P^{(1)}(T) = P^F(T) = M_{\varphi}^\HS$ and $P(T) = I \circ P^{(1)}(T)= I(M_{\varphi}^\HS) = M_\varphi = T$.
\end{proof}

Now, we state the following general theorem of complementation of the space of completely bounded Fourier multipliers. The proof uses in a crucial way that completely bounded Herz-Schur multipliers acting on $S^p_G$  are in one-to-one (linear, norm and order) correspondence with completely bounded $\L^p$-Fourier multipliers, thanks to the amenability of the group $G$.

\begin{thm}
\label{thm-general-complementation}
Let $G$ be a second-countable unimodular amenable locally compact group. Suppose that $1 < p < \infty$.
%Assume that there are nets $(x_j)$ and $(y_j)$ of positive elements in the space $\L^1(\VN(G)) \cap \VN(G)$  such that
%\begin{itemize}
%\item $\norm{x_j}_p\leq C$, $\norm{y_j}_{p^*} \leq C$ for all $j$,
%\item $\tau_G(x_j y_j) = 1$ for all $j$,
%\end{itemize}
Let $(f_j)$ and $(g_j)$ be nets of positive functions with compact support belonging to the space $\C_e(G)$ such that if $x_j \ov{\mathrm{def}}{=} \lambda(f_j)$, $y_j \ov{\mathrm{def}}{=} \lambda(g_j)$ we have
\begin{itemize}
\item for some positive constant $C$ we have $\norm{x_j}_{\L^p(\VN(G))}, \norm{y_j}_{\L^{p^*}(\VN(G))} \leq C$ for all $j$,

\item $\tau_G(x_j y_j) = 1$ for all $j$,

\item $\supp f_j \to \{e\}$ or $\supp g_j \to \{e\}$. 
\end{itemize}

Then there exists a bounded projection $P_G^p \co \CB(\L^p(\VN(G))) \to \CB(\L^p(\VN(G)))$  
onto the space $\mathfrak{M}^{p,\cb}(G)$ of completely bounded Fourier multipliers with the properties
\begin{enumerate}
\item $\norm{P_G^p} \leq C^2$,
%\item $P(M_\varphi) = M_\varphi$ for any completely bounded multiplier $M_\varphi \co\L^p(\VN(G))  \to \L^p(\VN(G))$,
\item $P_G^p(T)$ is completely positive whenever $T$ is completely positive.
\end{enumerate}
%Finally, let $p = \infty$ or $p = 1$. Then there exists contractive projections
%\[ 
%P \co \CB_{\w^*}(\VN(G)) \to \CB_{\w^*}(\VN(G)) 
%\quad \text{and} \quad 
%P \co \CB(\L^1(\VN(G))) \to \CB(\L^1(\VN(G)))
%\]
%onto the subspaces $\mathfrak{M}^{\infty,\cb}(G)$ and $\mathfrak{M}^{1,\cb}(G)$ satisfying the previous second property.
\end{thm}

\begin{proof}
The proof consists of several steps.
\paragraph{Step 1}
Recall that the function $\phi_{j,T} \ov{\mathrm{def}}{=} \varphi_{x_j,y_j,T}$ belonging to the space $\L^\infty(G \times G)$ is defined in \eqref{def-symbol-phi-alpha}. Consider the linear map $P_j \co  \CB(\L^p(\VN(G))) \to  \CB(S^p_G)$, $T \mapsto M_{\phi_{j,T}}$. Using Lemma \ref{Lemma-estimation-cb}, we obtain the estimate 
\begin{align*}
\MoveEqLeft
\norm{P_j(T)}_{\cb,S^p_G \to S^p_G}  
\ov{\eqref{div-987}}{\leq} \norm{T}_{\cb, \L^p(\VN(G)) \to \L^p(\VN(G))} \norm{x_j}_{\L^p(\VN(G))} \norm{y_j}_{\L^{p^*}(\VN(G))} \\
&\leq C^2 \norm{T}_{\cb,\L^p(\VN(G)) \to \L^p(\VN(G))},
\end{align*}
according to the assumptions of the theorem. Then $(P_j)$ is a bounded net in the space $\cal{B}(\CB(\L^p(\VN(G))),\CB(S^p_G))$, which is a dual space with predual
\begin{equation*}
\label{equ-predual-bracket-bis}
\CB(\L^p(\VN(G))) \hat \ot \big(S^p_G \widehat{\ot} S^{p^*}_G \big),
\end{equation*}
where $\hat{\ot}$ denotes the Banach space projective tensor product and where $\widehat{\ot}$ denotes the operator space projective tensor product. The duality bracket is given by
\begin{equation}
\label{Derniere-eq}
\big\langle P , T \ot (x \ot y) \big\rangle
=\big\langle P(T) x, y \big\rangle_{S^p_G,S^{p^*}_G}.
\end{equation}
By the Banach-Alaoglu theorem, the net $(P_j)$ admits a convergent subnet $(P_{j(k)})$, which converges to some element $P^{(1)}$, i.e.~the net $(P_{j(k)})$ converges to $P^{(1)}$ for the weak* topology. With \eqref{Derniere-eq}, we see that for any $T \in \CB(\L^p(\VN(G)))$ this implies that $(P_{j(k)}(T))$ converges for the weak operator topology to $P^{(1)}(T)$.

Observe that the weak* topology on the space $\CB(S^p_G)$ coincides on bounded subsets with the weak operator topology by \cite[Lemma 7.2 p.~85]{Pau02}. We conclude by Lemma \ref{lem-Schur-weak-star-closed} that $P^{(1)}(T)$ is a Schur multiplier. 
%Since the subspace $\mathfrak{M}^{p,\cb}_G$ is weak* closed in the space $\CB(S^p_G)$ according to Lemma \ref{lem-Schur-weak-star-closed}, it is easy to check that $P^{(1)}$ takes its image in the space $\mathfrak{M}^{p,\cb}_G$.
%If $T$ is completely positive, then for $[K_{\phi_{ij}}],\: [K_{\psi_{ij}}]$ positive,
%\[ 
%\sum_{ij} \big\langle P_\alpha(T) K_{\phi_{ij}}, K_{\psi_{ij}} \big\rangle 
%= \sum_{ij} \big\langle (\Id \ot T)(W(K_{\phi_{ij}} \ot x_\alpha)W^{-1}), W (K_{\psi_{ij}} \ot y_\alpha) W^{-1} \big\rangle 
%\geq 0 .
%\]
%Indeed, $\Id \ot \Id \ot T$ is then positivity preserving and $[W(K_{\phi_{ij}} \ot x_\alpha)D^{-1}]$ and $[W (K_{\psi_{ij}} \ot y_\alpha) W^{-1}]$ are positive.
%We infer by \cite[Lemma 2.6]{ArK23} that $[P_\alpha(T) K_{\phi_{ij}}]$ is positive, 
Lemma \ref{Lemma-estimation-cb} says that the Schur multiplier $P_j(T) \co S^p_G \to S^p_G$ is completely positive if $T$ is completely positive. By \cite[Lemma 2.10 p.~15]{ArK23}, the weak operator limit $P^{(1)}(T)$ is also completely positive if $T$ is so. % according to Proposition \ref{prop-referees-proof-step-1-weak-star-convergence}).then according to Proposition \ref{prop-referees-proof-step-1-weak-star-convergence}

Suppose that $T = M_\phi$ is a Fourier multiplier. By Proposition \ref{th-convergence}, the sequence $(\phi_{j,T})_j$ of elements in $\L^\infty(G \times G)$ converges for the weak* topology to the function $\phi^\HS \co (s,t) \mapsto \phi(st^{-1})$. By Lemma \ref{Lemma-symbol-weak}, we deduce that the net $(P_{j}(T))$ converges to the Schur multiplier $M_{\phi^\HS}$ in the weak operator topology. So the symbol of the Schur multiplier $P^{(1)}(T)$ is $\phi^\HS$.

\paragraph{Step 2}
We will use the contractive projection onto the space $\mathfrak{M}^{p,\cb,\HS}_{G}$ of completely bounded Herz-Schur multipliers of Proposition \ref{prop-referee-step-2}. Seeing this projection as a map $Q \co \mathfrak{M}^{p,\cb}_{G} \to \mathfrak{M}^{p,\cb,\HS}_{G}$, we let $P^{(2)}\ov{\mathrm{def}}{=} Q \circ P^{(1)} \co \CB(\L^p(\VN(G))) \to \mathfrak{M}^{p,\cb,\HS}_{G}$. Then by this proposition together with Step 1, the map $P^{(2)}$ satisfies $\bnorm{P^{(2)}} \leq C^2$ and preserves the complete positivity and $P^{(2)}(M_{\varphi})$ has again symbol $(s,t) \mapsto \varphi(st^{-1})$.

\paragraph{Step 3} With the result of the paper \cite[Corollary 5.3 p.~7008]{CaS15} (see also \cite{NeR11}), we can consider the contractive\footnote{\thefootnote. The result \cite[Corollary 5.3 p.~7008]{CaS15} is stated for some class of symbols. However, the proof shows that the result is true for any completely bounded Herz-Schur multiplier on $S^p_G$.} map $I \co \mathfrak{M}^{p,\cb,\HS}_G \to  \CB(\L^p(\VN(G)))$, $M_\varphi^\HS \mapsto M_\varphi$ onto the space $\mathfrak{M}^{p,\cb}(G)$ of completely bounded Fourier multipliers. As in the proof of Theorem \ref{thm-SAIN-tilde-kappa}, if the Herz-Schur multiplier $M_\varphi^\HS \co S^p_G \to S^p_G$ is completely positive, it is easy to see that the Fourier multiplier $I(M_\varphi^\HS)$ is completely positive.% as a  limit in the weak operator topology of completely positive maps.
%\begin{align*}
%\MoveEqLeft
%M_\phi(x)
%= \lim_j (\Phi^\infty_j \Psi^\infty_j M_\phi)(x) 
%= \lim_j (\Phi^p_j \Psi^p_j M_\phi)(x) 
 %= \lim_j \big(\Phi^p_j (M_\phi^{\HS}|S^p(L^2(F_j))) \Psi^p_j\big)(x)
%\end{align*}
%
%\textbf{Cas $p=\infty$: }Indeed, a  symbol $\varphi$ giving rise to a completely positive Herz-Schur multiplier is an integrally positive definite kernel by \cite[Proposition 4.9]{Arh24}. By \cite[Proposition 3.4]{Arh24}, we deduce that $\varphi$ is equal almost everywhere to a bounded and measurable positive definite kernel. Note that the symbol of the Fourier multiplier $I(M_\varphi^\HS)$ is continuous by \cite[page 22]{Spr04} \textbf{(on aura peut-etre pas si $p<\infty$)}. By \cite[Proposition 6.11]{ArK23} and \cite[Definition C.4.1]{BHV08}, we conclude that $I(M_\varphi^\HS)$ is completely positive. 

We finally put $P_G^p \ov{\mathrm{def}}{=}  I \circ P^{(2)} \co \CB(\L^p(\VN(G))) \to \CB(\L^p(\VN(G)))$ and need to check the claimed properties. First, we have the estimate $\norm{P_G^p} \leq \norm{I} \norm{P^{(2)}} \leq C^2$. Second, if the operator $T$ is completely positive, then $P^{(2)}(T)$ is also completely positive, and consequently, $P_G^p(T) = I \circ P^{(2)}(T)$ inherits this property.
Third, if $T = M_{\varphi}$ is itself a multiplier, then $P^{(2)}(T)$ is the Herz-Schur multiplier with symbol $(s,t) \mapsto \varphi(st^{-1})$. So the map $P_G^p(T)$ is the Fourier multiplier with symbol $\varphi$ again.
\end{proof}

Note that the proof of the previous result can be adapted to the cases $p=1$ and $p=\infty$ using Section \ref{Section-p=1-p-infty}. It is worth mentioning that the construction of the projection in the case $p=\infty$ differs from that in Theorem \ref{thm-SAIN-tilde-kappa}.

%in the situation of Corollary \ref{cor-1-referees-proof-step-1-weak-star-convergence} or Section \ref{Section-p=1-p-infty}
Combining this result with Example \ref{Essai}, we obtain the following result.

\begin{cor}%[Recall of Theorem \ref{thm-cor-the-full-referees-complementation}]
\label{cor-the-full-referees-complementation}
Let $G$ be a second-countable unimodular amenable locally compact group. Let $1 < p < \infty$ such that $\frac{p}{p^*}$ is rational. Then there exists a contractive projection
\[ 
P^p_G \co \CB(\L^p(\VN(G))) \to \CB(\L^p(\VN(G))) 
\]
on the space $\mathfrak{M}^{p,\cb}(G)$ of completely bounded Fourier multipliers such that the map $P^p_G(T)$ is completely positive whenever $T$ is completely positive.
\end{cor}

%\begin{proof}
%It suffices to pick the net $(f_\alpha)_\alpha$ in $\C_c(G)$ and the associated $x_\alpha$ and $y_\alpha = \lambda(f_\alpha)^{\frac{2}{p^*}}$ ($x_\alpha = 1$ in case $p = \infty$, $y_\alpha = 1$ in case $p = 1$) from Corollary \ref{cor-1-referees-proof-step-1-weak-star-convergence} . Choosing $f_\alpha = g_\alpha^* \ast g_\alpha$ whith $g_\alpha \in \C_c(G)$ ensures that $x_\alpha$ and $y_\alpha$ belong to $\L^1(\VN(G)) \cap \VN(G)$. Now, it suffices to use this corollary together with Theorem \ref{thm-general-complementation}.
%\end{proof}

\begin{remark}\normalfont
Going through the different steps of the proof, one sees that the projection in the Corollary \ref{cor-the-full-referees-complementation} has the property that $P^p_G(T)^* = P^{p^*}_G(T^*)$ for any completely bounded map $T \co \L^p(\VN(G)) \to \L^p(\VN(G))$ and any $1 \leq p < \infty$. Here, we emphasize that $P = P^p_G$ depends on $1 \leq p \leq \infty$. 
However, it is not clear that the maps $P^p_G(T)$ and $P^q_G(T)$ coincide for $1 \leq  p,q \leq \infty$ and an operator $T$ acting on both $\L^p(\VN(G))$ and $\L^q(\VN(G))$. %See also Remark \ref{rem-referee-non-compatibilite}.
\end{remark}

Using Theorem \ref{thm-general-complementation}, we obtain the following important property of compatibility of projections.

\begin{cor}
\label{cor-the-compatible-complementation}
Any second-countable unimodular finite-dimensional amenable locally compact group $G$ has property $(\kappa)$. More precisely, for all $1 \leq p \leq \infty$ there exists a bounded projection
\[ 
P^p_G \co \CB(\L^p(\VN(G))) \to \CB(\L^p(\VN(G))) 
\]
onto the subspace $\mathfrak{M}^{p,\cb}(G)$ (resp $P^\infty_G \co \CB_{\w^*}(\VN(G)) \to \CB_{\w^*}(\VN(G))$ on $\mathfrak{M}^{\infty,\cb}(G)$) with the properties
\begin{enumerate}
\item $\norm{P_G^p} \leq C$, where the constant $C$ depends on $G$ but not on $p$. 

%\item $P^p(M_\varphi) = M_\varphi$ for any completely bounded multiplier $M_\varphi \co \L^p(\VN(G)) \to \L^p(\VN(G))$,

\item $P^p_G(T)$ is completely positive whenever the map $T$ is completely positive.

\item If $T$ belongs to $\CB(\L^p(\VN(G)))$ and to $\CB(\L^q(\VN(G)))$ for two values $1 \leq p, q \leq \infty$, then the Fourier multipliers $P^p_G(T)$ and $P^q_G(T)$ are compatible mappings coinciding on $\L^p(\VN(G)) \cap \L^q(\VN(G))$, i.e.~have the same symbol.
\end{enumerate}
\end{cor}

\begin{proof}
Consider first the case $1 < p \leq \infty$.
It suffices to pick the sequence $(f_j)$ in $\C_c(G)$ and the associated $x_j \ov{\eqref{def-fj}}{=} a_j \lambda(f_j)$, $y_j \ov{\eqref{def-fj}}{=} b_j \lambda(f_j)$ from Corollary \ref{cor-2-referees-proof-step-1-weak-star-convergence-bis}.
Then apply this corollary together with Theorem \ref{thm-general-complementation} to deduce the two first points. Let us show the last point 3. Due to the specific choice of $x_j = a_j \lambda(f_j)$ and $y_j = b_j \lambda(f_j)$ with $a_j \cdot b_j$ being independent of $p$ (see Corollary \ref{cor-2-referees-proof-step-1-weak-star-convergence-bis}), we deduce that $P_j^p(T)K_\phi = P_j^q(T)K_\phi$ for $K_\phi \in S^p \cap S^q$, where $P_j^p$ and $P_j^q$ denote the mappings from Step 1 of the proof of Theorem \ref{thm-general-complementation}. Indeed, the symbol of the Schur multiplier $P_j^p(T)$ is 
$$
\phi_{j,T}(s,t) 
\ov{\eqref{def-symbol-phi-alpha}}{=} a_j b_j\tau_G \big(\lambda_t \lambda(f_j) \lambda_{s^{-1}} T(\lambda_s \lambda(f_j) \lambda_{t^{-1}}) \big), \quad s,t \in G.
$$
Consider the product 
$$
X 
\ov{\mathrm{def}}{=} \prod_{1  \leq p < \infty} \cal{B}(\CB(\L^p(\VN(G))),\CB(S^p_G)) \times \cal{B}(\CB(\VN(G)),\CB(S^\infty_G,\cal{B}(\L^2(G)))),
$$ 
which is a topological space when equipped with the product topology of the weak* topology.
Then $P_j = (P_j^p)_{1  \leq p \leq \infty}$ lies in a compact subspace of $X$ for all $j$, since $\norm{P_j^p} \leq C$ with a constant independent of $p$ and Tychonoff's theorem for the product of compact spaces applies here.
Thus, the net $(P_j)$ admits an accumulation point in $X$, that is to say that for the same subnet $j(k)$ for all $p$, we have $P^{(1),p} = \lim_{j(k)} P_{j(k)}^{p}$.
We infer that $P^{(1),p}(T) K_\phi = P^{(1),q}(T) K_\phi$ for any $K_\phi \in S^p \cap S^q$, where $P^{(1),p}$ and $P^{(1),q}$ denote the mappings from Step 1 of the proof of Theorem \ref{thm-general-complementation}.
Since the mapping $Q$ from Proposition \ref{prop-referee-step-2} is compatible for different values $p$ and $q$, also $P^{(2),p}$ and $P^{(2),q}$ from Step 2 of the proof of Theorem \ref{thm-general-complementation} are compatible.
Finally, since the mappings $I=I_p$ from Step 3 of the proof of Theorem \ref{thm-general-complementation} are compatible for different values $p$ and $q$, the mappings $P^p_G(T)$ and $P^q_G(T)$ are compatible for any $T$.
\end{proof}

%\textbf{In the case $p = \infty$ or $p = 1$, we put $P^{(1)} \ov{\mathrm{def}}{=}  P_{j_0}$ for some fixed index $j_0$.} 

\begin{example} \normalfont
\label{ex-totally-disconnected-contractive-complementation}
Let $G$ be a second-countable unimodular finite-dimensional amenable locally compact group. Consequently, Corollary \ref{cor-the-compatible-complementation} applies. Suppose that $G$ is in addition totally disconnected. %, hence zero-dimensional by Remark \ref{0-dim-space}. 
Then the projections $P_G^p$ in that result are in fact contractions.  Indeed, an inspection of the proof of Corollary \ref{cor-the-compatible-complementation} (see also Step 1 of the proof of Theorem \ref{thm-general-complementation}) shows that $\norm{P^p_G} \leq \sup_j \norm{x_j}_p \norm{y_j}_{p^*}$. According to Corollary \ref{Cor-38}, the right-hand side is less than $1$.

Therefore, Corollary \ref{cor-the-compatible-complementation} gives a variant of \cite[Theorem 6.38 p.~121]{ArK23}. Note that Corollary \ref{cor-the-compatible-complementation} needs that the group $G$ is amenable in the cases $p=\infty$ and $p=1$ in contrast to \cite[Theorem 6.38 p.~121]{ArK23}.
% We conjecture that inner amenability is equivalent to $(\kappa_\infty)$. Hence the amenability seems necessary by the last observation of Example \ref{Contre-example}.
\end{example}

%\begin{example} \normalfont
%\label{example-Lie}
%Recall that a semisimple connected Lie group is unimodular by \cite[p.~47]{Fol16} and amenable if and only if it is compact by \cite[Proposition 14.10 p.~134]{Pie84}. 
%%Reilter 8.7.9
%A connected nilpotent Lie group is unimodular and amenable by \cite[p.~1487]{Pal01}. %\cite[Corollary 13.5]{Pie84}. 
%So Corollary \ref{cor-the-compatible-complementation} applies to second-countable groups belonging to these classes.
%\end{example}

%\begin{remark} \normalfont
%\label{rem-drawback}
%The drawback of the method of the proof of Corollary \ref{cor-the-compatible-complementation} is the problem of the estimation of the best constant $C$ or the constant $\kappa(G)$. We have no precise estimate apart from Example \ref{ex-totally-disconnected-contractive-complementation}. In \cite{ArK23}, we obtained quantitative results for some locally compact groups. \textbf{FAUX refaire avec Rem \ref{Rem-doubling}} et la constante doubling.
%\end{remark}

\begin{remark} \normalfont
\label{rem-K2}
We have resisted to the temptation to write a matricial version of Theorem \ref{thm-SAIN-tilde-kappa} in the spirit of \cite[Theorem 4.2 p.~62]{ArK23}. It is likely that the same method works.
\end{remark}

\begin{remark} \normalfont
\label{non-unimodular}
We make no attempt with non-unimodular locally compact groups. It is likely that the same strategy works in the case $p=\infty$. We leave this case as an exercise for the reader. %At the level of quantum groups, the first author has an unpublished result which is a generalization of a result of \cite{ArK23}.
\end{remark}

%%%%%%%%%%%%%%%%%%%%%%%%%%%%%%%%%%%%%%%%%%%%%%%%%%%%%%%%%%%%%%%%%%%%%%%%%%%%%%%%%%%%%%%%%%%%%
\section{Final remarks}
\label{Sec-applications}
%%%%%%%%%%%%%%%%%%%%%%%%%%%%%%%%%%%%%%%%%%%%%%%%%%%%%%%%%%%%%%%%%%%%%%%%%%%%%%%%%%%%
\subsection{A characterization of the amenability of unimodular locally compact groups}
\label{Sec-charac-amen}

We present a new characterization of amenability, in the same spirit as that of Theorem \ref{Th-Lau-Paterson}. Recall that property $(\kappa_\infty)$ is defined in Definition \ref{Defi-tilde-kappa}.

\begin{thm}
\label{thm-links-K-injective}
Let $G$ be a second-countable unimodular locally compact group. Then the following are equivalent.
\begin{enumerate}
	\item The von Neumann algebra $\VN(G)$ is injective and $G$ has property $(\kappa_\infty)$.
	\item The group $G$ is amenable.
\end{enumerate}
Moreover, the implication 1. $\Rightarrow$ 2. is true without the assumption <<second-countable unimodular>>.
\end{thm}

\begin{proof}
1. $\Rightarrow$ 2. Since the locally compact group $G$ has property $(\kappa_\infty)$, Proposition \ref{conj-1-1-correspondance} gives the equality
$$
\frak{M}^{\infty,\dec}(G)
=\B(G).
$$
Since $\VN(G)$ is injective, by \cite[Theorem 1.6 p.~184]{Haa85} each completely bounded operator $T \co \VN(G) \to \VN(G)$ is decomposable with $\norm{T}_{\cb,\VN(G) \to \VN(G)} \ov{\eqref{dec=cb}}{=} \norm{T}_{\dec,\VN(G) \to \VN(G)}$. In particular, we have $\frak{M}^{\infty,\cb}(G)=\frak{M}^{\infty,\dec}(G)$ isometrically. We deduce that
$$
\frak{M}^{\infty,\cb}(G)
=\B(G).
$$
Using a result stated in \cite[p.~54]{Pis01} (see also \cite[p.~190]{Spr04}), which says that this equality is equivalent to the amenability of $G$, we conclude that the group $G$ is amenable. 

2. $\Rightarrow$ 1. If the locally compact group $G$ is amenable, then by Theorem \ref{Th-Lau-Paterson} the von Neumann algebra $\VN(G)$ is injective. Note that the group $G$ is inner amenable by Example \ref{Ex-inner-2}. Consequently, for the second property, it suffices to use Theorem \ref{thm-SAIN-tilde-kappa}.
\end{proof}

%Now, we prove that there exists some locally compact group $G$ such the space $\mathfrak{M}^{\infty,\dec}(G)$ of completely bounded Fourier multipliers on the von Neumann algebra $\VN(G)$ is not complemented in the space $\CB_{\w^*}(\VN(G))$. 

As a result, we can provide explicit examples of locally compact groups that do not satisfy property $(\kappa_\infty)$.

\begin{cor}
\label{Cor-gropes-without-K}
Any non-amenable second-countable connected locally compact group $G$ does not have property $(\kappa_\infty)$.
\end{cor}

\begin{proof}
By \cite[Corollary 7 p.~75]{Con76}, the von Neumann algebra $\VN(G)$ of a second-countable connected locally compact group is injective. Suppose that the group $G$ has property $(\kappa_\infty)$. By the implication 1. $\Rightarrow$ 2. of Theorem \ref{thm-links-K-injective}, we obtain that $G$ is amenable, i.e.~a contradiction.
\end{proof}

\begin{example} \normalfont
\label{example-SL}
This result applies for example to the connected locally compact group $\SL_2(\R)$, which is non-amenable by \cite[Example G.2.4 (i) p.~426]{BHV08} and unimodular by \cite[p.~4]{Lan75}, contradicting\footnote{\thefootnote. We would like to thank Adam Skalski for his confirmation of this problem in this small remark by email \textit{on his own initiative.} The results of the nice paper \cite{DFSW16} remain correct.} the observation \cite[Remark 7.6 p.~24]{DFSW16}, stated for unimodular locally compact quantum groups.
\end{example}

The following result generalizes Theorem \ref{Thm-conj-discrete-case} when second countability is not assumed.

\begin{cor}
\label{cor-inner-66}
Let $G$ be a second-countable unimodular inner amenable locally compact group. Then the von Neumann algebra $\VN(G)$ is injective if and only if we have $\frak{M}^{\infty,\dec}(G)= \frak{M}^{\infty,\cb}(G)$.
\end{cor}

\begin{proof}
The <<only if>> part %de gauch à droite
is true by a result of Haagerup \cite[Corollary 2.8 p.~201]{Haa85}. Now, assume that $\frak{M}^{\infty,\dec}(G)= \frak{M}^{\infty,\cb}(G)$. By Proposition \ref{conj-1-1-correspondance} and Theorem \ref{thm-SAIN-tilde-kappa}, we have a bijection from the space $\frak{M}^{\infty,\cb}(G)$ onto the Fourier-Stieltjes algebra $\B(G)$. So the group $G$ is amenable by a result stated in \cite[p.~54]{Pis01}. We conclude that the von Neumann algebra $\VN(G)$ is injective by Theorem \ref{Th-Lau-Paterson}.
\end{proof}

%%%%%%%%%%%%%%%%%%%%%%%%%%%%%%%%%%%%%%%%%%%%%%%%%%%%%%%%%%%%%%%%%%%%%%%%%%%%%
%%%%%%%%%%%%%%%%%%%%%%%%%%%%%%%%%%%%%%%%%%%%%%%%%%%%%%%%%%%%%%%%%%%%%%%%%%%%%%%%
\subsection{Bounded Herz–Schur multipliers and their relation to amenability}
\label{Sec-Herz}
%%%%%%%%%%%%%%%%%%%%%%%%%%%%%%%%%%%%%%%%%%%%%%%%%%%%%%%%%%%%%%%%%%%%%%%%%%%%%%%%

Consider a locally compact group $G$. To know whether the amenability of $G$ is characterized by the complementation of the space $\mathfrak{M}^{\infty,\HS}_{G}$ of (completely) bounded Herz-Schur multipliers in the space $\mathfrak{M}^{\infty}_G$ of (completely) bounded Schur multipliers over $\cal{B}(\L^2(G))$ is a well-known open question, explicitly stated in \cite[p.~184]{Spr04}. Now, we present the first progress on this classical question. % \cite[p.~]{SpT02} \cite[p.~]{ArK23}

\begin{prop}
Let $G$ be a second-countable unimodular locally compact group such that the von Neumann algebra $\VN(G)$ is injective. Suppose that there exists a bounded projection $Q \co \mathfrak{M}^{\infty}_G \to \mathfrak{M}^{\infty}_{G}$ onto the space $\mathfrak{M}^{\infty,\HS}_{G}$ of bounded Herz-Schur multipliers over the space $\cal{B}(\L^2(G))$, preserving the complete positivity. Then the group $G$ is amenable.
Conversely, if $G$ is amenable, then such a $Q$ exists according to Proposition \ref{prop-referee-step-2}.
\end{prop}

\begin{proof}
Since the von Neumann algebra $\VN(G)$ is injective, it suffices by Theorem \ref{thm-links-K-injective} to show that the group $G$ has property $(\kappa_\infty)$. Now, we follow the proof of Theorem \ref{thm-general-complementation}. For the first step, we use Lemma \ref{lemma-symbol-step-1-p=infty} and we define the linear map $P^{(1)}_G \co  \CB_{\w^*}(\VN(G)) \to  \CB(\cal{B}(\L^2(G)))$, $T \mapsto M_{\varphi_{1,y,T}}$ for some positive element $y$ in the space $\L^1(\VN(G)) \cap \VN(G)$ such that $\tau_G(y) = 1$. The second step of the proof is addressed using the existence of the projection $Q$. If $p=\infty$, the third step is done without amenability. So we obtain property $(\kappa_\infty)$.
\end{proof}

\begin{example} \normalfont
By Example \ref{Example-almost}, the von Neumann algebra $\VN(G)$ of a second-countable almost connected unimodular locally compact group is injective. In particular, the previous result applies to any connected unimodular Lie group. Such a group is amenable if and only if there exists a bounded projection $Q \co \mathfrak{M}^{\infty}_G \to \mathfrak{M}^{\infty}_G$ onto the subspace $\mathfrak{M}^{\infty,\HS}_{G}$  that preserves complete positivity.
\end{example}

\paragraph{Declaration of interest} None.

\paragraph{Competing interests} The authors declare that they have no competing interests.

\paragraph{Acknowledgment}
The authors are deeply grateful to the anonymous referee of \cite{ArK23}. His report suggested a promising research direction which, when combined with unpublished results, new ideas, and further work, led to the development of this paper. We extend our thanks to Claire Anantharaman-Delaroche and Martin Walter for brief but valuable discussions, as well as to Matthew Daws, Adam Skalski, and Ami Viselter for earlier discussions and their continued encouragement. We are also grateful to Marco Peloso for his insights regarding \eqref{Equivalence-measure-ball} in the non-unimodular case. %Special thanks go to \'Eric Ricard for suggesting the operator featured in Theorem \ref{thm-strongly-non-dec}.

This work was supported by the ANR-18-CE40-0021 grant (HASCON project). The second author also acknowledges partial support from the ANR-17-CE40-0021 grant (FRONT project), funded by the French National Research Agency.

{\footnotesize

\vspace{0.1cm}

\noindent C\'edric Arhancet\\
6 rue Didier Daurat \\
81000 ALBI \\
FRANCE\\
URL: \href{http://sites.google.com/site/cedricarhancet}{https://sites.google.com/site/cedricarhancet}\\
cedric.arhancet@protonmail.com\\
ORCID: 0000-0002-5179-6972 \\

\noindent Christoph Kriegler\\
Universit\'e Clermont Auvergne\\
CNRS\\
LMBP\\
F-63000 CLERMONT-FERRAND\\
FRANCE \\
URL: \href{https://lmbp.uca.fr/~kriegler/indexenglish.html}{https://lmbp.uca.fr/{\raise.17ex\hbox{$\scriptstyle\sim$}}\hspace{-0.1cm} kriegler/indexenglish.html}\\
christoph.kriegler@uca.fr \\
ORCID: 0000-0001-8120-6251
}
\end{document}